\documentclass[10pt]{article}
\usepackage{amsfonts}
\usepackage{amssymb}
\usepackage{amsmath}
\usepackage{amsfonts}
\usepackage[a4paper, total={6in, 9.1in}]{geometry}
\usepackage{hyperref}
\usepackage{tikz}
\usepackage{eurosym}
\usepackage{amsfonts}
\usepackage{amssymb}
\usepackage{amsmath}
\usepackage{amsfonts}
\usepackage{float}
\usepackage{tikz}
\usepackage{paralist}
\usepackage{epstopdf}
\usepackage{mathrsfs}
\usepackage{bbm}

\usetikzlibrary{arrows}
\hypersetup{
colorlinks=true,
linkcolor=blue,
filecolor=magenta,      
urlcolor=cyan,
}
\newtheorem{theorem}{Theorem}

\newtheorem{corollary}[theorem]{Corollary}

\newtheorem{lemma}[theorem]{Lemma}

\newtheorem{proposition}[theorem]{Proposition}
\newtheorem{remark}[theorem]{Remark}

\newenvironment{proof}[1][Proof]{\noindent\textbf{#1.} }{\ \rule{0.5em}{0.5em}}

\begin{document}

\title{ Boundary controllability of two coupled wave equations with
space-time first-order coupling in $1-D$}
\author{$^{1}$Farid Ammar Khodja and $^{2}$Yacine Mokhtari}
\maketitle

\begin{abstract}
This paper is devoted to study exact controllability of two one-dimensional
coupled wave equations with first-order coupling terms with coefficients
depending on space and time. We give a necessary and sufficient condition
for both exact controllability in high frequency in the general case and the
unique continuation in the cascade case.
\end{abstract}

\tableofcontents

\footnote{%
Laboratoire de Math\'{e}matiques UMR 6623, Universit\'{e} de Franche-Comt%
\'{e}, 16, route de Gray, 25030 Besan\c{c}on cedex, France. \hyperlink{Email}%
{farid.ammar-khodja@univ-fcomte.fr}} \footnote{%
Laboratoire de Math\'{e}matiques UMR 6623, Universit\'{e} de Franche-Comt%
\'{e}, 16, route de Gray, 25030 Besan\c{c}on cedex, France - University of
Sciences and Technology Houari Boumedienne P.O.Box 32 El-Alia 16111, Bab
Ezzouar, Algiers, Algeria. \hyperlink{Email}{yacine.mokhtari@univ-fcomte.fr}}

\section{Introduction}

We are interested in the boundary controllability of the following system of
two strongly coupled $1-D$ wave equations 
\begin{equation}
	\left\{ 
	\begin{array}{lll}
		y_{tt}=y_{xx}+M(\left( ay\right) _{t}+\left( by\right) _{x}), & \mathrm{in}
		& Q_{T}:=(0,T)\times (0,1), \\ 
		y(t,0)=Bu(t),\text{ }y(t,1)=0,\text{ \ \ \ \ \ \ \ \ \ } & \mathrm{in} & 
		(0,T),\text{ \ \ \ \ \ \ \ \ \ \ \ \ \ \ \ \ \ \ \ \ } \\ 
		y(0,x)=y_{0}(x),\text{ }y_{t}(0,x)=y_{1}(x),\text{ \ \ } & \mathrm{in} & 
		(0,1),\text{ \ \ \ \ \ \ \ \ \ \ \ \ \ \ \ \ \ \ \ \ \ }%
	\end{array}%
	\right.  \label{Wave 1}
\end{equation}%
where $y=(y_{1},y_{2})^{t}$ is a vector function and 
\begin{equation}
	M=(m_{ij})_{1\leq i,j\leq 2}\in \mathcal{L}\left( 
	\mathbb{R}
	^{2}\right) ,\text{ }B=(b_{1},b_{2})^{t}\in 
	\mathbb{R}
	^{2},\text{ }a,b\in C^{1}(\overline{Q_{T}};%
	\mathbb{R}
	),  \label{M,a,b}
\end{equation}%
and $u$ a is scalar control function acting at $x=0$.

This work is motivated by some previous papers. One of them is the result of
Zhang \cite{Zhang}: a single wave equation in any space dimension with lower
order terms is proved to be exactly controllable by a control acting on part
of the boundary under a suitable geometric condition and independently from
the lower terms. The author extended earlier Carleman inequalities proved by
Fursikov-Imanuvilov \cite{Furikov} for the wave equation without these lower
order terms.

The same issue arises for systems of $n$ $(\geq 2)$ coupled wave equations
with boundary or distributed controls. In \cite{Alabau 1}, Alabau-Boussouira
studied the controllability of $2$-coupled wave equations with zero order
coupling operator with constant coefficients. Later, this result has been
generalized by Alabau-Boussouira and L\'{e}autaud in \cite{Alabau 2}, \cite%
{Alabau 3}, for coupling coefficients depending on the space variable under
the geometric control condition introduced in \cite{Bardos-Lebeau-Rauch}. In
these works, one of the coupling coefficients is supposed to be small.

Dehman, Le Rousseau and L\'{e}autaud \cite{Dehman} studied distributed
controllability of $2$-coupled wave equations on a Riemannian manifold
without boundary (periodic boundary conditions in the $1-D$ case) with a
particular zero order coupling operator of cascade type. They proved that
exact controllability holds provided that the Geometric Control Condition is
satisfied. Further, they gave a characterization of the minimal time of
control. An abstract result on the exact controllability of cascade systems
is due to Alabau-Boussouira \cite{Alabau 4} (abstract setting with
application to various coupled second order PDEs). In all these cited works,
it has been assumed that the coupling functions are of constant sign.

A boundary controllability result has been established without the sign or
the smallness conditions by Bennour et al. \cite{Bennour} for $2$-coupled
wave equations in $1-D$ with cascade type coupling through velocity.

Concerning the constant case, Avdonin and De Tereza gave a complete answer
for the exact boundary controllability issue for $2$-coupled wave equations
by zero order operator in $1-D$. The same authors came back in \cite{Avdonin
	2} and generalized their result to $n$ $(\geq 2)$ coupled wave equations in $%
1-D$ but always with constant coupling coefficients under a Kalman rank
condition. The same condition appears for distributed controllability of $n$ 
$(\geq 2)$-coupled multidimensional wave equations with zero order coupling
matrix with constant coefficients. It has been proved by Liard and Lissy in 
\cite{Liard} that it is necessary and sufficient for the exact
controllability in more regular energy space. An extension of this result
can be found in Duprez and Olive \cite{Duprez} for cascade systems with zero
order coupling operator whose coefficients depend on the space variable.
However, boundary controllability has not been treated yet.

Recently, in \cite{Cui}, Cui et al. studied distributed controllability of $n
$ $(\geq 2)$-coupled wave equations with zero and first-order coupling
operator whose coefficients depend on both space and time variables on a
compact Riemannian manifold without boundary (periodic boundary conditions
in the 1-D case). It has been shown that the exact controllability issue can
be reduced to the controllability of a finite dimensional system along the
associated Hamiltonian flow. The authors also gave a unique continuation
results in the \emph{autonomous} case under classical support and sign
assumptions. The same idea appears in \cite{Alabau 5} by Alabau-Boussouira
et al. where distributed controllability of $1-D$ first-order system with
periodic boundary conditions is considered. The authors proved that exact
controllability is reduced to the controllability of parameterized
non-autonomous finite dimensional system. We would also like to mention the recent paper by Coron and Nguyen \cite{Coron} where they proved exact boundary controllability result for a hyperbolic system with space-time zero order term in $1-D.$ We emphasize that in this work the control matrix is invertible (the control acts on all the components with negative speeds). 

In light of all of the cited works, we can see that the main issue that has
to be solved is to figure out the optimal assumptions the coupling
coefficients (or operators) of such systems must satisfy so that exact or
approximate controllability hold with less number of controls.

In this article, and by using the characteristics method and a perturbation argument introduced in the
pioneer work of Russell \cite{Russell}, we give a necessary and
sufficient condition for the boundary exact controllability of System (\ref%
{Wave 1}) in high frequency for a general matrix $M$. We shall also propose
a criterion for the unique continuation property in the cascade case.
Actually, we will prove that the unique continuation property is equivalent
to solving a system a $2$-coupled Fredholm integral equations of the third
kind. We apply this criterion to nontrivial examples.

This paper is organized as follows: after some preliminaries and fixing some
notations, well-posedness and equivalence with a first-order symmetric
hyperbolic system gathered in Section \ref{Section preliminaries}, we
present the main results of exact controllability of System (\ref{Wave 1})
in high frequency (weak observability) in Section \ref{Section weak
	observability}. Section \ref{Section unique continuation} is devoted to the
unique continuation issue for System (\ref{Wave 1}). \ Appendix \ref%
{Appendix} contains the proof of some technical lemmas used in the previous
sections.

\section{Preliminaries\label{Section preliminaries}}

In this section, we recall some results about well-posedness which can be
proved exactly as in the scalar case. For the proof of these results in the
scalar case, we refer for instance to \cite{Zhang} and the references
therein.

\begin{proposition}
	\label{Exist_wave}Let $T>0.$ Under the assumption (\ref{M,a,b}), suppose
	that:%
	\begin{equation*}
		\left( y_{0},y_{1},u\right) \in L^{2}\left( 0,1\right) ^{2}\times
		H^{-1}\left( 0,1\right) ^{2}\times L^{2}\left( 0,T\right) .
	\end{equation*}%
	Then there exists a unique weak solution $y$ to System (\ref{Wave 1}) such
	that 
	\begin{equation*}
		\left( y,y_{t}\right) \in C\left( \left[ 0,T\right] ,L^{2}\left( 0,1\right)
		^{2}\times H^{-1}\left( 0,1\right) ^{2}\right) .
	\end{equation*}%
	Moreover, there exists a constant $C=C\left( T,a,b\right) >0$ such that: 
	\begin{equation*}
		\left\Vert y\right\Vert _{C\left( \left[ 0,T\right] ,L^{2}\left( 0,1\right)
			^{2}\times H^{-1}\left( 0,1\right) ^{2}\right) }\leq C\left( \left\Vert
		Bu\right\Vert _{L^{2}\left( 0,T\right) ^{2}}+\left\Vert \left(
		y_{0},y_{1}\right) \right\Vert _{L^{2}\left( 0,1\right) ^{2}\times
			H^{-1}\left( 0,1\right) ^{2}}\right) .
	\end{equation*}
\end{proposition}

The adjoint problem associated with (\ref{Wave 1}) writes:%
\begin{equation}
	\left\{ 
	\begin{array}{lll}
		\varphi _{tt}=\varphi _{xx}-M^{\ast }(a\varphi _{t}+b\varphi _{x}), & 
		\mathrm{in} & (0,T)\times (0,1), \\ 
		\varphi _{\mid x=0,1}=0,\text{ \ \ \ \ \ \ \ \ \ } & \mathrm{in} & (0,T),%
		\text{ \ \ \ \ \ \ \ \ \ \ \ \ \ \ \ \ \ \ \ \ } \\ 
		\left( \varphi ,\varphi _{t}\right) _{\mid t=T}=\left( \varphi
		_{0}^{T},\varphi _{1}^{T}\right) ,\text{ \ \ } & \mathrm{in} & (0,1).\text{
			\ \ \ \ \ \ \ \ \ \ \ \ \ \ \ \ \ \ \ \ \ }%
	\end{array}%
	\right.  \label{Wave_Adjoint}
\end{equation}

\begin{proposition}
	\label{Exist_Adjoint_wave}Let $T>0$ and assume (\ref{M,a,b}). Then:
	
	\begin{enumerate}
		\item For any 
		\begin{equation*}
			\left( \varphi _{0}^{T},\varphi _{1}^{T}\right) \in H_{0}^{1}\left(
			0,1\right) ^{2}\times L^{2}\left( 0,1\right) ^{2},
		\end{equation*}%
		there exists a unique weak solution $\varphi $ to System (\ref{Wave_Adjoint}%
		) such that 
		\begin{equation}
			\varphi \in C\left( \left[ 0,T\right] ,H_{0}^{1}\left( 0,1\right)
			^{2}\right) \cap C^{1}\left( \left[ 0,T\right] ,L^{2}\left( 0,1\right)
			^{2}\right) .  \label{weak_sol}
		\end{equation}%
		Moreover, $\varphi _{x\mid x=0,1}\in L^{2}\left( 0,T\right) ^{2}$ \ and
		there exists a constant $C=C\left( T,a,b\right) >0$ such that: 
		\begin{equation*}
			\left\Vert \left( \varphi ,\varphi _{t}\right) \right\Vert _{C\left( \left[
				0,T\right] ,H_{0}^{1}\left( 0,1\right) ^{2}\times L^{2}\left( 0,1\right)
				^{2}\right) }+\left\Vert \varphi _{x\mid x=0,1}\right\Vert _{L^{2}\left(
				0,T\right) ^{2}}\leq C\left\Vert \left( \varphi _{0}^{T},\varphi
			_{1}^{T}\right) \right\Vert _{H_{0}^{1}\left( 0,1\right) ^{2}\times
				L^{2}\left( 0,1\right) ^{2}}.
		\end{equation*}
		
		\item For any 
		\begin{equation}
			\left( \varphi _{0}^{T},\varphi _{1}^{T}\right) \in \left( H^{2}\cap
			H_{0}^{1}\left( 0,1\right) \right) ^{2}\times H_{0}^{1}\left( 0,1\right) ^{2}
			\label{strong_Sol}
		\end{equation}%
		there exists a unique strong solution $\varphi $ to System (\ref%
		{Wave_Adjoint}) such that 
		\begin{equation*}
			\varphi \in C\left( \left[ 0,T\right] ,\left( H^{2}\cap H_{0}^{1}\left(
			0,1\right) \right) ^{2}\right) \cap C^{1}\left( \left[ 0,T\right]
			,H_{0}^{1}\left( 0,1\right) ^{2}\right) \cap C^{2}\left( \left[ 0,T\right]
			,L^{2}\left( 0,1\right) ^{2}\right) ,
		\end{equation*}
	\end{enumerate}
\end{proposition}

We are interested in the controllability issue for System (\ref{Wave 1}).
Recall that System (\ref{Wave 1}) is said to be

\begin{enumerate}
	\item \emph{exactly controllable at time }$T>0$ if for any 
	\begin{equation*}
		\left( y_{0},y_{1}\right) ,\left( \tau _{0},\tau _{1}\right) \in L^{2}\left(
		0,1\right) ^{2}\times H^{-1}\left( 0,1\right) ^{2},
	\end{equation*}%
	there exists $u\in L^{2}\left( 0,T\right) $ such that the associated
	solution $y$ to (\ref{Wave 1}) satisfies 
	\begin{equation*}
		\left( y,y_{t}\right) _{\mid t=T}=\left( \tau _{0},\tau _{1}\right) ,~%
		\mathrm{in}~\left( 0,1\right) .
	\end{equation*}
	
	\item \emph{approximately controllable at time }$T>0$ if for any 
	\begin{equation*}
		\left( y_{0},y_{1}\right) ,\left( \tau _{0},\tau _{1}\right) \in L^{2}\left(
		0,1\right) ^{2}\times H^{-1}\left( 0,1\right) ^{2},
	\end{equation*}%
	and any $\varepsilon >0,$ there exists $u\in L^{2}\left( 0,T\right) $ such
	that the associated solution $y$ to (\ref{Wave 1}) satisfies:%
	\begin{equation*}
		\left\Vert \left( y,y_{t}\right) _{\mid t=T}-\left( \tau _{0},\tau
		_{1}\right) \right\Vert _{L^{2}\left( 0,1\right) ^{2}\times H^{-1}\left(
			0,1\right) ^{2}}<\varepsilon .
	\end{equation*}
\end{enumerate}

These controllability concepts are known to be connected with the
observability properties of the adjoint system (\ref{Wave_Adjoint}) (see 
\cite[Part 4, Chapter 2 ]{Zabczyk}). Namely:

\begin{itemize}
	\item System (\ref{Wave 1}) is exactly controllable at time $T>0$ if, and
	only if, there exists $C=C_{T}>0$ such that for any $\left( \varphi
	_{0},\varphi _{1}\right) \in H_{0}^{1}\left( 0,1\right) ^{2}\times
	L^{2}\left( 0,1\right) ^{2},$ the associated solution $\varphi $ to (\ref%
	{Wave_Adjoint}) satisfies the \emph{observability inequality}: 
	\begin{equation}
		\left\Vert \left( \varphi _{0}^{T},\varphi _{1}^{T}\right) \right\Vert
		_{H_{0}^{1}\left( 0,1\right) ^{2}\times L^{2}\left( 0,1\right) ^{2}}^{2}\leq
		C\int_{0}^{T}\left\vert B^{\ast }\varphi _{x}\left( t,0\right) \right\vert
		^{2}dt.  \label{Exact_Obs_1}
	\end{equation}%
	In this case, the adjoint system is said exactly observable.
	
	\item System (\ref{Wave 1}) is approximately controllable at time $T>0$ if,
	and only if, for any $\left( \varphi _{0},\varphi _{1}\right) \in
	H_{0}^{1}\left( 0,1\right) ^{2}\times L^{2}\left( 0,1\right) ^{2}$, the
	associated solution $\varphi $ to (\ref{Wave_Adjoint}) satisfies the \emph{%
		unique} \emph{continuation property}:%
	\begin{equation}
		\left( B^{\ast }\varphi _{x}\left( t,0\right) =0,~t\in \left( 0,T\right)
		\right) \Rightarrow \varphi \equiv 0~\mathrm{in}~Q_{T}.  \label{Approx_Obs_1}
	\end{equation}
\end{itemize}

To study the observability inequality (\ref{Exact_Obs_1}) for solutions to (%
\ref{Wave_Adjoint}), we transform this system into a hyperbolic system of
order one. Introduce the Riemann invariants:%
\begin{equation}
	p=\varphi _{t}-\varphi _{x},\text{ \ \ }q=\varphi _{t}+\varphi _{x},~\mathrm{%
		in~}Q_{T}.  \label{Riemann}
\end{equation}%
We will have, under assumption (\ref{weak_sol}) of Proposition \ref%
{Exist_Adjoint_wave}: 
\begin{equation}
	p_{\mid t=T}=\varphi _{1}^{T}-\frac{d\varphi _{0}^{T}}{dx}=p_{T}\in
	L^{2}\left( 0,1\right) ^{2},~q_{\mid t=T}=\varphi _{1}^{T}+\frac{d\varphi
		_{0}^{T}}{dx}=q_{T}\in L^{2}\left( 0,1\right) ^{2}.  \label{weak_sol_hyp}
\end{equation}%
Thus, $\left( p_{T},q_{T}\right) \in L^{2}\left( 0,1\right) ^{2}\times
L^{2}\left( 0,1\right) ^{2}.$ Moreover, since $\varphi _{0}^{T}\in
H_{0}^{1}\left( 0,1\right) ^{2}$ and $q_{T}-p_{T}=2\dfrac{d\varphi _{0}^{T}}{%
	dx},$ we must have:%
\begin{equation}
	\int_{0}^{1}\left( q_{T}-p_{T}\right) =0.  \label{moyenne}
\end{equation}%
Under assumption (\ref{strong_Sol}) of the same proposition, we get:%
\begin{equation}
	p_{\mid t=T}=p_{T}\in H^{1}\left( 0,1\right) ^{2},~q_{\mid t=T}=q_{T}\in
	H^{1}\left( 0,1\right) ^{2}.  \label{strong_sol_hyp}
\end{equation}%
It is readily seen that System (\ref{Wave_Adjoint}) writes:%
\begin{equation}
	\left\{ 
	\begin{array}{lll}
		p_{t}+p_{x}+M^{\ast }\left( \alpha _{1}p+\alpha _{2}q\right) =0, & \mathrm{in%
		} & Q_{T},\text{ \ \ } \\ 
		q_{t}-q_{x}+M^{\ast }\left( \alpha _{1}p+\alpha _{2}q\right) =0,\text{ \ } & 
		\mathrm{in} & Q_{T},\text{ \ \ } \\ 
		(p+q)_{\mid x=0,1}=0_{%
			\mathbb{R}
			^{2}}, & \mathrm{in} & (0,T),\text{ } \\ 
		\left( p,q\right) _{|t=T}=(p_{T},q_{T}),\text{\ \ \ \ \ \ \ \ \ \ \ \ \ \ \ }
		& \mathrm{in} & (0,1),\text{ \ }%
	\end{array}%
	\right.   \label{Z system}
\end{equation}%
where:%
\begin{equation}
	\alpha _{1}=\frac{a-b}{2},\text{ }\alpha _{2}=\frac{a+b}{2},  \label{A_1-A_2}
\end{equation}%
($a,b$ and $M$ being defined in (\ref{M,a,b})). From which it appears in
particular that 
\begin{equation}
	\alpha _{1},\alpha _{2}\in C^{1}(\overline{Q_{T}};%
	\mathbb{R}
	).  \label{alpha1_alpha 2}
\end{equation}%
Thanks to Proposition \ref{Exist_Adjoint_wave}, the weak solution $\left(
p,q\right) $ to System \ref{Z system} associated with $\left(
p_{T},q_{T}\right) \in L^{2}\left( 0,1\right) ^{2}\times L^{2}\left(
0,1\right) ^{2}$ will satisfy 
\begin{equation*}
	\left. 
	\begin{array}{l}
		\left( p,q\right) \in C\left( \left[ 0,T\right] ,L^{2}\left( 0,1\right)
		^{2}\times L^{2}\left( 0,1\right) ^{2}\right) , \\ 
		\\ 
		\left( q-p\right) _{\mid x=0,1}\in L^{2}\left( 0,T\right) ^{2}, \\ 
		\\ 
		\left\Vert \left( p,q\right) \right\Vert _{C\left( \left[ 0,T\right]
			,L^{2}\left( 0,1\right) ^{2}\times L^{2}\left( 0,1\right) ^{2}\right)
		}+\left\Vert \left( q-p\right) _{\mid x=0,1}\right\Vert _{L^{2}\left(
			0,T\right) ^{2}} \\ 
		\leq C\left\Vert (p_{T},q_{T})\right\Vert _{L^{2}\left( 0,1\right)
			^{2}\times L^{2}\left( 0,1\right) ^{2}},%
	\end{array}%
	\right. 
\end{equation*}%
and when $\left( p_{T},q_{T}\right) \in H^{1}\left( 0,1\right) ^{2}\times
H^{1}\left( 0,1\right) ^{2}$%
\begin{equation*}
	\left. 
	\begin{array}{l}
		\left( p,q\right) \in C\left( \left[ 0,T\right] ,H^{1}\left( 0,1\right)
		^{2}\times H^{1}\left( 0,1\right) ^{2}\right) \cap C^{1}\left( \left[ 0,T%
		\right] ,L^{2}\left( 0,1\right) ^{2}\times L^{2}\left( 0,1\right)
		^{2}\right) , \\ 
		\\ 
		\left( q-p\right) _{\mid x=0,1}\in L^{2}\left( 0,T\right) ^{2}, \\ 
		\\ 
		\left\Vert \left( p,q\right) \right\Vert _{C\left( \left[ 0,T\right]
			,L^{2}\left( 0,1\right) ^{2}\times L^{2}\left( 0,1\right) ^{2}\right)
		}+\left\Vert \left( q-p\right) _{\mid x=0,1}\right\Vert _{L^{2}\left(
			0,T\right) ^{2}} \\ 
		\leq C\left\Vert (p_{T},q_{T})\right\Vert _{L^{2}\left( 0,1\right)
			^{2}\times L^{2}\left( 0,1\right) ^{2}},%
	\end{array}%
	\right. 
\end{equation*}

Conversely, if $\left( p,q\right) $ is a solution to (\ref{Z system})
associated with $(p_{T},q_{T})\in H^{1}\left( 0,1\right) ^{2}\times
H^{1}\left( 0,1\right) ^{2}$ satisfying (\ref{moyenne}), then there exists $%
\varphi \in H^{1}\left( Q_{T}\right) ^{2}$ such that:%
\begin{equation*}
	\left( 
	\begin{array}{c}
		\varphi _{t} \\ 
		\varphi _{x}%
	\end{array}%
	\right) =\left( 
	\begin{array}{c}
		\dfrac{q+p}{2} \\ 
		\dfrac{q-p}{2}%
	\end{array}%
	\right) ,
\end{equation*}%
since$~$in$~Q_{T},$ from System (\ref{Z system}), the scalar curl of $\left(
q+p,q-p\right) ^{t}$ is: 
\begin{equation*}
	\left( q+p\right) _{x}-\left( q-p\right) _{t}=\left( p_{x}+p_{t}\right)
	-\left( q_{t}-q_{x}\right) \equiv 0.
\end{equation*}%
Moreover, taking into account the definition of $\alpha _{1}~$and $\alpha
_{2}$ in (\ref{A_1-A_2}), it is straightforward that:%
\begin{equation*}
	\varphi _{tt}-\varphi _{xx}=-M^{\ast }\left( a\varphi _{t}+b\varphi
	_{x}\right) ,~\mathrm{in}~Q_{T}.
\end{equation*}%
We note moreover that%
\begin{equation*}
	\varphi _{x}=\dfrac{q-p}{2}\Rightarrow \varphi \left( t,x\right)
	=\int_{0}^{x}\dfrac{q-p}{2}\left( t,\xi \right) d\xi +C.
\end{equation*}%
From (\ref{Z system}), it appears that 
\begin{equation*}
	\left( q-p\right) _{t}=\left( q+p\right) _{x}\Rightarrow \left(
	\int_{0}^{1}\left( q-p\right) \right) _{t}=0\Rightarrow \int_{0}^{1}\left(
	q-p\right) =0,
\end{equation*}%
the last equality coming from (\ref{moyenne}) and the continuity in time of $%
\left( p,q\right) $. It follows that: 
\begin{equation*}
	\varphi _{\mid x=0,1}=0,~\mathrm{in}~\left( 0,T\right) .
\end{equation*}

To summarize, let us introduce the space:%
\begin{equation}
	H=\left\{ \left( f,g\right) \in L^{2}\left( 0,1\right) ^{2}\times
	L^{2}\left( 0,1\right) ^{2},~\int_{0}^{1}\left( f-g\right) =0\right\} .
	\label{H}
\end{equation}%
This is clearly a closed subspace of $L^{2}\left( 0,1\right) ^{2}\times
L^{2}\left( 0,1\right) ^{2}$and thus a Hilbert space with the usual norm
(and scalar product) of $L^{2}\left( 0,1\right) ^{2}\times L^{2}\left(
0,1\right) ^{2}.$ In view of the previous considerations, we have:

\begin{proposition}
	Let $T>0$ and $H$ defined in (\ref{H}).
	
	\begin{enumerate}
		\item For any $\left( p_{T},q_{T}\right) \in H,$ there exists a unique weak
		solution $\left( p,q\right) $ to (\ref{Z system}) such that $\left(
		p,q\right) \in C\left( \left[ 0,T\right] ,H\right) .$ Moreover $\left(
		p-q\right) _{\mid x=0,1}\in L^{2}\left( 0,T\right) ^{2}$ and there exists a
		constant $C=C\left( T,\alpha _{1},\alpha _{2}\right) >0$ such that: 
		\begin{equation*}
			\left\Vert \left( p,q\right) \right\Vert _{C\left( \left[ 0,T\right]
				,H\right) }+\left\Vert \left( p-q\right) _{\mid x=0,1}\right\Vert
			_{L^{2}\left( 0,T\right) ^{2}}\leq C\left\Vert (p_{T},q_{T})\right\Vert _{H}.
		\end{equation*}
		
		\item The observability inequality (\ref{Exact_Obs_1}) is equivalent to 
		\begin{equation}
			\left\Vert (p_{T},q_{T})\right\Vert _{H}^{2}\leq C_{T}\int_{0}^{T}\left\vert
			B^{\ast }p\left( t,0\right) \right\vert ^{2}dt.  \label{Exact_Obs_2}
		\end{equation}
	\end{enumerate}
\end{proposition}

We will need in an essential way the block diagonal system associated with
System (\ref{Z system}):

\begin{equation}
	\left\{ 
	\begin{array}{lll}
		p_{t}+p_{x}+M^{\ast }\alpha _{1}p=0,\text{\ } & \mathrm{in} & Q_{T},\text{ \
			\ } \\ 
		q_{t}-q_{x}+M^{\ast }\alpha _{2}q=0,\text{ } & \mathrm{in} & Q_{T},\text{ \
			\ } \\ 
		(p+q)_{\mid x=0,1}=0_{%
			\mathbb{R}
			^{2}}, & \mathrm{in} & (0,T),\text{ } \\ 
		\left( p,q\right) _{|t=T}=(p_{T},q_{T}),\text{ \ \ } & \mathrm{in} & (0,1).%
		\text{ \ }%
	\end{array}%
	\right.  \label{Z diag system}
\end{equation}

System (\ref{Z system}) is a perturbation of System (\ref{Z diag system}) by
the multiplication operator defined on $H$ by:

\begin{equation}
	\mathcal{P}\left( 
	\begin{array}{c}
		p \\ 
		q%
	\end{array}%
	\right) =\left( 
	\begin{array}{cc}
		0_{2\times 2} & \alpha _{2}I_{2\times 2} \\ 
		\alpha _{1}I_{2\times 2} & 0_{2\times 2}%
	\end{array}%
	\right) \left( 
	\begin{array}{c}
		p \\ 
		q%
	\end{array}%
	\right) .  \label{P}
\end{equation}%
The plan now is the following:

\begin{enumerate}
	\item In a first step (Section \ref{Section weak observability}), we will
	give necessary and sufficient condition for the solution to the diagonal
	system (\ref{Z diag system}) to satisfy the observability inequality (\ref%
	{Exact_Obs_2}).
	
	\item In a second step and in the same section (Subsection \ref%
	{Subsectioncompactness}), and in the same spirit of \cite{Neves}, we will
	prove that if the solutions of (\ref{Z diag system}) satisfy (\ref%
	{Exact_Obs_2}), then up to a finite dimensional subspace of initial data in $%
	H,$ the same is true for solutions to System (\ref{Z system}). More
	precisely, we will prove that there exists a compact operator $%
	N:H\rightarrow L^{2}(0,T),$ such that the following \emph{weak observability}
	inequality 
	\begin{equation}
		\left\Vert (p_{T},q_{T})\right\Vert _{H}^{2}\leq C_{T}\int_{0}^{T}\left\vert
		B^{\ast }p\left( t,0\right) \right\vert ^{2}dt+\left\Vert
		N(p_{T},q_{T})\right\Vert _{L^{2}(0,T)}^{2},  \label{weak obsevation}
	\end{equation}
	
	holds. Actually, we will see that $N:=p(t,0)-p_{d}(t,0)$ where $p$ and $%
	p_{d} $ are the solutions of Systems (\ref{Z system}) and (\ref{Z diag
		system}) respectively. Note that by the \emph{weak observability} inequality
	entails exact controllability up to finite dimensional space of target
	states (which are known as the invisible states). More precisely, the \emph{%
		observability }inequality (\ref{Exact_Obs_2}) holds in the orthogonal in $H$
	of the operator $(p_{T},q_{T})\mapsto p(t,0)$ which is finite co-dimensional
	space. (See \cite[Lemma 3]{Peetre}).
	
	\item The last step (Section \ref{Section unique continuation}) will provide
	sufficient (and necessary in some cases) for the unique continuation
	property to be satisfied (or the Fattorini criterion) for some particular
	matrix $M$ and functions $\alpha _{1}$ and $\alpha _{2}.$ Nontrivial
	examples will be developed at the end the section.
\end{enumerate}

\section{Weak observability\label{Section weak observability}}

In the block diagonal system (\ref{Z diag system}), the change of variables $%
t\rightleftharpoons T-t$ (we keep the same notations) leads to a system of
the from%
\begin{equation}
	\left\{ 
	\begin{array}{lll}
		p_{t}+p_{x}-M^{\ast }\eta _{1}p=0, & \mathrm{in} & Q_{T},\text{ \ \ } \\ 
		q_{t}-q_{x}-M^{\ast }\eta _{2}q=0,\text{\ } & \mathrm{in} & Q_{T},\text{ \ \ 
		} \\ 
		(p+q)_{\mid x=0,1}=0_{%
			\mathbb{R}
			^{2}}, & \mathrm{in} & (0,T),\text{ } \\ 
		\left( p,q\right) _{|t=0}=(p_{0},q_{0}),\text{\ \ \ \ } & \mathrm{in} & 
		(0,1),\text{ \ }%
	\end{array}%
	\right.  \label{Z diag system_bis}
\end{equation}%
where%
\begin{equation}
	\eta _{1}\left( t,x\right) =\alpha _{2}\left( T-t,x\right) ,~\eta _{2}\left(
	t,x\right) =\alpha _{1}\left( T-t,x\right) ,~\left( t,x\right) \in Q_{T}.
	\label{eta_1-eta_2}
\end{equation}%
Note that the observed component does not change since $(p+q)_{\mid x=0}=0_{%
	\mathbb{R}
	^{2}}.$

In this subsection, for any numbers $s,T$ such that $0<s<T$, we compute the
explicit solution $Z=(p,q)$ to the system%
\begin{equation}
	\left\{ 
	\begin{array}{lll}
		p_{t}+p_{x}-M^{\ast }\eta _{1}p=0, & \mathrm{in} & Q_{T},\text{ \ \ } \\ 
		q_{t}-q_{x}-M^{\ast }\eta _{2}q=0, & \mathrm{in} & Q_{T},\text{ \ \ } \\ 
		(p+q)_{\mid x=0,1}=0_{%
			\mathbb{R}
			^{2}}, & \mathrm{in} & (0,T),\text{ } \\ 
		\left( p,q\right) _{|t=s}=(p_{s},q_{s}),\text{ \ \ } & \mathrm{in} & (0,1).%
	\end{array}%
	\right.  \label{Z diag s}
\end{equation}%
For given real numbers $s,t$ such that $0\leq s<t$, the function $%
Z(t,x)=Z(t,x;s,Z_{s})=(p,q)(t,x;s,Z_{s})$ for $t\in (s,T)$ and $x\in (0,1)$
will denote the solution to (\ref{Z diag s}) with its dependence on the
starting time $s\geq 0$ and the initial data $Z_{s}=\left(
p_{s},q_{s}\right) \in H.$

When $s=0$, we simply write $Z(t,x)=Z(t,x;Z_{0})$ but unless necessary, 
\emph{all along this section, }$Z=\left( p,q\right) $\emph{\ will denote the
	solution to (\ref{Z diag system_bis}).}

The following assumption is fixed and is assumed in all the results of this
section:%
\begin{equation*}
	\eta _{i}\in C^{1}\left( \overline{Q_{T}},\mathbb{R}\right) ,~i=1,2.
\end{equation*}%
It is simply derived from the assumption on $a,b$ in (\ref{M,a,b}).

Notice that the exact observability property of System (\ref{Z diag
	system_bis}) amounts to the observability inequality:%
\begin{equation}
	\exists C_{T}>0,~\left\Vert \left( p_{0},q_{0}\right) \right\Vert
	_{H}^{2}\leq C_{T}\int_{0}^{T}\left\vert B^{\ast }p\left( t,0\right)
	\right\vert ^{2}dt,~\forall \left( p_{0},q_{0}\right) \in H.  \label{OBS}
\end{equation}

To express more compactly the formulas for the solutions to (\ref{Z diag s}%
), we introduce the function $\phi :\mathbb{R}_{+}\mathbb{\times }\mathbb{R}%
_{+}\rightarrow \mathbb{R}$ defined by%
\begin{equation}
	\phi \left( t,s\right) =\left\{ 
	\begin{array}{ll}
		0, & \mathrm{if}~~t\leq s, \\ 
		&  \\ 
		\int_{\max \{s,t-2\}}^{\max \{s,t-1\}}\eta _{1}(\tau ,\tau -\left(
		t-2\right) )d\tau +\int_{\max \{s,t-1\}}^{t}\eta _{2}(\tau ,t-\tau )d\tau ,
		& \mathrm{if}~~t>s,%
	\end{array}%
	\right.  \label{phi}
\end{equation}%
and the sequence of functions:%
\begin{equation}
	f_{n}(t,s)=\sum_{k=0}^{n}\phi (t-2k,s),\text{ }n\geq 0.  \label{f_n}
\end{equation}%
When $s=0$, we simply write:%
\begin{equation*}
	\phi (t,0)=\phi (t),\text{ \ }f_{n}(t,0)=f_{n}(t),\text{ }t\in \mathbb{R},%
	\text{ }n\geq 0.
\end{equation*}

At this level, it is useful to clarify the geometric meaning of the function 
$\phi .$ Actually the characteristic curves associated with the hyperbolic
systems (\ref{Z diag system}), (\ref{Z diag system_bis}) are the lines 
\begin{equation*}
	x+t=c_{1}\text{ },~x-t=c_{2},\text{ }c_{1},c_{2}\in 
	\mathbb{R}
	.
\end{equation*}%
Introduce the vector field $\boldsymbol{F}=\left( \frac{\eta _{1}+\eta _{2}}{%
	2},\frac{\eta _{1}-\eta _{2}}{2}\right) $ and let $\gamma _{j}$ $\left(
j=1,2\right) $ the two directions $\gamma _{1}=\left( 1,1\right) ~$and $%
\gamma _{2}=\left( 1,-1\right) $ of the characteristic lines. For the
canonical scalar product in $\mathbb{R}^{2},$ one has $\boldsymbol{F}\cdot
\gamma _{j}=\eta _{j}$, $\left( j=1,2\right) $ and for $t>0,$ $\phi \left(
t\right) $ is then the line integral of the vector field $\boldsymbol{F}$
along the line $\Gamma _{t}$ defined by the function:%
\begin{equation}
	\gamma _{t}\left( \tau \right) =\left\{ 
	\begin{array}{lll}
		\left( \tau ,\tau -\left( t-2\right) \right) , & \text{\textrm{if}} & \max
		\{0,t-2\}\leq \tau \leq \max \{0,t-1\}, \\ 
		&  &  \\ 
		\left( \tau ,t-\tau \right) , & \text{\textrm{if}} & \max \{0,t-1\}\leq \tau
		\leq t.%
	\end{array}%
	\right.  \label{gamma}
\end{equation}%
As a consequence, for $t>0$ and $n\geq 0,$ the function $f_{n}\left(
t\right) $ is the line integral of the vector field $\boldsymbol{F}$ along
the lines $\cup _{1\leq k\leq n+1}\Gamma _{t-2k}$ with the convention that
if $t-2k<0,$ $\Gamma _{t-2k}=\emptyset $ (see the figure below for the
representation of these lines). 
\begin{figure}[H]
	\centering
	\label{figure1} 
	\begin{tikzpicture}[scale=1.2]
		\draw[line width=1.2pt] (0,0) -- (0,2);
		\draw[line width=1.2pt] (0,0) -- (8,0);
		\draw[line width=1.2pt] (0,2) -- (8,2);
		\draw [line width=1.2pt](0,0) -- (8,0);
		\draw [line width=1pt,color=red](5,2) -- (7,0);
		\draw [line width=1pt,color=red](3,0) -- (5,2);
		\draw [line width=1pt,color=red](3,0) -- (1,2);
		\draw [line width=1pt,color=red](1,2) -- (0,1);
		\draw (6.9,0) node[anchor=north west] {$t$};
		\draw (4.7,2.4) node[anchor=north west] {$t-1$};
		\draw (0.7,2.4) node[anchor=north west] {$t-3$};
		\draw (2.7,0) node[anchor=north west] {$t-2$};
		\draw (-0.9,2.2) node[anchor=north west] {$x=1$};
	\end{tikzpicture}
	\caption{$\cup _{1\leq k\leq 2}\Gamma _{t-2k}$ is represented by the union
		of the reflected red lines on the boundary.}
\end{figure}
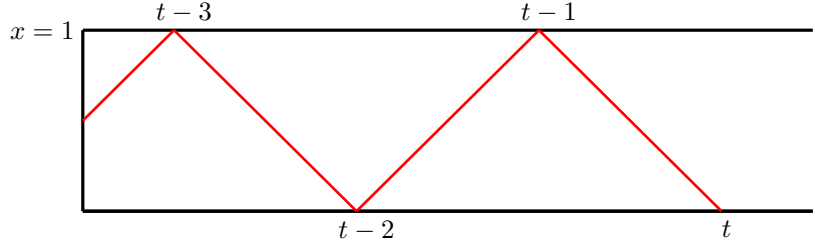

\subsection{Main results}

As pointed out above, observability inequality for System (\ref{Z system})
will hold modulo compact operator. A classical functional analysis result
shows that the space of invisible target is finite codimension and it might
be reduced to zero if approximate controllability (or the unique
continuation property) holds (See Section \ref{Section unique continuation}).

We start by a negative controllability result:

\begin{theorem}
	If $T<4,$ the weak observability inequality (\ref{weak obsevation}) doesn't
	hold. More precisely, there is an infinite dimensional space of unreachable
	target states.
\end{theorem}

For the proof see Proposition \ref{control_T<4} and Remark \ref{Remark not
	control}.

Now, we present a positive controllability results. Denote by $\lambda
_{1},\lambda _{2}$ the eigenvalues of $M^{\ast }$ if it is diagonalizable
and by $\mu $ the multiple eigenvalue of $M^{\ast }$ if it is not. We have
the following controllability result:

\begin{theorem}
	\label{Theorem control}Let $n\geq 2$ be an integer.
	
	\begin{itemize}
		\item If $2n\leq T<2n+1$. Then System (\ref{Z system}) is weakly observable
		(see (\ref{weak obsevation})) if, and only if \ the following three
		conditions are satisfied:
		
		\begin{enumerate}
			\item $\mathrm{rank}\left[ B\mid MB\right] =2.$
			
			\item For any $x\in \left[ 0,1\right] ,$ there exists $1\leq k\leq n-1$ such
			that:%
			\begin{equation*}
				\left\{ 
				\begin{array}{lll}
					\phi \left( 2k+2-x\right) \neq 0, & \mathrm{if} & \lambda _{1},\lambda
					_{2}\in 
					\mathbb{R}
					\text{ \textrm{or} }\sigma (M)=\{\mu \}, \\ 
					&  &  \\ 
					\phi \left( 2k+2-x\right) \notin \text{ }\frac{\pi }{\mathfrak{\Im }(\lambda
						_{1})}%
					\mathbb{Z}
					, & \mathrm{if} & \lambda _{1},\lambda _{2}\in 
					\mathbb{C}
					\backslash 
					\mathbb{R}
					.%
				\end{array}%
				\right.
			\end{equation*}
			
			\item For any $x\in \lbrack 0,T-2n)$ and $x^{\ast }\in \lbrack T-2n,1)$,
			there exist $1\leq k\leq n$ and $1\leq k^{\ast }\leq n-1$ respectively such
			that:%
			\begin{equation*}
				\left\{ 
				\begin{array}{lll}
					\phi \left( x+2k\right) \neq 0,\phi \left( x^{\ast }+2k^{\ast }\right) \neq
					0, & \mathrm{if} & \lambda _{1},\lambda _{2}\in 
					\mathbb{R}
					\text{ \textrm{or} }\sigma (M)=\{\mu \}, \\ 
					&  &  \\ 
					\phi \left( x+2k\right) ,\phi \left( x^{\ast }+2k^{\ast }\right) \notin 
					\text{ }\frac{\pi }{\mathfrak{\Im }(\lambda _{1})}%
					\mathbb{Z}
					, & \mathrm{if} & \lambda _{1},\lambda _{2}\in 
					\mathbb{C}
					\backslash 
					\mathbb{R}
					.%
				\end{array}%
				\right.
			\end{equation*}
		\end{enumerate}
		
		\item If $2n+1\leq T<2n+2$. Then System (\ref{Z system}) is weakly
		observable (see (\ref{weak obsevation})) if, and only if \ the following
		three conditions are satisfied:
		
		\begin{enumerate}
			\item $\mathrm{rank}\left[ B\mid MB\right] =2.$
			
			\item For any $x\in \lbrack 2n+2-T,1)$ and $x^{\ast }\in \lbrack 0,2n+2-T)$,
			there exist $1\leq k\leq n$ and\ $1\leq k^{\ast }\leq n-1$ respectively such
			that:%
			\begin{equation*}
				\left\{ 
				\begin{array}{lll}
					\phi \left( 2k+2-x\right) \neq 0,\phi \left( 2k^{\ast }+2-x^{\ast }\right)
					\neq 0, & \mathrm{if} & \lambda _{1},\lambda _{2}\in 
					\mathbb{R}
					\text{ \textrm{or} }\sigma (M)=\{\mu \}, \\ 
					&  &  \\ 
					\phi \left( 2k+2-x\right) ,\phi \left( 2k^{\ast }+2-x^{\ast }\right) \notin 
					\text{ }\frac{\pi }{\mathfrak{\Im }(\lambda _{1})}%
					\mathbb{Z}
					, & \mathrm{if} & {\small \lambda }_{1}{\small ,\lambda }_{2}{\small \in 
						\mathbb{C}
						\backslash 
						\mathbb{R}
						.}%
				\end{array}%
				\right.
			\end{equation*}
			
			\item For any $x\in \lbrack 0,1],$ there exists $1\leq k\leq n$ such that:%
			\begin{equation*}
				\left\{ 
				\begin{array}{lll}
					\phi \left( x+2k\right) \neq 0, & \mathrm{if} & \lambda _{1},\lambda _{2}\in 
					\mathbb{R}
					\text{ \textrm{or} }\sigma (M)=\{\mu \}, \\ 
					&  &  \\ 
					\phi \left( x+2k\right) \notin \text{ }\frac{\pi }{\mathfrak{\Im }(\lambda
						_{1})}%
					\mathbb{Z}
					, & \mathrm{if} & \lambda _{1},\lambda _{2}\in 
					\mathbb{C}
					\backslash 
					\mathbb{R}
					.%
				\end{array}%
				\right.
			\end{equation*}
		\end{enumerate}
	\end{itemize}
\end{theorem}

The proof is rather long, it relies on studying the exact controllability of
the diagonal block system (\ref{Z diag system_bis}) combined with a
compactness argument. See Propositions \ref{CNS_2n<T<2n+1} and \ref%
{CNS_2n+1<T<2n+2} and Remark \ref{Remark explanation}.

Let us make several observations:

\begin{remark}
	To illustrate geometrically the assertions of the above Theorem, we recall
	that from each point $\left( 0,x\right) $ (with $x\in \left[ 0,1\right] $)
	come two characteristics which stop at some point of the line $t=T.$ If for
	example $2n\leq T<2n+1$ with $n\geq 2,$ these characteristics touch the
	observability boundary $\left[ 0,T\right] \times \left\{ 0\right\} $ at
	least two times at points of the form $\left( 2k-x,0\right) $ for one of
	them and of the form $\left( 2l+x,0\right) $ for the other. The conditions
	on $\phi $ means that there exist at least two consecutive points of the
	form $\left( 2k-x,0\right) $ and two consecutive points of the form $\left(
	2l+x,0\right) $ such that the line integrals of the vector field $%
	\boldsymbol{F}=\left( \frac{\eta _{1}+\eta _{2}}{2},\frac{\eta _{1}-\eta _{2}%
	}{2}\right) ,$ namely $\int_{\Gamma _{\left( 2\ell +x,0\right) }^{\left(
			2\left( \ell +1\right) +x,0\right) }}F\cdot \gamma $ and $\int_{\Gamma
		_{\left( 2k-x,0\right) }^{\left( 2\left( k+1\right) -x,0\right) }}F\cdot
	\gamma ,$ are not, depending on the coupling matrix nature, zero or are not
	in some discrete set.
\end{remark}

\begin{remark}
	\label{Remark b independent of time}Let us first recall that for any $t\geq
	2,$ the function $\phi $ defined in (\ref{phi}) is given by%
	\begin{equation*}
		\phi \left( t\right) =\int_{t-2}^{t-1}\eta _{1}(\tau ,\tau -\left(
		t-2\right) )d\tau +\int_{t-1}^{t}\eta _{2}(\tau ,t-\tau )d\tau .
	\end{equation*}%
	Then, by (\ref{eta_1-eta_2}) and (\ref{alpha1_alpha 2}) we get%
	\begin{equation*}
		\phi \left( t\right) =\frac{1}{2}\int_{0}^{1}\left( a+b\right) \left( T-\tau
		+t-2,\tau \right) d\tau +\frac{1}{2}\int_{0}^{1}\left( a-b\right) (T-t+\tau
		,\tau )d\tau .
	\end{equation*}%
	Observe that in the autonomous case ($a$ and $b$ are time independent) the
	above formulas becomes 
	\begin{equation}
		\phi \left( t\right) =\int_{0}^{1}a(s)ds.  \label{average}
	\end{equation}%
	Hence, the coupling with first-order derivative in space doesn't have any
	influence on the controllability of System (\ref{Z diag system_bis}) in high
	frequency unless $b$ depends on time. More precisely, if we let $a=0$ and $%
	b=b(x),$ then the weak observability inequality (\ref{weak obsevation})
	doesn't hold in any time and for any $b$ since $\phi $ will be zero. The
	situation is not the same for parabolic systems. In \cite{Duprez 2},
	boundary controllability of a cascade system of two parabolic equations in $%
	1-D$ has been studied with coupling acting on first-order component. It has
	been shown that the underlying system is exactly controllable if the
	coupling function satisfies a moment assumption for the low frequency part
	and an average assumption like (\ref{average}) for the high frequency. This
	shows that differences between hyperbolic and parabolic systems are not
	limited to the geometric control condition introduced in \cite%
	{Bardos-Lebeau-Rauch} or the minimal time of control.
\end{remark}

\subsection{Construction of the solution to the diagonal system}

Given $(t,x)\in Q_{T}$, the value of $p(t,x)$ and $q(t,x)$ is determined
either by $\left( p_{s},q_{s}\right) $ or by their values at $x=0$ or $x=1$.
More precisely, we have by the characteristics method:%
\begin{equation}
	p(t,x)=\left\{ 
	\begin{array}{lll}
		\exp \left( M^{\ast }\int_{t-x}^{t}\eta _{1}(\tau ,\tau -(t-x))d\tau \right)
		p(t-x,0), & \mathrm{if} & t-x>s, \\ 
		&  &  \\ 
		\exp \left( M^{\ast }\int_{s}^{t}\eta _{1}(\tau ,\tau -(t-x))d\tau \right)
		p_{s}(x-t+s), & \mathrm{if} & t-x<s,%
	\end{array}%
	\right.  \label{p(t,x)}
\end{equation}%
and%
\begin{equation}
	q(t,x)=\left\{ 
	\begin{array}{lll}
		\exp \left( M^{\ast }\int_{x+t-1}^{t}\eta _{2}(\tau ,t+x-\tau )d\tau \right)
		q(t+x-1,1), & \mathrm{if} & t+x-1>s, \\ 
		&  &  \\ 
		\exp \left( M^{\ast }\int_{s}^{t}\eta _{2}(\tau ,t+x-\tau )d\tau \right)
		q_{s}(x+t-s), & \mathrm{if} & t+x-1<s.%
	\end{array}%
	\right.  \label{q(t,x)}
\end{equation}%
Thus, computing $p(t,x)$ and $q(t,x)$ amounts to evaluate $p(t,0)$ and $%
q(t,1)$ (respectively) as functions of the initial data $(p_{0},q_{0})$,
keeping in mind the boundary conditions. The following lemma can be proved
by induction:

\begin{lemma}
	\label{Lemma p,q diag}Let $n\geq 0$ be an integer and $Z_{s}=(p_{s},q_{s})%
	\in H.$ Then if $Z=(p,q)$ is the solution to System (\ref{Z diag s}), one
	has:%
	\begin{equation}
		p(t,0)=-e^{f_{n}(t,s)M^{\ast }}q_{s}(t-s-2n),\text{ }\mathrm{if}\text{ }%
		2n\leq t-s<2n+1,  \label{p(t,0;s) 2n,2n+1}
	\end{equation}%
	\begin{equation}
		p(t,0)=e^{f_{n}(t,s)M^{\ast }}p_{s}(2n+2-t+s),\text{ }\mathrm{if}\text{ }%
		2n+1\leq t-s<2n+2.  \label{p(t,0;s) 2n+1,2n+2}
	\end{equation}%
	As a consequence:%
	\begin{equation}
		q(t,1)=e^{M^{\ast }\left( \int_{s}^{t}\eta _{1}(\tau ,\tau -(t-1))d\tau
			\right) }p_{s}(t-s-2n),\text{ }\mathrm{if}\text{ }0\leq t-s<1,
		\label{q(t,1) 1}
	\end{equation}%
	\begin{equation}
		q(t,1)=-e^{M^{\ast }\left( \int_{s}^{t}\eta _{1}(\tau ,\tau -(t-1))d\tau
			\right) }p(t-1,0),\text{ }\mathrm{if}\text{ }t-s>1.  \label{q(t,1)2}
	\end{equation}
\end{lemma}

\begin{proof}
	We give the proof for $s=0$, and a simple change of variable $%
	t\rightleftharpoons t-s$ leads to the formulas of the lemma.
	
	Assume $n=0$ in (\ref{p(t,0;s) 2n,2n+1})-(\ref{p(t,0;s) 2n+1,2n+2}). For $%
	0<t<1$, the characteristics method, the boundary conditions and (\ref{phi})-(%
	\ref{f_n}) give:%
	\begin{eqnarray*}
		q(t,0) &=&e^{M^{\ast }\int_{0}^{t}\eta _{2}(\tau ,t-\tau )d\tau }q_{0}(t) \\
		&=&e^{f_{0}(t)M^{\ast }}q_{0}(t).
	\end{eqnarray*}%
	For $1\leq t<2,$ as previously:%
	\begin{eqnarray*}
		q(t,0) &=&e^{M^{\ast }\int_{t-1}^{t}\eta _{2}(\tau ,t-\tau )d\tau }q(t-1,1)
		\\
		&=&-e^{M^{\ast }\int_{t-1}^{t}\eta _{2}(\tau ,t-\tau )d\tau }p(t-1,1) \\
		&=&-e^{M^{\ast }\phi \left( t\right) }p_{0}(2-t) \\
		&=&-e^{f_{0}(t)M}p_{0}(2-t).
	\end{eqnarray*}%
	Thus (\ref{p(t,0;s) 2n,2n+1})-(\ref{p(t,0;s) 2n+1,2n+2}) are verified for $%
	n=0.$
	
	Given $n\geq 0$, let us assume (\ref{p(t,0;s) 2n,2n+1}) and (\ref{p(t,0;s)
		2n+1,2n+2}). Let $2n+2\leq t<2n+3.$ Then, by the same computations using the
	characteristics method:%
	\begin{equation}
		q(t,0)=-e^{\phi \left( t\right) M^{\ast }}p(t-2,0).  \label{eq_2}
	\end{equation}%
	Since $2n\leq t-2<2n+1,$ formula (\ref{p(t,0;s) 2n,2n+1}) applies and gives%
	\begin{equation}
		p(t-2,0)=e^{f_{n}(t-2)M^{\ast }}q_{0}(t-2(n+1)).  \label{eq_3}
	\end{equation}%
	Now, from (\ref{f_n})%
	\begin{eqnarray}
		f_{n}(t-2) &=&\sum_{k=0}^{n}\phi \left( t-2-2k\right)  \notag \\
		&=&\sum_{k=1}^{n+1}\phi \left( t-2k\right)  \notag \\
		&=&f_{n+1}\left( t\right) -\phi \left( t\right) .  \label{eq_4}
	\end{eqnarray}%
	Thus, inserting (\ref{eq_3})-(\ref{eq_4}) in (\ref{eq_2}) leads to:%
	\begin{equation*}
		2n+2\leq t<2n+3\Rightarrow q\left( t,0\right) =e^{f_{n+1}\left( t\right)
			M^{\ast }}q_{0}(t-2(n+1)),
	\end{equation*}%
	and (\ref{p(t,0;s) 2n,2n+1}) is proved with $n$ replaced by $n+1.$
	
	The proof by induction of (\ref{p(t,0;s) 2n+1,2n+2}) can be performed in the
	same way. \hfill
\end{proof}

\begin{remark}
	\label{remark evolution family} \ 
	
	\begin{enumerate}
		\item System (\ref{Z diag s}) defines an evolution family $\left( U_{\mathrm{%
				diag}}(t,s)\right) _{0\leq s\leq t}$ on $H$ (see \cite{Pazy} for instance)
		which is explicitly computed by mean of formulas (\ref{p(t,x)})-(\ref{q(t,x)}%
		)-(\ref{p(t,0;s) 2n,2n+1})-(\ref{p(t,0;s) 2n+1,2n+2}): 
		\begin{equation}
			U_{\mathrm{diag}}(t,s)Z_{0}=(p,q)(t,\cdot ;s,Z_{0}),~0\leq s\leq t,~Z_{0}\in
			H.  \label{Evolution family}
		\end{equation}
		
		\item From\ (\ref{p(t,0;s) 2n,2n+1})-(\ref{p(t,0;s) 2n+1,2n+2}), the
		following formula follows: if $Z=(p,q)$ is a solution to System (\ref{Z diag
			system_bis}) associated with an initial data $Z_{s}=\left(
		p_{s},q_{s}\right) ,$ then%
		\begin{equation}
			B^{\ast }p\left( t,0;s,Z_{s}\right) =\left\{ 
			\begin{array}{ccc}
				-B^{\ast }e^{f_{n}(t,s)M^{\ast }}q_{s}(t-s-2n),\text{ \ \ \ \ \ \ } & 
				\mathrm{if} & 2n\leq t-s<2n+1,\text{ \ \ \ \ } \\ 
				&  &  \\ 
				B^{\ast }e^{f_{n}(t,s)M^{\ast }}p_{s}(2n+2-t+s), & \mathrm{if} & 2n+1\leq
				t-s<2n+2.%
			\end{array}%
			\right.  \label{B*p(t,0)}
		\end{equation}%
		for $n\geq 0.$ Formula (\ref{B*p(t,0)}) will be used in the next subsection
		in the study of the observability issue for System (\ref{Z diag system_bis}).
	\end{enumerate}
\end{remark}

\subsection{Some technical results on multiplication operators}

In order to make clear the proof of our exact observability results, we will
need some preliminary results on multiplications operators defined from $%
L^{2}\left( 0,1\right) ^{k}$ in $L^{2}\left( 0,1\right) ^{n}$ for some
positive integers $k,n.$

Let $M=\left( m_{ij}\right) _{1\leq i,j\leq n}$ a $n\times n$ matrix whose
entries satisfies $m_{ij}\in C\left( \left[ 0,1\right] ,\mathbb{R}\right) .$
The multiplication operator $\mathbb{M}:L^{2}\left( 0,1\right) ^{k}$ $%
\rightarrow L^{2}\left( 0,1\right) ^{n}$ associated with $M$ is defined by:%
\begin{equation*}
	\left( \mathbb{M}h\right) \left( x\right) =M\left( x\right) h(x),~x\in
	\left( 0,1\right) ,~h\in L^{2}\left( 0,1\right) ^{n}.
\end{equation*}%
Clearly $\mathbb{M}$ is a bounded operator. When $s=n,$ the following
characterization of the invertibility of $\mathbb{M}$ is derived from \cite[%
Proposition 2.2]{Hardt}:

\begin{proposition}
	\label{MO_sxs}The operator $\mathbb{M}:L^{2}\left( 0,1\right) ^{k}$ $%
	\rightarrow L^{2}\left( 0,1\right) ^{k}$ is invertible if, and only if:%
	\begin{equation*}
		\inf_{x\in \left[ 0,1\right] }\left\vert \det M\left( x\right) \right\vert
		>0.
	\end{equation*}
\end{proposition}

We are now interested by the case $k<n$ . \ 

\begin{proposition}
	\label{MO_sxn}Let $k<n$ and $\mathbb{M}:L^{2}\left( 0,1\right) ^{k}$ $%
	\rightarrow L^{2}\left( 0,1\right) ^{n}.$ The following properties are
	equivalent:
	
	\begin{enumerate}
		\item There exists a constant $C>0$ such that%
		\begin{equation*}
			\left\Vert h\right\Vert _{L^{2}\left( 0,1\right) ^{k}}\leq C\left\Vert 
			\mathbb{M}h\right\Vert _{L^{2}\left( 0,1\right) ^{n}},~\forall h\in
			L^{2}\left( 0,1\right) ^{k}.
		\end{equation*}
		
		\item For all $x\in \left[ 0,1\right] ,$ there exists a $s\times s$ matrix $%
		M_{\mathrm{ext}}$, extracted from $M,~$such that 
		\begin{equation*}
			\det M_{\mathrm{ext}}\left( x\right) \neq 0.
		\end{equation*}
	\end{enumerate}
\end{proposition}

\begin{proof}
	For the proof, see Appendix \ref{Appendix}.
\end{proof}

\subsection{Observability results}

Let us start with the following Lemma:

\begin{lemma}
	\label{diag-Control-1}Let $T>0$ and $\left( p,q\right) $ be the solution to (%
	\ref{Z diag system_bis}) associated with $Z_{0}=\left( p_{0},q_{0}\right)
	\in H$. Then:
	
	\begin{itemize}
		\item If $\ 2n\leq T<2n+1$ for some $n\geq 0$, one has:%
		\begin{equation}
			\int_{0}^{T}\left\vert B^{\ast }p\left( t,0\right) \right\vert
			^{2}dt=\int_{0}^{T}\left\vert B^{\ast }e^{f_{0}(x)M^{\ast
			}}q_{0}(x)\right\vert ^{2}dx,~\mathrm{for}~n=0,  \label{B*p(t,0)_norme_(0,1)}
		\end{equation}%
		and for any$~n\geq 1$%
		\begin{eqnarray}
			\int_{0}^{T}\left\vert B^{\ast }p\left( t,0\right) \right\vert ^{2}dt
			&=&\sum\limits_{k=0}^{n-1}\int_{0}^{1}\left\vert B^{\ast
			}e^{f_{k}(2k+2-x)M^{\ast }}p_{0}(x)\right\vert ^{2}dx
			\label{B*p(t,0)_norme_2n-} \\
			&&+\sum\limits_{k=0}^{n-1}\int_{0}^{1}\left\vert B^{\ast
			}e^{f_{k}(x+2k)M^{\ast }}q_{0}(x)\right\vert ^{2}dx  \notag \\
			&&+\int_{0}^{T-2n}\left\vert B^{\ast }e^{f_{n}(x+2n)M^{\ast
			}}q_{0}(x)\right\vert ^{2}dx.  \notag
		\end{eqnarray}
		
		\item If $\ 2n+1\leq T<2n+2$ for some $n\geq 0$, then 
		\begin{equation}
			\int_{0}^{T}\left\vert B^{\ast }p\left( t,0\right) \right\vert
			^{2}dt=\int_{0}^{T-1}\left\vert B^{\ast }e^{f_{0}(2-x)M^{\ast
			}}p_{0}(x)\right\vert ^{2}dx+\int_{0}^{1}\left\vert B^{\ast
			}e^{f_{0}(x)M^{\ast }}q_{0}(x)\right\vert ^{2}dx,\text{ }\mathrm{for}\text{%
				\textrm{\ }}n=0,  \label{B*p(t,0)_norme_(1,2)}
		\end{equation}%
		and for any $n\geq 1$%
		\begin{eqnarray}
			\int_{0}^{T}\left\vert B^{\ast }p\left( t,0\right) \right\vert ^{2}dt
			&=&\sum\limits_{k=0}^{n}\int_{0}^{1}\left\vert e^{f_{k}(x+2k)M^{\ast
			}}q_{0}(x)\right\vert ^{2}dx  \label{B*p(t,0)_norme_2n+1} \\
			&&+\sum\limits_{k=0}^{n-1}\int_{0}^{1}\left\vert e^{f_{k}(2k+2-x)M^{\ast
			}}p_{0}(x)\right\vert ^{2}dx  \notag \\
			&&+\int_{2n+2-T}^{1}\left\vert e^{f_{n}(2n+2-x)M^{\ast }}p_{0}(x)\right\vert
			^{2}dx.  \notag
		\end{eqnarray}
	\end{itemize}
\end{lemma}

\begin{proof}
	Let $n\geq 0$ be an integer and suppose that $2n\leq T<2n+1.$ Then, if $n=0,$%
	we have from (\ref{p(t,0;s) 2n,2n+1}):%
	\begin{equation*}
		\int_{0}^{T}\left\vert B^{\ast }p\left( t,0\right) \right\vert
		^{2}dt=\int_{0}^{T}\left\vert B^{\ast }e^{f_{0}(t)M^{\ast
		}}q_{0}(t)\right\vert ^{2}dt.
	\end{equation*}%
	If $n\geq 1:$%
	\begin{eqnarray}
		\int_{0}^{T}\left\vert B^{\ast }p\left( t,0\right) \right\vert ^{2}dt
		&=&\sum\limits_{k=0}^{n-1}\left( \int_{2k}^{2k+1}+\int_{2k+1}^{2k+2}\right)
		\left\vert B^{\ast }p\left( t,0\right) \right\vert
		^{2}dt+\int_{2n}^{T}\left\vert B^{\ast }p\left( t,0\right) \right\vert ^{2}dt
		\notag \\
		&:&=\sum\limits_{k=1}^{n}\left( I_{k}+J_{k}\right) +\int_{2n}^{T}\left\vert
		B^{\ast }p\left( t,0\right) \right\vert ^{2}dt.  \label{eq_1}
	\end{eqnarray}%
	Using (\ref{B*p(t,0)}) leads to:%
	\begin{eqnarray*}
		I_{k} &=&\int_{2k}^{2k+1}\left\vert B^{\ast }p\left( t,0\right) \right\vert
		^{2}dt \\
		&=&\int_{T-2k}^{T-2k+1}\left\vert B^{\ast }e^{f_{k}(t)M^{\ast
		}}q_{0}(t-2k)\right\vert ^{2}dt \\
		&=&\int_{0}^{1}\left\vert B^{\ast }e^{f_{k}(x+2k)M^{\ast
		}}q_{0}(x)\right\vert ^{2}dx.
	\end{eqnarray*}%
	For the second integral, in the same way:%
	\begin{eqnarray*}
		J_{k} &=&\int_{2k+1}^{2k+2}\left\vert B^{\ast }p\left( t,0\right)
		\right\vert ^{2}dt \\
		&=&\int_{2k+1}^{2k+2}\left\vert B^{\ast }e^{f_{k}(t)M^{\ast
		}}p_{0}(2k+2-t)\right\vert ^{2}dt \\
		&=&\int_{0}^{1}\left\vert B^{\ast }e^{f_{k}(2k+2-x)M^{\ast
		}}p_{0}(x)\right\vert ^{2}dt.
	\end{eqnarray*}%
	And last:%
	\begin{eqnarray*}
		\int_{2n}^{T}\left\vert B^{\ast }p\left( t,0\right) \right\vert ^{2}dt
		&=&\int_{2n}^{T}\left\vert B^{\ast }e^{f_{n}(t)M^{\ast
		}}q_{0}(t-2n)\right\vert ^{2}dt \\
		&=&\int_{0}^{T-2n}\left\vert B^{\ast }e^{f_{n}(x+2n)M^{\ast
		}}q_{0}(x)\right\vert ^{2}dx.
	\end{eqnarray*}%
	Inserting the last formula in (\ref{eq_1}), we get (\ref{B*p(t,0)_norme_2n-}%
	).
	
	If $2n+1\leq T<2n+2,$ exactly as in the previous computations, if $n=0,$ we
	get: 
	\begin{eqnarray*}
		\int_{0}^{T}\left\vert B^{\ast }p\left( t,0\right) \right\vert ^{2}dt
		&=&\int_{0}^{1}\left\vert B^{\ast }p\left( t,0\right) \right\vert
		^{2}dt+\int_{1}^{T}\left\vert B^{\ast }p\left( t,0\right) \right\vert ^{2}dt
		\\
		&=&\int_{0}^{1}\left\vert B^{\ast }e^{f_{0}(x)M^{\ast }}q_{0}(x)\right\vert
		^{2}dx+\int_{0}^{T-1}\left\vert B^{\ast }e^{f_{0}(2-x)M^{\ast
		}}p_{0}(x)\right\vert ^{2}dx,
	\end{eqnarray*}%
	and if $n\geq 1:$%
	\begin{eqnarray*}
		\int_{0}^{T}\left\vert B^{\ast }p\left( t,0\right) \right\vert ^{2}dt
		&=&\left(
		\sum\limits_{k=0}^{n}\int_{2k}^{2k+1}+\sum\limits_{k=0}^{n-1}%
		\int_{2k+1}^{2k+2}\right) \left\vert B^{\ast }p\left( t,0\right) \right\vert
		^{2}dt+\int_{2n+1}^{T}\left\vert B^{\ast }p\left( t,0\right) \right\vert
		^{2}dt \\
		&=&\sum\limits_{k=0}^{n}\int_{0}^{1}\left\vert e^{f_{k}(x+2k)M^{\ast
		}}q_{0}(x)\right\vert ^{2}dx+\sum\limits_{k=0}^{n-1}\int_{0}^{1}\left\vert
		e^{f_{k}(2k+2-x)M^{\ast }}p_{0}(x)\right\vert ^{2}dx \\
		&&+\int_{2n+2-T}^{1}\left\vert e^{f_{n}(2n+2-x)M^{\ast }}p_{0}(x)\right\vert
		^{2}dx,
	\end{eqnarray*}%
	which is exactly (\ref{B*p(t,0)_norme_2n+1}). This ends the proof of the
	lemma.
\end{proof}

As an immediate consequence, we have:

\begin{corollary}
	Let $n\geq 1.$ \ 
	
	\begin{itemize}
		\item If $\ 2n\leq T<2n+1,$ a necessary and sufficient condition for exact
		observability of System (\ref{Z diag system_bis}) is \ that: 
		\begin{equation}
			\exists C_{T}>0:\int_{0}^{1}\left\vert p_{0}(x)\right\vert ^{2}dx\leq
			C_{T}\sum\limits_{k=0}^{n-1}\int_{0}^{1}\left\vert B^{\ast
			}e^{f_{k}(2k+2-x)M^{\ast }}p_{0}(x)\right\vert ^{2}dx,~\forall p_{0}\in
			L^{2}\left( 0,1\right) ^{2},  \label{NCS_Observ_2n_p0}
		\end{equation}%
		and%
		\begin{eqnarray}
			&&\sum\limits_{k=0}^{n-1}\int_{0}^{1}\left\vert B^{\ast
			}e^{f_{k}(x+2k)M^{\ast }}q_{0}(x)\right\vert
			^{2}dx+C_{T}\int_{0}^{T-2n}\left\vert B^{\ast }e^{f_{n}(x+2n)M^{\ast
			}}q_{0}(x)\right\vert ^{2}dx  \label{NCS_Observ_2n_q0} \\
			&\geq &C_{T}\int_{0}^{1}\left\vert q_{0}(x)\right\vert ^{2},\text{ }\forall
			q_{0}\in L^{2}\left( 0,1\right) ^{2}.  \notag
		\end{eqnarray}
		
		\item If $\ 2n+1\leq T<2n+2,$ a necessary and sufficient condition for exact
		observability of System (\ref{Z diag system_bis}) is:%
		\begin{eqnarray}
			&&\sum\limits_{k=0}^{n-1}\int_{0}^{1}\left\vert B^{\ast
			}e^{f_{k}(2k+2-x)M^{\ast }}p_{0}(x)\right\vert
			^{2}dx+\int_{2n+2-T}^{1}\left\vert B^{\ast }e^{f_{n}(2n+2-x)M^{\ast
			}}p_{0}(x)\right\vert ^{2}dx  \label{NCS_Observ_2n+1_p0} \\
			&\geq &C_{T}\int_{0}^{1}\left\vert p_{0}(x)\right\vert ^{2}dx,\text{ }%
			\forall p_{0}\in L^{2}\left( 0,1\right) ^{2},  \notag
		\end{eqnarray}%
		\ and 
		\begin{equation}
			\int_{0}^{1}\left\vert q_{0}(x)\right\vert ^{2}dx\leq
			C_{T}\sum\limits_{k=0}^{n}\int_{0}^{1}\left\vert B^{\ast
			}e^{f_{k}(x+2k)M^{\ast }}q_{0}(x)\right\vert ^{2}dx,~\forall q_{0}\in
			L^{2}\left( 0,1\right) ^{2}.  \label{NCS_Observ_2n+1_q0}
		\end{equation}
	\end{itemize}
\end{corollary}

\begin{proof}
	The proof is a straightforward consequence of formulas (\ref%
	{B*p(t,0)_norme_2n-}) and (\ref{B*p(t,0)_norme_2n+1}).
\end{proof}

\begin{remark}
	\label{NCS_Obs_New}Let $n\geq 2.$ For $2n\leq T<2n+1$, introduce the
	matrices:%
	\begin{equation}
		P_{2n}\left( x,T\right) =\left[ 
		\begin{array}{c}
			B^{\ast }e^{f_{0}(2-x)M^{\ast }} \\ 
			B^{\ast }e^{f_{1}(4-x)M^{\ast }} \\ 
			\vdots \\ 
			B^{\ast }e^{f_{n-1}(2n-x)M^{\ast }}%
		\end{array}%
		\right] ;\text{ }Q_{2n}\left( x,T\right) =\left[ 
		\begin{array}{c}
			B^{\ast }e^{f_{0}(x)M^{\ast }} \\ 
			\vdots \\ 
			B^{\ast }e^{f_{n-1}(x+2\left( n-1\right) )M^{\ast }} \\ 
			\mathbbm{1} _{\left( 0,T-2n\right) }(x)B^{\ast }e^{f_{n}(x+2n)M^{\ast }}%
		\end{array}%
		\right] ,  \label{P_2n-Q_2n}
	\end{equation}%
	and their associated multiplication operators $\mathbb{P}_{2n}:L^{2}\left(
	0,1\right) ^{2}\rightarrow L^{2}\left( 0,1\right) ^{n}$ and $\mathbb{Q}%
	_{2n}:L^{2}\left( 0,1\right) ^{2}\rightarrow L^{2}\left( 0,1\right) ^{n+1}.$
	With these notations, (\ref{NCS_Observ_2n_p0}) and (\ref{NCS_Observ_2n_q0})
	respectively write: 
	\begin{eqnarray}
		\exists C_{T} &>&0,\text{ }\int_{0}^{1}\left\vert p_{0}(x)\right\vert
		^{2}dx\leq C_{T}\left\Vert \mathbb{P}_{2n}p_{0}\right\Vert _{L^{2}\left(
			0,1\right) ^{n}}^{2},~\forall p_{0}\in L^{2}\left( 0,1\right) ^{2},
		\label{OMP_2n} \\
		\exists C_{T} &>&0,\text{ }\int_{0}^{1}\left\vert q_{0}(x)\right\vert
		^{2}dx\leq C_{T}\left\Vert \mathbb{Q}_{2n}q_{0}\right\Vert _{L^{2}\left(
			0,1\right) ^{n+1}}^{2},~\forall q_{0}\in L^{2}\left( 0,1\right) ^{2}.
		\label{OMQ_2n}
	\end{eqnarray}
	
	For $2n+1\leq T<2n+2$, introduce the matrices:%
	\begin{equation}
		P_{2n+1}\left( x,T\right) =\left[ 
		\begin{array}{c}
			B^{\ast }e^{f_{0}(2-x)M^{\ast }} \\ 
			\vdots \\ 
			B^{\ast }e^{f_{n-1}(2n-x)M^{\ast }} \\ 
			\mathbbm{1} _{\left( 1,2n+2-T\right) }(x)B^{\ast }e^{f_{n}(2n+2-x)M^{\ast }}%
		\end{array}%
		\right] ;~Q_{2n+1}\left( x,T\right) =\left[ 
		\begin{array}{c}
			B^{\ast }e^{f_{0}(x)M^{\ast }} \\ 
			\vdots \\ 
			B^{\ast }e^{f_{n-1}(x+2\left( n-1\right) )M^{\ast }} \\ 
			B^{\ast }e^{f_{n}(x+2n)M^{\ast }}%
		\end{array}%
		\right] ,  \label{P_2n+1-Q_2n+1}
	\end{equation}%
	and their associated multiplication operators $\mathbb{P}_{2n+1}:L^{2}\left(
	0,1\right) ^{2}\rightarrow L^{2}\left( 0,1\right) ^{n+1}$ and $\mathbb{Q}%
	_{2n}:L^{2}\left( 0,1\right) ^{2}\rightarrow L^{2}\left( 0,1\right) ^{n+1}.$
	With these notations, (\ref{NCS_Observ_2n+1_p0}) and (\ref%
	{NCS_Observ_2n+1_q0}) respectively write: 
	\begin{eqnarray}
		\exists C_{T} &>&0,\text{ }\int_{0}^{1}\left\vert p_{0}(x)\right\vert
		^{2}dx\leq C_{T}\left\Vert \mathbb{P}_{2n+1}p_{0}\right\Vert _{L^{2}\left(
			0,1\right) ^{n+1}}^{2},~\forall p_{0}\in L^{2}\left( 0,1\right) ^{2},
		\label{OMP_2n+1} \\
		\exists C_{T} &>&0,\text{ }\int_{0}^{1}\left\vert q_{0}(x)\right\vert
		^{2}dx\leq C_{T}\left\Vert \mathbb{Q}_{2n+1}p_{0}\right\Vert _{L^{2}\left(
			0,1\right) ^{n+1}}^{2},~\forall q_{0}\in L^{2}\left( 0,1\right) ^{2}.
		\label{OMQ_2n+1}
	\end{eqnarray}
\end{remark}

We are ready to state our first (negative) results on the controllability of
System (\ref{Z diag system_bis}).

\begin{proposition}
	\label{control_T<4} For $T<4$, there exists an infinite dimensional subspace
	of initial data $\left( p_{0},q_{0}\right) \in H$ for which the exact
	observability inequalities (\ref{NCS_Observ_2n_p0})-(\ref{NCS_Observ_2n+1_q0}%
	) are not satisfied by the associated solution $\left( p,q\right) $ to
	System (\ref{Z diag system_bis}).
\end{proposition}

\begin{proof}
	The goal is to prove (\ref{OBS}) does not hold for any $0\leq T<4$. If $%
	0\leq T<1,$ from (\ref{B*p(t,0)_norme_(0,1)}), one has: 
	\begin{equation*}
		\int_{0}^{T}\left\vert B^{\ast }p\left( t,0\right) \right\vert
		^{2}dt=\int_{0}^{T}\left\vert B^{\ast }e^{f_{0}(x)M^{\ast
		}}q_{0}(x)\right\vert ^{2}dx,
	\end{equation*}%
	and clearly the observability inequality (\ref{OBS}) does not hold for all $%
	\left( p_{0},q_{0}\right) \in H\times V_{T}$ where 
	\begin{equation*}
		V_{T}=\left\{ q_{0}\in H:q_{0}\cdot e^{f_{0}M}B=0~\mathrm{in}~\left(
		0,T\right) \right\} .
	\end{equation*}%
	If $\ 1\leq T<2$, from (\ref{B*p(t,0)_norme_(1,2)})$:$%
	\begin{equation*}
		\int_{0}^{T}\left\vert B^{\ast }p\left( t,0\right) \right\vert
		^{2}dt=\int_{0}^{T-1}\left\vert B^{\ast }e^{f_{0}(2-x)M^{\ast
		}}p_{0}(x)\right\vert ^{2}dx+\int_{0}^{1}\left\vert B^{\ast
		}e^{f_{0}(x)M^{\ast }}q_{0}(x)\right\vert ^{2}dx.
	\end{equation*}%
	The observability inequality (\ref{OBS}) does not hold for all nontrivial $%
	\left( p_{0},q_{0}\right) \in U_{T}\times H$ (for instance) where%
	\begin{equation*}
		U_{T}=\left\{ p_{0}\in H:p_{0}\cdot e^{f_{0}\left( 2-\cdot \right) M}B=0~%
		\mathrm{in}~\left( 0,T-1\right) \right\} .
	\end{equation*}%
	If $2\leq T<3,$ from \ (\ref{B*p(t,0)_norme_2n-}) follows the equality:%
	\begin{eqnarray*}
		\int_{0}^{T}\left\vert B^{\ast }p\left( t,0\right) \right\vert ^{2}dt
		&=&\int_{0}^{1}\left\vert B^{\ast }e^{f_{0}(2-x)M^{\ast
		}}p_{0}(x)\right\vert ^{2}dx \\
		&&+\int_{0}^{1}\left\vert B^{\ast }e^{f_{0}(x)M^{\ast }}q_{0}(x)\right\vert
		^{2}dx+ \\
		&&+\int_{0}^{T-2}\left\vert B^{\ast }e^{f_{1}(x+2)M^{\ast
		}}q_{0}(x)\right\vert ^{2}dx,
	\end{eqnarray*}%
	and again (\ref{OBS}) does not hold for all $\left( p_{0},q_{0}\right) \in
	U_{T=1}\times H$.
	
	Last, for $3\leq T<4,$ from (\ref{B*p(t,0)_norme_2n+1}):%
	\begin{eqnarray*}
		\int_{0}^{T}\left\vert B^{\ast }p\left( t,0\right) \right\vert ^{2}dt
		&=&\sum\limits_{k=0}^{n-1}\int_{0}^{1}\left\vert B^{\ast
		}e^{f_{0}(2-x)M^{\ast }}p_{0}(x)\right\vert ^{2}dx \\
		&&+\int_{0}^{T-3}\left\vert B^{\ast }e^{f_{1}(4-x)M^{\ast
		}}p_{0}(x)\right\vert ^{2}dx \\
		&&+\sum\limits_{k=0}^{1}\int_{0}^{1}\left\vert B^{\ast
		}e^{f_{k}(x+2k)M^{\ast }}q_{0}(x)\right\vert ^{2}dx,
	\end{eqnarray*}%
	and (\ref{OBS}) does not hold for all $\left( p_{0},q_{0}\right) \in W\times
	H$ where%
	\begin{equation*}
		W=U_{T=1}\cap \left\{ p_{0}\in H:~\text{\textrm{supp}}\left( p_{0}\right)
		\subset \left( T-3,1\right) \right\} .
	\end{equation*}
	
	All the introduced subspaces of non-observable initial data are actually
	infinite dimensional.
	
	In these spaces can be found initial data $\left( p_{0},q_{0}\right) $ for
	which approximate observability does not hold too: if $0<T<4,$ there exists $%
	\left( p_{0},q_{0}\right) \in H\times H$ such that $\left\Vert \left(
	p_{0},q_{0}\right) \right\Vert _{H\times H}=1$ and for which the associated
	solution $\left( p,q\right) $ satisfies:%
	\begin{equation*}
		B^{\ast }p\left( t,0\right) =0,\text{ }t\in \left( 0,T\right) .
	\end{equation*}
\end{proof}

\begin{remark}
	\label{Remark not control}The above proposition shows that if $T<4$ then the
	diagonal system (\ref{Z diag system_bis}) is not exactly controllable.
	Combining this and the fact that the difference of the input maps is compact
	(See Theorem \ref{theorem compactness}) entails that System (\ref{Z system})
	is not weakly controllable; more precisely, there exists an infinite
	dimensional space of unreachable target states.
\end{remark}

It has to be pointed out that the previous negative observability result
does not depend of the choice of $\eta _{1},\eta _{2}$ and $M.$

Before going one in the analysis, let us give a necessary condition for the
exact observability to hold:

\begin{lemma}
	\label{NC_Obs}Let $T>0.$ A necessary condition for the exact observability
	of System (\ref{Z diag system_bis}) is%
	\begin{equation}
		\mathrm{rank}\left[ B\mid MB\right] =2.  \label{rank_B-MB}
	\end{equation}
\end{lemma}

\begin{proof}
	If (\ref{rank_B-MB}) does not hold, there exists $\lambda \in \mathbb{R}$
	such that $MB=\lambda B$ (in other worlds, $B$ is an eigenvector to $M$). It
	follows that for any $r\in \mathbb{R},$ $e^{rM}B=e^{r\lambda }B~\left(
	\Leftrightarrow B^{\ast }e^{rM^{\ast }}=e^{r\lambda }B^{\ast }\right) .$
	Thus, in this case, (\ref{B*p(t,0)_norme_2n-}) writes:%
	\begin{eqnarray*}
		\int_{0}^{T}\left\vert B^{\ast }p\left( t,0\right) \right\vert ^{2}dt
		&=&\int_{0}^{1}\left( \sum\limits_{k=0}^{n-1}e^{2f_{k}(2k+2-x)\lambda
		}\right) \left\vert B^{\ast }p_{0}(x)\right\vert ^{2}dx \\
		&&+\int_{0}^{1}\left( \sum\limits_{k=0}^{n-1}e^{2f_{k}(x+2k)\lambda }\right)
		\left\vert B^{\ast }q_{0}(x)\right\vert ^{2}dx \\
		&&+\int_{0}^{T-2n}e^{2f_{n}(x+2n)\lambda }\left\vert B^{\ast
		}q_{0}(x)\right\vert ^{2}dx.
	\end{eqnarray*}%
	Clearly, the exact observability property will not hold for initial data of
	the form $p_{0}=\alpha B^{\perp }$ and $q_{0}=\beta B^{\perp }$ where $%
	\alpha ,\beta \in L^{2}\left( 0,1\right) $ and $B^{\perp }$ denotes any
	orthogonal vector to $B.$ The same conclusion is achieved starting from (\ref%
	{B*p(t,0)_norme_2n+1}).
\end{proof}

The purpose in the sequel is to give answers for the exact controllability
when $T\geq 4$. We begin by the limit case $T=4.$\ As a first step, we have
the following necessary and sufficient condition for observability:

\begin{proposition}
	\label{Prop_Obs_1}For $T=4$, System (\ref{Z diag system_bis}) is exactly
	observable if, and only if:%
	\begin{equation}
		\inf_{t\in \left[ 2,4\right] }\left\vert \det \left[ B^{\ast }\text{ }|\text{
		}B^{\ast }e^{\phi \left( t\right) M^{\ast }}\right] \right\vert >0~.
		\label{CNS_T=4}
	\end{equation}
\end{proposition}

\begin{proof}
	Let $\left( p_{0},q_{0}\right) \in H$ and $\left( p,q\right) $ the
	associated solution to System (\ref{Z diag system_bis}). For $n=2,$ the
	matrices $P_{4}$ and $Q_{4}$ defined in (\ref{P_2n-Q_2n}) write:%
	\begin{equation*}
		P_{4}\left( x\right) =\left[ 
		\begin{array}{c}
			B^{\ast }e^{f_{0}(2-x)M^{\ast }} \\ 
			B^{\ast }e^{f_{1}(4-x)M^{\ast }}%
		\end{array}%
		\right] ,~Q_{4}\left( x\right) =\left[ 
		\begin{array}{c}
			B^{\ast }e^{f_{0}(x)M^{\ast }} \\ 
			B^{\ast }e^{f_{1}(x+2)M^{\ast }}%
		\end{array}%
		\right] .
	\end{equation*}%
	From Remark \ref{NCS_Obs_New}, System (\ref{Z diag system_bis}) is exactly
	observable if, and only if, the conditions (\ref{OMP_2n})-(\ref{OMQ_2n}) are
	satisfied with $n=2$ and $T=4.$ From Proposition \ref{MO_sxs}, (\ref{OMP_2n}%
	)-(\ref{OMQ_2n}) are equivalent to 
	\begin{equation*}
		\inf_{x\in \left[ 0,1\right] }\left\vert \det P_{4}\left( x\right)
		\right\vert >0~\mathrm{and}~\inf_{x\in \left[ 0,1\right] }\left\vert \det
		Q_{4}\left( x\right) \right\vert >0.
	\end{equation*}%
	But since the multiplication operator on $L^{2}\left( 0,1\right) ^{2}$ whose
	matrix is $e^{f_{0}(2-x)M^{\ast }}~~$(resp. $e^{f_{0}(x)M^{\ast }}$) $\left(
	x\in \left( 0,1\right) \right) $ is invertible, the two last conditions are
	equivalent to the following:%
	\begin{equation*}
		\begin{array}{c}
			\inf_{x\in \left[ 0,1\right] }\left\vert \det \left( P_{4}\left( x\right)
			e^{-f_{0}(2-x)M^{\ast }}\right) \right\vert >0 \\ 
			\mathrm{and} \\ 
			\inf_{x\in \left[ 0,1\right] }\left\vert \det \left( Q_{4}\left( x\right)
			e^{-f_{0}(x)M^{\ast }}\right) \right\vert >0.%
		\end{array}%
	\end{equation*}%
	Now:%
	\begin{eqnarray*}
		P_{4}\left( x\right) e^{-f_{0}(2-x)M^{\ast }} &=&\left[ B^{\ast }\text{ }|%
		\text{ }B^{\ast }e^{\left( f_{1}(4-x)-f_{0}(2-x)\right) M^{\ast }}\right] \\
		&=&\left[ B^{\ast }\text{ }|\text{ }B^{\ast }e^{\phi \left( 4-x\right)
			M^{\ast }}\right] ,
	\end{eqnarray*}%
	and%
	\begin{eqnarray*}
		Q_{4}\left( x\right) e^{-f_{0}(x)M^{\ast }} &=&\left[ B^{\ast }\text{ }|%
		\text{ }B^{\ast }e^{\left( f_{1}(x+2\right) -f_{0}(x))M^{\ast }}\right] \\
		&=&\left[ B^{\ast }\text{ }|\text{ }B^{\ast }e^{\phi \left( x+2\right)
			M^{\ast }}\right] .
	\end{eqnarray*}%
	This leads to the desired inequalities (\ref{CNS_T=4}) after noting that $%
	3\leq 4-x\leq 4$ and $2\leq x+2\leq 3$ for $0\leq x\leq 1.$
\end{proof}

Denote by $\lambda _{1},\lambda _{2}$ the eigenvalues of $M^{\ast }$ if it
is diagonalizable and by $\mu $ the multiple eigenvalue if it is not. The
next lemma will provide an equivalent condition to (\ref{CNS_T=4}).

\begin{lemma}
	\label{Lem_det}Let $r\in \mathbb{R}.$ Then $\det \left[ B^{\ast }\text{ }|%
	\text{ }B^{\ast }e^{rM^{\ast }}\right] \neq 0$ if, and only if:%
	\begin{equation}
		\mathrm{rank}\left[ B\text{ }|\text{ }MB\right] =2\text{ }\mathrm{and}\text{ 
		}\left\{ 
		\begin{array}{lll}
			r\neq 0, & \text{\textrm{if}} & \lambda _{1},\lambda _{2}\in 
			\mathbb{R}
			\text{ \textrm{or} }\sigma (M^{\ast })=\{\mu \}, \\ 
			&  &  \\ 
			r\notin \frac{\pi }{\mathfrak{\Im }(\lambda _{1})}%
			\mathbb{Z}
			, & \text{\textrm{if}} & \lambda _{1},\lambda _{2}\in 
			\mathbb{C}
			\backslash 
			\mathbb{R}
			.%
		\end{array}%
		\right.  \label{diag}
	\end{equation}
\end{lemma}

\begin{proof}
	Let $P$ a $2\times 2$ invertible matrix and set $\widetilde{B}=P^{-1}B.$ Then%
	\begin{eqnarray*}
		\left[ B\mid e^{rM}B\right] &=&\left[ P\widetilde{B}\mid e^{rM}P\widetilde{B}%
		\right] \\
		&=&P\left[ \widetilde{B}\mid P^{-1}e^{rM}P\widetilde{B}\right] \\
		&=&P\left[ \widetilde{B}\mid e^{rP^{-1}MP}\widetilde{B}\right] .
	\end{eqnarray*}%
	But%
	\begin{equation*}
		\det \left[ B^{\ast }\text{ }|\text{ }B^{\ast }e^{rM^{\ast }}\right] \neq
		0\Leftrightarrow \det P\times \det \left[ \widetilde{B}\mid e^{rP^{-1}MP}%
		\widetilde{B}\right] \neq 0.
	\end{equation*}%
	If $M$ is diagonalizable in $\mathbb{R}$ then it admits a basis $\left\{
	V_{1},V_{2}\right\} $ of real eigenvectors associated with the real
	eigenvalues $\left\{ \lambda _{1},\lambda _{2}\right\} $. If $P=\left[
	V_{1}\mid V_{2}\right] $ is the eigenvectors matrix, we get 
	\begin{equation*}
		e^{rP^{-1}MP}\widetilde{B}=\left( 
		\begin{array}{cc}
			e^{\lambda _{1}r} & 0 \\ 
			0 & e^{\lambda _{2}r}%
		\end{array}%
		\right) \widetilde{B}.
	\end{equation*}%
	So that, if $\widetilde{B}=\left( 
	\begin{array}{c}
		\beta _{1} \\ 
		\beta _{2}%
	\end{array}%
	\right) $ (so that $B=\beta _{1}V_{1}+\beta _{2}V_{2}$)$,$ then:%
	\begin{equation*}
		\det \left[ \widetilde{B}\mid e^{rP^{-1}MP}\widetilde{B}\right] =\beta
		_{1}\beta _{2}\left( e^{\lambda _{1}r}-e^{\lambda _{2}r}\right) .
	\end{equation*}%
	Thus, in this case:%
	\begin{equation*}
		\det \left[ B^{\ast }\text{ }|\text{ }B^{\ast }e^{rM^{\ast }}\right] \neq
		0\Leftrightarrow \beta _{1}\beta _{2}\left( e^{\lambda _{1}r}-e^{\lambda
			_{2}r}\right) \neq 0\Leftrightarrow \left\{ 
		\begin{array}{c}
			\beta _{1}\beta _{2}\neq 0, \\ 
			\mathrm{and} \\ 
			e^{\lambda _{1}r}-e^{\lambda _{2}r}\neq 0.%
		\end{array}%
		\right.
	\end{equation*}%
	The condition $\beta _{1}\beta _{2}\neq 0$ expresses that $B$ is not an
	eigenvector for $M$ and this is equivalent to $\mathrm{rank}\left[ B\mid MB%
	\right] =2.$ For the second condition, one has:%
	\begin{equation*}
		e^{\lambda _{1}r}-e^{\lambda _{2}r}\neq 0\Leftrightarrow \left\{ 
		\begin{array}{lll}
			r\neq 0, & \mathrm{if} & \lambda _{1},\lambda _{2}\in 
			\mathbb{R}
			, \\ 
			&  &  \\ 
			r\notin \frac{\pi }{\mathfrak{\Im }(\lambda _{1})}%
			\mathbb{Z}
			, & \mathrm{if} & \lambda _{1},\lambda _{2}\in 
			\mathbb{C}
			\backslash 
			\mathbb{R}
			.%
		\end{array}%
		\right. \text{ }.
	\end{equation*}%
	If $M$ is not diagonalizable, then there exists a $2\times 2$ invertible
	matrix $P$ such that 
	\begin{equation*}
		M=P\left( 
		\begin{array}{cc}
			\mu & 1 \\ 
			0 & \mu%
		\end{array}%
		\right) P^{-1},
	\end{equation*}%
	then 
	\begin{equation*}
		e^{rP^{-1}MP}\widetilde{B}=\left( 
		\begin{array}{cc}
			e^{\mu r} & re^{\mu r} \\ 
			0 & e^{\mu r}%
		\end{array}%
		\right) \widetilde{B}
	\end{equation*}%
	and in this case:%
	\begin{equation*}
		\det \left[ \widetilde{B}\mid e^{rP^{-1}MP}\widetilde{B}\right] =-\beta
		_{2}^{2}re^{\mu r}.
	\end{equation*}%
	The proof follows immediately. ($\beta _{2}\neq 0$ says that $B$ is not an
	eigenvector to $M$ ). \ 
\end{proof}

An immediate consequence of Proposition \ref{Prop_Obs_1} and Lemma \ref%
{Lem_det} is the following

\begin{corollary}
	\label{CNS_T=4_bis}For $T=4,$ System (\ref{Z diag system_bis}) is exactly
	observable if, and only if, for any $t\in \left[ 2,4\right] $:%
	\begin{equation*}
		\mathrm{rank}\left[ B\text{ }|\text{ }MB\right] =2\text{ }\mathrm{and}\text{ 
		}\left\{ 
		\begin{array}{lll}
			\phi \left( t\right) \neq 0, & \text{\textrm{if}} & \lambda _{1},\lambda
			_{2}\in 
			\mathbb{R}
			\text{ \textrm{or} }\sigma (M^{\ast })=\{\mu \}, \\ 
			&  &  \\ 
			\phi \left( t\right) \notin \frac{\pi }{\mathfrak{\Im }(\lambda _{1})}%
			\mathbb{Z}
			, & \text{\textrm{if}} & \lambda _{1},\lambda _{2}\in 
			\mathbb{C}
			\backslash 
			\mathbb{R}
			.%
		\end{array}%
		\right.
	\end{equation*}
\end{corollary}

Now, we deal with the case $T\geq 4:$

\begin{proposition}
	\label{CNS_2n<T<2n+1}Let $n\geq 2$ be an integer and $2n\leq T<2n+1$. Then
	System (\ref{Z diag system_bis}) is exactly observable if, and only if \ the
	following three conditions are satisfied:
	
	\begin{enumerate}
		\item $\mathrm{rank}\left[ B\mid MB\right] =2.$
		
		\item For any $x\in \left[ 0,1\right] ,$ there exists $2\leq k\leq n$ such
		that:%
		\begin{equation*}
			\left\{ 
			\begin{array}{lll}
				\phi \left( 2k-x\right) \neq 0, & \mathrm{if} & \lambda _{1},\lambda _{2}\in 
				\mathbb{R}
				\text{ \textrm{or} }\sigma (M)=\{\mu \}, \\ 
				&  &  \\ 
				\phi \left( 2k-x\right) \notin \text{ }\frac{\pi }{\mathfrak{\Im }(\lambda
					_{1})}%
				\mathbb{Z}
				, & \mathrm{if} & \lambda _{1},\lambda _{2}\in 
				\mathbb{C}
				\backslash 
				\mathbb{R}
				.%
			\end{array}%
			\right.
		\end{equation*}
		
		\item For any $x\in \lbrack 0,T-2n)$ and $x^{\ast }\in \lbrack T-2n,1)$,
		there exist $1\leq k\leq n$ and $1\leq k^{\ast }\leq n-1$ such that:$\qquad $%
		\begin{equation*}
			\left\{ 
			\begin{array}{lll}
				\phi \left( x+2k\right) \neq 0,\phi \left( x^{\ast }+2k^{\ast }\right) \neq
				0, & \mathrm{if} & \lambda _{1},\lambda _{2}\in 
				\mathbb{R}
				\text{ \textrm{or} }\sigma (M)=\{\mu \}, \\ 
				&  &  \\ 
				\phi \left( x+2k\right) ,\phi \left( x^{\ast }+2k^{\ast }\right) \notin 
				\text{ }\frac{\pi }{\mathfrak{\Im }(\lambda _{1})}%
				\mathbb{Z}
				, & \mathrm{if} & \lambda _{1},\lambda _{2}\in 
				\mathbb{C}
				\backslash 
				\mathbb{R}
				.%
			\end{array}%
			\right.
		\end{equation*}
	\end{enumerate}
\end{proposition}

\begin{proof}
	From Remark \ref{NCS_Obs_New}, System (\ref{Z diag system_bis}) is exactly
	observable if, and only if the conditions (\ref{OMP_2n})-(\ref{OMQ_2n}) are
	satisfied where we recall that for $x\in \left[ 0,1\right] :$%
	\begin{equation*}
		P_{2n}\left( x,T\right) =\left[ 
		\begin{array}{c}
			B^{\ast }e^{f_{0}(2-x)M^{\ast }} \\ 
			B^{\ast }e^{f_{1}(4-x)M^{\ast }} \\ 
			\vdots \\ 
			B^{\ast }e^{f_{n-1}(2n-x)M^{\ast }}%
		\end{array}%
		\right] ;~Q_{2n}\left( x,T\right) =\left[ 
		\begin{array}{c}
			B^{\ast }e^{f_{0}(x)M^{\ast }} \\ 
			\vdots \\ 
			B^{\ast }e^{f_{n-1}(x+2\left( n-1\right) )M^{\ast }} \\ 
			\mathbbm{1} _{\left( 0,T-2n\right) }(x)B^{\ast }e^{f_{n}(x+2n)M^{\ast }}%
		\end{array}%
		\right] .
	\end{equation*}%
	From Proposition \ref{MO_sxn} with $s=2$ and $n$ given in the lemma, (\ref%
	{OMP_2n})-(\ref{OMQ_2n}) amount to say that for any $x\in \left[ 0,1\right] $%
	, there exist $2\times 2$ matrices $P_{2n}^{\mathrm{ext}}$ and $Q_{2n}^{%
		\mathrm{ext}}$ respectively extracted from $P_{2n}$ and $Q_{2n}$ such that%
	\begin{equation*}
		\det P_{2n}^{\mathrm{ext}}\left( x,T\right) \neq 0~\mathrm{and}~\det Q_{2n}^{%
			\mathrm{ext}}\left( x,T\right) \neq 0.
	\end{equation*}%
	Fix $x\in \left[ 0,1\right] $ and let us first deal with $P_{2n}\left(
	x,T\right) .$ We are going to prove that $P_{2n}$ satisfies the required
	property if, and only if, the following matrix satisfies it too:%
	\begin{equation*}
		\widetilde{P}_{2n}\left( x,T\right) =\left[ 
		\begin{array}{c}
			B^{\ast } \\ 
			B^{\ast }e^{\phi (4-x)M^{\ast }} \\ 
			\vdots \\ 
			B^{\ast }e^{\phi (2n-x)M^{\ast }}%
		\end{array}%
		\right]
	\end{equation*}%
	The proof of this last point is based on the identity:%
	\begin{equation}
		f_{k}\left( 2\left( k+1\right) -x\right) -f_{k-1}\left( 2k-x\right) =\phi
		\left( 2\left( k+1\right) -x\right) ,~k\geq 1,\text{ }x\in \left[ 0,1\right]
		,  \label{rel_f_k}
	\end{equation}%
	which is easily derived from the definitions of the function $\phi $ and the
	sequence $\left( f_{n}\right) $ in (\ref{phi}) and (\ref{f_n}).
	
	Assume first that there exists $x_{0}\in \left[ 0,1\right] $ such that for
	any $2\times 2$ matrices $\widetilde{P}_{2n}^{\mathrm{ext}}$ \ extracted
	from $\widetilde{P}_{2n}$ one has:%
	\begin{equation*}
		\det \widetilde{P}_{2n}^{\mathrm{ext}}\left( x_{0},T\right) =0.
	\end{equation*}%
	It follows that%
	\begin{equation*}
		\det \left[ B^{\ast }\text{ }|\text{ }B^{\ast }e^{\phi (4-x_{0})M^{\ast }}%
		\right] =0\Leftrightarrow \left\{ 
		\begin{array}{lll}
			\phi (4-x_{0})=0, & \mathrm{if} & \lambda _{1},\lambda _{2}\in 
			\mathbb{R}
			\text{ \textrm{or} }\sigma (M)=\{\mu \}, \\ 
			&  &  \\ 
			\phi (4-x_{0})\in \text{ }\frac{\pi }{\mathfrak{\Im }(\lambda _{1})}%
			\mathbb{Z}
			, & \mathrm{if} & \lambda _{1},\lambda _{2}\in 
			\mathbb{C}
			\backslash 
			\mathbb{R}
			.%
		\end{array}%
		\right. ,
	\end{equation*}%
	by the equivalence given by Lemma \ref{Lem_det}. By induction, we get that
	if for $k\geq 1,$ $\phi \left( 2k-x_{0}\right) =0$ then%
	\begin{equation*}
		\det \left[ B^{\ast }e^{\phi \left( 2k-x_{0}\right) M^{\ast }}\text{ }|\text{
		}B^{\ast }e^{\phi \left( 2\left( k+1\right) -x_{0}\right) M^{\ast }}\right]
		=\det \left[ B^{\ast }\text{ }|\text{ }B^{\ast }e^{\phi \left( 2\left(
			k+1\right) -x_{0}\right) M^{\ast }}\right] =0,
	\end{equation*}%
	\begin{equation*}
		\Updownarrow
	\end{equation*}%
	\begin{equation*}
		\left\{ 
		\begin{array}{lll}
			\phi \left( 2\left( k+1\right) -x_{0}\right) =0, & \mathrm{if} & \lambda
			_{1},\lambda _{2}\in 
			\mathbb{R}
			\text{ \textrm{or} }\sigma (M)=\{\mu \}, \\ 
			&  &  \\ 
			\phi \left( 2\left( k+1\right) -x_{0}\right) \in \text{ }\frac{\pi }{%
				\mathfrak{\Im }(\lambda _{1})}%
			\mathbb{Z}
			, & \mathrm{if} & \lambda _{1},\lambda _{2}\in 
			\mathbb{C}
			\backslash 
			\mathbb{R}
			.%
		\end{array}%
		\right.
	\end{equation*}%
	In view of (\ref{rel_f_k}), it readily follows that in this case:%
	\begin{equation*}
		P_{2n}\left( x_{0},T\right) =\left[ 
		\begin{array}{c}
			B^{\ast } \\ 
			\vdots \\ 
			B^{\ast }%
		\end{array}%
		\right] e^{\phi \left( 2-x_{0}\right) }
	\end{equation*}%
	and thus for this $x_{0},$ any $2\times 2$ matrices $P_{2n}^{\mathrm{ext}}$
	extracted from $P_{2n}\ $satisfies: $\det P_{2n}^{\mathrm{ext}}\left(
	x_{0},T\right) =0.$
	
	Conversely, if we assume there exists $x_{0}\in \left[ 0,1\right] $ such
	that for any $2\times 2$ matrices $P_{2n}^{\mathrm{ext}}$ \ extracted from $%
	P_{2n}$ one has:%
	\begin{equation*}
		\det P_{2n}^{\mathrm{ext}}\left( x_{0},T\right) =0,
	\end{equation*}%
	then again using (\ref{rel_f_k}), it is easily deduced that for any $1\leq
	k\leq n:$%
	\begin{equation*}
		\left\{ 
		\begin{array}{lll}
			\phi \left( 2k-x_{0}\right) =0, & \mathrm{if} & \lambda _{1},\lambda _{2}\in 
			\mathbb{R}
			\text{ \textrm{or} }\sigma (M)=\{\mu \}, \\ 
			&  &  \\ 
			\phi \left( 2k-x_{0}\right) \in \text{ }\frac{\pi }{\mathfrak{\Im }(\lambda
				_{1})}%
			\mathbb{Z}
			, & \mathrm{if} & \lambda _{1},\lambda _{2}\in 
			\mathbb{C}
			\backslash 
			\mathbb{R}
			.%
		\end{array}%
		\right.
	\end{equation*}%
	The same considerations hold for $Q_{2n}$ and this proves the proposition.
\end{proof}

\begin{proposition}
	\label{CNS_2n+1<T<2n+2}Let $n\geq 2$ be an integer and $2n+1\leq T<2n+2$.
	System (\ref{Z diag system_bis}) is exactly observable if, and only if \ the
	following three conditions are satisfied :
	
	\begin{enumerate}
		\item $\mathrm{rank}\left[ B\mid MB\right] =2.$
		
		\item For any $x\in \lbrack 2n+2-T,1)$ and $x^{\ast }\in \lbrack 0,2n+2-T)$,
		there exist $2\leq k\leq n+1$ and\ $2\leq k^{\ast }\leq n$ respectively such
		that:%
		\begin{equation*}
			\left\{ 
			\begin{array}{lll}
				\phi \left( 2k-x\right) ,\phi \left( 2k^{\ast }-x^{\ast }\right) \neq 0, & 
				\mathrm{if} & \lambda _{1},\lambda _{2}\in 
				\mathbb{R}
				\text{ \textrm{or} }\sigma (M)=\{\mu \}, \\ 
				&  &  \\ 
				\phi \left( 2k-x\right) ,\phi \left( 2k^{\ast }-x^{\ast }\right) \notin 
				\text{ }\frac{\pi }{\mathfrak{\Im }(\lambda _{1})}%
				\mathbb{Z}
				, & \mathrm{if} & {\small \lambda }_{1}{\small ,\lambda }_{2}{\small \in 
					\mathbb{C}
					\backslash 
					\mathbb{R}
					.}%
			\end{array}%
			\right.
		\end{equation*}
		
		\item For any $x\in \lbrack 0,1],$ there exists $1\leq k\leq n$ such that:%
		\begin{equation*}
			\left\{ 
			\begin{array}{lll}
				\phi \left( x+2k\right) \neq 0, & \mathrm{if} & \lambda _{1},\lambda _{2}\in 
				\mathbb{R}
				\text{ \textrm{or} }\sigma (M)=\{\mu \}, \\ 
				&  &  \\ 
				\phi \left( x+2k\right) \notin \text{ }\frac{\pi }{\mathfrak{\Im }(\lambda
					_{1})}%
				\mathbb{Z}
				, & \mathrm{if} & \lambda _{1},\lambda _{2}\in 
				\mathbb{C}
				\backslash 
				\mathbb{R}
				.%
			\end{array}%
			\right.
		\end{equation*}
	\end{enumerate}
\end{proposition}

\begin{proof}
	Exactly as previously, it suffices to develop the same arguments for the
	matrices%
	\begin{equation*}
		P_{2n+1}\left( x,T\right) =\left[ 
		\begin{array}{c}
			B^{\ast }e^{f_{0}(2-x)M^{\ast }} \\ 
			\vdots \\ 
			B^{\ast }e^{f_{n-1}(2n-x)M^{\ast }} \\ 
			\mathbbm{1} _{\left( 2n+2-T,1\right) }(x)B^{\ast }e^{f_{n}(2n+2-x)M^{\ast }}%
		\end{array}%
		\right] ;\text{ }Q_{2n+1}\left( x,T\right) =\left[ 
		\begin{array}{c}
			B^{\ast }e^{f_{0}(x)M^{\ast }} \\ 
			\vdots \\ 
			B^{\ast }e^{f_{n-1}(x+2\left( n-1\right) )M^{\ast }} \\ 
			B^{\ast }e^{f_{n}(x+2n)M^{\ast }}%
		\end{array}%
		\right] .~
	\end{equation*}
\end{proof}

\begin{remark}
	\label{Remark explanation}Up to now, we have dealt only with the block
	diagonal system (\ref{Z diag s}), and by Propositions \ref{CNS_2n<T<2n+1}\
	and \ref{CNS_2n+1<T<2n+2}, we have obtained a necessary and sufficient
	conditions that guarantee the exact controllability of System (\ref{Z diag s}%
	). In the next subsection, we will prove a compactness result which will
	allow to extend the obtained results to the complete system (\ref{Z system})
	up to finite dimensional space of target states. Indeed, we have proved that
	there exists a $C_{T}>0$ such that%
	\begin{equation*}
		\left\Vert (p_{0},q_{0})\right\Vert _{H}^{2}\leq C_{T}\int_{0}^{T}\left\vert
		B^{\ast }p_{d}\left( t,0\right) \right\vert ^{2}dt,
	\end{equation*}%
	where $\left( p_{d},q_{d}\right) $ is the solution of the diagonal system (%
	\ref{Z diag s}). Now, by using the triangular inequality we obtain%
	\begin{equation*}
		\left\Vert (p_{0},q_{0})\right\Vert _{H}^{2}\leq C_{T}\int_{0}^{T}\left\vert
		B^{\ast }p\left( t,0\right) \right\vert ^{2}dt+C_{T}\int_{0}^{T}\left\vert
		B^{\ast }\left( p-p_{d}\right) \left( t,0\right) \right\vert ^{2}dt,
	\end{equation*}%
	where $\left( p,q\right) $ is the solution of the complete system (\ref{Z
		system}). So, if we can prove that the map $(p_{0},q_{0})^{t}\mapsto B^{\ast
	}\left( p-p_{d}\right) _{\mid x=0}$ is compact, we can then conclude that
	the complete system (\ref{Z system}) is exactly controllable for any target
	state in the orthogonal in $H$ of the operator $(p_{0},q_{0})^{t}\mapsto
	B^{\ast }p(t,0)$ (see for instance \cite[Lemma 3]{Peetre}). Notice that if
	its kernel is reduced to zero then exact controllability holds for the
	complete system (\ref{Z system}).
\end{remark}

\subsection{Compactness \label{Subsectioncompactness}}

In section \ref{Section weak observability}, we have, after the change of
variable $t\rightleftharpoons T-t,$ considered the System (\ref{Z diag
	system_bis}) that we recall here:%
\begin{equation}
	\left\{ 
	\begin{array}{lll}
		p_{t}+p_{x}-M^{\ast }\eta _{1}p=0, & \mathrm{in} & Q_{T},\text{ \ \ } \\ 
		q_{t}-q_{x}-M^{\ast }\eta _{2}q=0,\text{\ } & \mathrm{in} & Q_{T},\text{ \ \ 
		} \\ 
		(p+q)_{\mid x=0,1}=0_{%
			\mathbb{R}
			^{2}}, & \mathrm{in} & (0,T),\text{ } \\ 
		\left( p,q\right) _{|t=0}=(p_{0},q_{0}),\text{\ \ \ \ } & \mathrm{in} & 
		(0,1).\text{ \ }%
	\end{array}%
	\right.  \label{Z diag}
\end{equation}%
\ The associated whole system obtained by the same change of variable from (%
\ref{Z system}) is: 
\begin{equation}
	\left\{ 
	\begin{array}{lll}
		p_{t}+p_{x}-M^{\ast }\eta _{1}p-M^{\ast }\eta _{2}q=0, & \mathrm{in} & Q_{T},%
		\text{ \ \ } \\ 
		q_{t}-q_{x}-M^{\ast }\eta _{1}p-M^{\ast }\eta _{2}q=0,\text{\ } & \mathrm{in}
		& Q_{T},\text{ \ \ } \\ 
		(p+q)_{\mid x=0,1}=0_{%
			\mathbb{R}
			^{2}}, & \mathrm{in} & (0,T),\text{ } \\ 
		\left( p,q\right) _{|t=0}=(p_{0},q_{0}),\text{\ \ \ \ } & \mathrm{in} & 
		(0,1).\text{\ }%
	\end{array}%
	\right.  \label{Z system complete}
\end{equation}%
In the sequel, we denote by $Z=Z\left( t,\cdot ,s;p_{0},q_{0}\right) =\left(
p,q\right) \left( t,\cdot ,s;p_{0},q_{0}\right) $ the solution to (\ref{Z
	diag}) and by $Z_{d}=Z_{d}\left( t,\cdot ;s,p_{0},q_{0}\right) =\left(
p_{d},q_{d}\right) \left( t,\cdot ,s;p_{0},q_{0}\right) $ the solution to
the diagonal system (\ref{Z system complete}).

This section is devoted to the proof of the compactness of the following
operator:%
\begin{equation*}
	\begin{array}{cccc}
		D_{T}: & H & \rightarrow & L^{2}\left( 0,T\right) \\ 
		& (p_{0},q_{0})^{t} & \mapsto & B^{\ast }\left( p-p_{d}\right) _{\mid x=0}.%
	\end{array}%
\end{equation*}%
In fact, we have:

\begin{theorem}
	\label{theorem compactness}Let $T>0$. Then the operator $D_{T}$ is compact.
\end{theorem}

The proof of Theorem \ref{theorem compactness} will need some preliminaries.
Recall that the solution to System (\ref{Z diag}) can be expressed in terms
of the evolution family $\left( U_{d}(t,s)\right) _{s\leq t}$ as 
\begin{equation*}
	U_{d}(t,s)Z_{0}=(p_{d},q_{d})(t,\cdot ;s,Z_{0}),~0\leq s\leq t,~Z_{0}\in H.
\end{equation*}%
Therefore, there exist two operators $\left( S_{d}^{\pm }(t,s)\right)
_{s\leq t}\in H\rightarrow L^{2}(0,1)^{2}$ such that%
\begin{equation*}
	p_{d}(t,\cdot ;s,Z_{0})=S_{d}^{-}(t,s)Z_{0}(\cdot ),\text{ \ }q_{d}(t,\cdot
	;s,Z_{0})=S_{d}^{+}(t,s)Z_{0}(\cdot ),\text{\ }0\leq s\leq t,~Z_{0}\in H.
\end{equation*}%
Since System (\ref{Z system complete}) is a bounded perturbation of System (%
\ref{Z diag}), then by \cite[Chapter 5, Theorem 2.3]{Pazy}, there exists a
unique evolution family associated with System (\ref{Z system complete})
defined by 
\begin{equation*}
	U(t,s)Z_{0}=(p,q)(t,\cdot ;s,Z_{0}),~0\leq s\leq t,~Z_{0}\in H.
\end{equation*}%
Similarly, there exist two operators $\left( S^{\pm }(t,s)\right) _{s\leq
	t}\in H\rightarrow L^{2}(0,1)^{2}$ such that%
\begin{equation*}
	p(t,\cdot ;s,Z_{0})=S^{-}(t,s)Z_{0},\text{\ \ }q(t,\cdot
	;s,Z_{0})=S^{+}(t,s)Z_{0},\text{\ }0\leq s\leq t,~Z_{0}\in H.
\end{equation*}%
With these new notations, the operator $D_{T}$ takes the form%
\begin{equation*}
	D_{T}Z_{0}=\mathcal{C}\left( S^{-}(t,s)-S_{d}^{-}(t,s)\right) Z_{0},\text{ }%
	0\leq s\leq t,~Z_{0}\in H,
\end{equation*}%
where $\mathcal{C}$ is the operator%
\begin{equation*}
	\begin{array}{cccc}
		\mathcal{C}: & C(0,T;L^{2}(0,1)^{2}) & \rightarrow & L^{2}\left( 0,T\right)
		\\ 
		& v & \mapsto & B^{\ast }v_{\mid x=0}.%
	\end{array}%
\end{equation*}

Since $\left( S(t,s)\right) _{s\leq t}$ \ is a perturbation of $\left(
S_{d}(t,s)\right) _{s\leq t}$ \ by a multiplication operators with
multiplier $\mathcal{P}_{T}(s)=\mathcal{P}_{T}(T-s),$ it is clear that the
two evolutions families are linked by the Duhamel formula: 
\begin{equation*}
	U(t,0)Z_{0}=U_{d}(t,0)Z_{0}+\int_{0}^{t}U_{d}(t,s)\mathcal{P}%
	_{T}(s)U(s,0)Z_{0}ds,~Z_{0}\in H.
\end{equation*}%
Therefore, 
\begin{eqnarray*}
	\left( U(t,0)-U_{d}(t,0)\right) Z_{0} &=&\int_{0}^{t}U_{d}(t,s)\mathcal{P}%
	_{T}(s)\left( U(s,0)-U_{d}(s,0)\right) Z_{0}ds \\
	&&+\int_{0}^{t}U_{d}(t,s)\mathcal{P}_{T}(s)U_{d}(s,0)Z_{0}ds.
\end{eqnarray*}%
Consequently,%
\begin{eqnarray*}
	\left( S^{-}(t,s)-S_{d}^{-}(t,s)\right) Z_{0} &=&\int_{0}^{t}S_{d}^{-}(t,s)%
	\mathcal{P}_{T}(s)\left( U(s,0)-U_{d}(s,0)\right) Z_{0}ds \\
	&&+\int_{0}^{t}S_{d}^{-}(t,s)\mathcal{P}_{T}(s)U_{d}(s,0)Z_{0}ds \\
	&=&\Psi _{1}(t,\cdot ,Z_{0})+\Psi _{2}(t,\cdot ,Z_{0}).
\end{eqnarray*}%
Thus,%
\begin{equation*}
	D_{T}(p_{0},q_{0})^{t}=\mathcal{C}\Psi _{1}(t,\cdot ,Z_{0})+\mathcal{C}\Psi
	_{2}(t,\cdot ,Z_{0}),~Z_{0}\in H.
\end{equation*}%
So, proving that $D_{T}$ is compact amounts to prove the compactness of the
operators $\mathcal{C}\Psi _{i}(t,\cdot ,Z_{0}),$ $i=1,2.$

Since $\left( U_{d}(t,0)\right) _{0\leq t\leq T}$ is completely known, the
compactness will be just a consequence of the explicit formula of the
operator $\mathcal{C}\Psi _{2}(t,\cdot ,Z_{0})$. To deal with $\mathcal{C}%
\Psi _{1}(t,\cdot ,Z_{0})$ we use the following lemma inspired from (\cite%
{Duprez}) which has been used also in (\cite{Olive}) in the same context to
deal with more general autonomous hyperbolic systems.

\begin{lemma}
	\label{Lemma admissible}For any $f\in C([0,T],H),$ there exists $C_{T}>0$
	such that 
	\begin{equation}
		\int_{0}^{T}\left\vert \mathcal{C}\int_{0}^{t}S_{d}^{-}(t,s)f(s)ds\right%
		\vert ^{2}dt\leq C_{T}\left\Vert f\right\Vert _{L^{2}(0,T;H)}^{2}.
		\label{admissible}
	\end{equation}
\end{lemma}

\begin{proof}
	For the time being, let $f=\left( f_{1},f_{2}\right) \in C([0,T]\times D(%
	\mathcal{A}^{\ast }))$ ($D(\mathcal{A}^{\ast })$ is defined in (\ref{D(A)}))
	which is a dense subspace of $C([0,T],H).$ By using the characteristics
	method, we have by (\ref{p(t,0;s) 2n,2n+1}) and (\ref{p(t,0;s) 2n+1,2n+2})
	for any $n\geq 0:$%
	\begin{equation}
		\left( S_{d}^{-}(t,s)f(s)\right) \left( x\right) =\left\{ 
		\begin{array}{cc}
			-R_{n}(t,x;s)f_{2}(s,t-x-s-2n),\text{ \ \ \ \ \ \ } & \mathrm{if}~~t-x-s\in
			\lbrack 2n,2n+1),\text{ \ \ \ \ } \\ 
			&  \\ 
			R_{n}(t,x;s)f_{1}(s,2n+2-t+x+s), & \mathrm{if}~~t-x-s\in \lbrack 2n+1,2n+2),%
		\end{array}%
		\right.   \label{S^--}
	\end{equation}%
	where 
	\begin{equation*}
		R_{n}(t,x;s)=e^{f_{n}(t-x,s)M^{\ast }+\int_{t-x}^{t}\eta _{1}(\tau ,\tau
			-t+x)d\tau M^{\ast }},\text{ }n\geq 0.
	\end{equation*}%
	In particular, if $t-x\in (0,1):$ 
	\begin{equation*}
		\left( S_{d}^{-}(t,s)f(s)\right) \left( x\right) =\left\{ 
		\begin{array}{cc}
			-e^{\int_{t-x}^{t}\eta _{1}(\tau ,\tau -t+x)d\tau +\int_{s}^{t-x}\eta
				_{2}(\tau ,t-x-\tau )d\tau }f_{2}(s,t-x-s),\text{ \ } & \mathrm{if}~~s\in
			\lbrack 0,t-x), \\ 
			&  \\ 
			e^{\int_{s}^{t}\eta _{1}(\tau ,\tau -t+x)d\tau }f_{1}(s,s-t+x),\text{ \ \ \
				\ \ \ \ \ \ \ \ \ \ \ \ \ \ \ \ \ \ \ \ \ } & \mathrm{if}~~s\in \lbrack
			t-x,t).%
		\end{array}%
		\right. 
	\end{equation*}%
	Consider first the case $t-x\in \lbrack 0,1).$ Since $f$ is a continuous
	function, the trace operator makes sense and thus we obtain 
	\begin{eqnarray*}
		&&\mathcal{C}\int_{0}^{t}S_{d}^{-}(t,s)f(s)ds \\
		&=&B^{\ast }\left( \int_{0}^{t}S_{d}^{-}(t,s)f(s)ds\right) \left( 0\right) 
		\\
		&=&-B^{\ast }e^{\int_{s}^{t}\eta _{2}(\tau ,t-\tau )d\tau }f_{2}(s,t-s)ds.
	\end{eqnarray*}%
	Therefore%
	\begin{equation*}
		\left\vert B^{\ast }\left( \int_{0}^{t}S_{d}^{-}(t,s)f(s)ds\right) \left(
		0\right) \right\vert ^{2}\leq C\int_{0}^{t}\left\vert
		f_{2}(t-s,s)\right\vert ^{2}ds.
	\end{equation*}%
	By integrating over $(0,T)$ the above inequality we obtain for any $T\leq 1:$%
	\begin{eqnarray}
		\int_{0}^{T}\left\vert \mathcal{C}\int_{0}^{t}S_{d}^{-}(t,s)f(s)ds\right%
		\vert ^{2}dt &\leq &C\int_{0}^{T}\int_{0}^{t}\left\vert
		f_{2}(t-s,s)\right\vert ^{2}dsdt  \label{C} \\
		&\leq &C\int_{0}^{T}\int_{0}^{t}\left\vert f_{2}(t-s,s)\right\vert
		^{2}+\left\vert f_{1}(t-s,s)\right\vert ^{2}dsdt  \notag \\
		&\leq &C\int_{0}^{T}\int_{0}^{1}\left\vert f_{2}(t,s)\right\vert
		^{2}+\left\vert f_{1}(t,s)\right\vert ^{2}dsdt  \notag \\
		&=&C\left\Vert f\right\Vert _{L^{2}(0,T;H)}^{2}.  \notag
	\end{eqnarray}%
	Now, we deal with the case $T\geq 1.$ Let $t-x\in \lbrack 2n-1,2n),$ $n\geq
	1.$ We write 
	\begin{eqnarray*}
		&&\left( \int_{0}^{t}S_{d}^{-}(t,s)f(s)ds\right) \left( x\right) \text{ } \\
		&=&\left(
		\int_{0}^{t-x-2n+1}+\sum_{k=1}^{n-1}\int_{t-x-2k}^{t-x-2k+1}+%
		\sum_{k=0}^{n-1}\int_{t-x-2k-1}^{t-x-2k}+\int_{t-x}^{t}\right) \left(
		S_{d}^{-}(t,s)f(s)ds\right) \left( x\right)  \\
		&=&\int_{t-x}^{t}e^{\int_{s}^{t}\eta _{1}(\tau ,t-x-\tau )d\tau M^{\ast
		}}f_{1}(s,s-t+x)ds \\
		&&+\sum_{k=1}^{n-1}\int_{t-x-2k}^{t-x-2k+1}R_{k-1}(t,x;s)f_{1}(s,2k-t+x+s)ds
		\\
		&&-\sum_{k=0}^{n-1}\int_{t-x-2k-1}^{t-x-2k}R_{k}(t,x;s)f_{2}(s,t-x-s-2k)ds \\
		&&+\int_{0}^{t-x-2n+1}R_{n-1}(t,x;s)f_{1}(s,2n-t+x+s)ds.
	\end{eqnarray*}%
	A simple variable substitution yields%
	\begin{eqnarray*}
		&&\left( \int_{0}^{t}S_{d}^{-}(t,s)f(s)ds\right) \left( x\right)  \\
		&=&\int_{0}^{x}e^{\int_{s+t-x}^{t}\eta _{1}(\tau ,t-x-\tau )d\tau M^{\ast
		}}f_{1}(s+t-x,s)ds \\
		&&+\sum_{k=1}^{n-1}\int_{0}^{1}R_{k-1}(t,x;s+t-x-2k)f_{1}(s+t-x-2k,sds \\
		&&-\sum_{k=0}^{n-1}\int_{0}^{1}R_{k}(t,x;t-x-s-2k)f_{2}(t-x-s-2k,s)ds \\
		&&+\int_{2n-t+x}^{1}R_{n-1}(t,x;s+t-x-2n)f_{1}(s+t-x-2n,s)ds.
	\end{eqnarray*}%
	Again, since $f$ is a continuous function, the trace operator makes sense
	and we obtain: 
	\begin{eqnarray*}
		&&\mathcal{C}\left( \int_{0}^{t}S_{d}^{-}(t,s)f(s)ds\right)  \\
		&=&B^{\ast }\left( \int_{0}^{t}S_{d}^{-}(t,s)f(s)ds\right) \left( 0\right) 
		\\
		&=&\sum_{k=1}^{n-1}\int_{0}^{1}B^{\ast }R_{k-1}(t,0;s+t-2k)f_{1}(s+t-2k,s)ds
		\\
		&&-\sum_{k=0}^{n-1}\int_{0}^{1}B^{\ast }R_{k}(t,0;t-s-2k)f_{2}(t-s-2k,s)ds \\
		&&-\int_{2n-t}^{1}B^{\ast }R_{n-1}(t,0;s+t-2n)f_{1}(s+t-2n,s)ds.
	\end{eqnarray*}%
	Since $\eta _{1},\eta _{2}$ are bounded, we get, using Cauchy-Shwarz
	inequality: 
	\begin{eqnarray}
		\left\vert \mathcal{C}\int_{0}^{t}S_{d}^{-}(t,s)f(s)ds\right\vert ^{2} &\leq
		&C_{n,t}\sum_{k=1}^{n-1}\int_{0}^{1}\left\vert f_{1}(s+t-2k,s)\right\vert
		^{2}ds  \label{AA} \\
		&&+C_{n,t}\sum_{k=0}^{n-1}\int_{0}^{1}\left\vert f_{2}(t-s-2k,s)\right\vert
		^{2}ds  \notag \\
		&&+C_{n,t}\int_{2n-t}^{1}\left\vert f_{1}(s+t-2n,s)\right\vert ^{2}ds, 
		\notag
	\end{eqnarray}%
	where $C_{n,t}$ is a positive constant depending on $t$ and $n.$ Taking the
	integral of (\ref{AA}) over $(0,T)$ for $T\in \lbrack 2n-1,2n),$ $n\geq 1,$
	and using the fact%
	\begin{eqnarray*}
		\int_{0}^{T}\int_{0}^{1}\left\vert f_{1}(s+t-2k,s)\right\vert ^{2}dsdt &\leq
		&\int_{0}^{T}\int_{0}^{1}\left\vert f_{1}(t,s)\right\vert ^{2}dsdt,\text{ }%
		\forall k\in \{1,...,n-1\}, \\
		\int_{0}^{T}\int_{0}^{1}\left\vert f_{2}(t-s-2k,s)\right\vert ^{2}dsdt &\leq
		&\int_{0}^{T}\int_{0}^{1}\left\vert f_{2}(t,s)\right\vert ^{2}dsdt,\text{ }%
		\forall k\in \{0,...,n-1\},
	\end{eqnarray*}%
	yields%
	\begin{equation}
		\int_{0}^{T}\left\vert \int_{0}^{t}S_{d}^{-}(t,s)f(s)ds\right\vert
		^{2}dt\leq n\left\Vert f\right\Vert _{L^{2}(0,T;H)}^{2}.  \label{A}
	\end{equation}%
	Similarly, we obtain the following estimate for $T\in (2n,2n+1),$ $n\geq 1$ 
	\begin{equation}
		\int_{0}^{T}\left\vert \int_{0}^{t}S_{d}^{-}(t,s)f(s)ds\right\vert
		^{2}dt\leq \left( n+1\right) \left\Vert f\right\Vert _{L^{2}(0,T;H)}^{2},
		\label{B}
	\end{equation}%
	which ends the proof. The estimates (\ref{C}), (\ref{A}) and (\ref{B}) can
	be extended for any $f\in C(0,T;H)$ by using a standard density argument.
\end{proof}

\begin{proposition}
	\label{Proposition admissible}The operator $Z_{0}\mapsto \mathcal{C}\Psi
	_{1}(t,\cdot ,Z_{0})$ acting from $H$ to $L^{2}(0,T)$ is compact.
\end{proposition}

\begin{proof}
	The proof of Proposition \ref{Proposition admissible} is a direct
	consequence of the result proved in \cite{Neves} (and an extension of this
	result in \cite{Ammar khodja-Bader} to the case where some wave speeds are
	equal) which asserts that the difference of the evolution operators defined
	by systems (\ref{Z system complete}) and (\ref{Z diag}) are compact) and
	Lemma \ref{Lemma admissible}. More precisely, by letting%
	\begin{equation*}
		f(s)=\mathcal{P}_{T}(s)\left( U(s,0)-U_{d}(s,0)\right) Z_{0},\text{ }%
		Z_{0}\in H,\text{ }s\leq t,
	\end{equation*}%
	in (\ref{admissible}), we obtain 
	\begin{eqnarray*}
		\left\Vert \mathcal{C}\Psi _{1}(t,\cdot ,Z_{0})\right\Vert _{L^{2}(0,T)}
		&=&\int_{0}^{T}\mathcal{C}\left\vert \int_{0}^{t}S_{d}^{-}(t,s)\mathcal{P}%
		_{T}(s)\left( U(s,0)-U_{d}(s,0)\right) Z_{0}ds\right\vert ^{2}dt \\
		&\leq &C_{T}\left\Vert \left( U(\cdot ,0)-U_{d}(\cdot ,0)\right)
		Z_{0}\right\Vert _{L^{2}(0,T;H),}^{2}
	\end{eqnarray*}%
	which is a compact operator by the result in \cite{Ammar khodja-Bader}. It
	remains to deal with the operator $\mathcal{C}\Psi _{2}(t,\cdot ,Z_{0}).$
\end{proof}

\begin{proposition}
	\label{Proposition admissible 2}The operator $Z_{0}\mapsto \mathcal{C}\Psi
	_{2}(t,\cdot ,Z_{0})$ acting from $H$ to $L^{2}\left( 0,T\right) $ is
	compact.
\end{proposition}

\begin{proof}
	The proof is purely constructive. First, we find the solution $q(t,\cdot
	;s,Z_{0})=S^{+}(t,s)Z_{0}.$
	
	($D(\mathcal{A}^{\ast })$ is defined in (\ref{D(A)}))
	
	Let $(p_{0},q_{0})\in D(\mathcal{A}^{\ast })$ (see (\ref{D(A)})). By using the
	characteristics method, we compute (\ref{S^--}) for $x=1$ by (\ref{q(t,1) 1}
	and (\ref{q(t,1)2}) we obtain:%
	\begin{equation}
		\left( S_{d}^{+}(t,s)Z_{0}\right) (x)=\left\{ 
		\begin{array}{cc}
			N_{n}(t,x;s)f_{2}(s,x+t-s-2n),\text{ \ \ } & \mathrm{if}~~x+t-s\in \lbrack
			2n,2n+1),\text{ \ \ \ \ } \\ 
			&  \\ 
			-N_{n}(t,x;s)f_{1}(s,2n+2-x-t+s), & \mathrm{if}~~x+t-s\in \lbrack 2n+1,2n+2),%
		\end{array}%
		\right.  \label{S^++}
	\end{equation}%
	where%
	\begin{equation*}
		N_{n}(t,x;s)=R_{n-1}(x+t-1,1;s)+e^{\int_{x+t-1}^{t}\eta _{2}(\tau ,-\tau
			+t+x)d\tau },\text{ }n\geq 0.
	\end{equation*}%
	The aim now is to compute $\Psi _{2}(t,\cdot ,Z_{0})$ explicitly. We recall
	that 
	\begin{equation*}
		\Psi _{2}(t,\cdot ,Z_{0})=\int_{0}^{t}S_{d}^{-}(t,s)\mathcal{P}%
		_{T}(s)U_{d}(s,0)Z_{0}ds,\text{ }t\leq T,
	\end{equation*}%
	where $U_{d}(s,0)Z_{0}=\left( S_{d}^{-}(s,0)Z_{0},S_{d}^{+}(s,0)Z_{0}\right)
	^{t}$, $0\leq s\leq t,$ and $\left( S_{d}^{\pm }(t,s)\right) _{s\leq t\leq
		T} $ are given in (\ref{S^--}) and (\ref{S^++}).
	
	Applying $\mathcal{P}_{T}(\cdot )$ yields%
	\begin{equation*}
		\mathcal{P}_{T}(s)U_{d}(s,0)=\left( \eta _{2}(s)S_{d}^{+}(s,0)Z_{0},\eta
		_{1}(s)S_{d}^{-}(s,0)Z_{0}\right) ^{t},\text{ }s\leq t.
	\end{equation*}%
	Therefore, 
	\begin{equation}
		\mathcal{C}\Psi _{2}(t,\cdot ,Z_{0})=\mathcal{C}\int_{0}^{t}S_{d}^{-}(t,s)%
		\left( \eta _{2}(s)S_{d}^{+}(s,0),\eta _{1}(s)S_{d}^{-}(s,0)\right)
		^{t}Z_{0}ds,\text{ }t\leq T.  \label{integrand}
	\end{equation}%
	Let us start by computing the integrand in (\ref{integrand}). We have for
	any $n\geq 0$ 
	\begin{eqnarray*}
		&&S_{d}^{-}(t,s)\left( \eta _{2}(s)S_{d}^{+}(s,0)Z_{0},\eta
		_{1}(s)S_{d}^{-}(s,0)Z_{0}\right) ^{t}\left( x\right) \\
		&=&\left\{ 
		\begin{array}{cc}
			M_{n}^{1}(t,x;s)\left( S_{d}^{-}(s,0)Z_{0}\right) \left( t-x-s-2n\right) ,%
			\text{\ \ \ \ \ } & \mathrm{if}~~t-x-s\in \lbrack 2n,2n+1),\text{ \ \ \ \ }
			\\ 
			&  \\ 
			M_{n}^{2}(t,x;s)\left( S_{d}^{+}(s,0)Z_{0}\right) (2n+2-t+x+s), & \mathrm{if}%
			~~t-x-s\in \lbrack 2n+1,2n+2),%
		\end{array}%
		\right.
	\end{eqnarray*}%
	where 
	\begin{eqnarray*}
		M_{n}^{1}(t,x;s) &=&-R_{n}(t,x;s)\eta _{1}(s,t-x-s-2n),\text{ \ \ \ }%
		t-x-s\in \lbrack 2n,2n+1),\text{ }n\geq 0, \\
		M_{n}^{2}(t,x;s) &=&R_{n}(t,x;s)\eta _{2}(s,2n+2-t+x+s),\text{ }t-x-s\in
		\lbrack 2n+1,2n+2),\text{ }n\geq 0.
	\end{eqnarray*}%
	Now, by using (\ref{S^--}) and (\ref{S^++}) for $s=0$ we get for any $%
	n,k\geq 0$%
	\begin{eqnarray}
		&&\left( S_{d}^{-}(\tau ,0)Z_{0}\right) \left( t-x-\tau -2n\right)
		\label{SS^-} \\
		&=&\left\{ 
		\begin{array}{cc}
			\begin{array}{c}
				-R_{k}(\tau ,t-x-\tau -2n;0)\times \\ 
				q_{0}(2\tau -t+x+2(n-k)),\text{ \ \ \ \ \ \ \ }%
			\end{array}
			& \mathrm{if}~\tau \in (\frac{2(k-n)+t-x}{2},\frac{2(k-n)+t-x+1}{2})\cap
			(0,t),\text{ \ } \\ 
			&  \\ 
			\begin{array}{c}
				R_{k}(\tau ,t-x-\tau -2n;0)\times \\ 
				p_{0}(2(k-n+1)-2\tau +t-x),%
			\end{array}
			& \mathrm{if}\text{ }\tau \in (\frac{2(k-n)+t-x+1}{2},\frac{2(k-n)+t-x+2}{2}%
			)\cap (0,t),%
		\end{array}%
		\right.  \notag
	\end{eqnarray}%
	and%
	\begin{eqnarray}
		&&S^{+}(t,0)Z_{0}(2n+2-t+x+\tau )  \label{SS+} \\
		&=&\left\{ 
		\begin{array}{cc}
			\begin{array}{c}
				N_{k}(\tau ,2n+2-t+x+\tau ;0)\times \\ 
				q_{0}(2(n-k+1)-t+x+2\tau ),%
			\end{array}
			& \mathrm{if}~\tau \in ~(\frac{2(k-n-1)-t-x}{2},\frac{2(k-n)-t-x-1}{2})\cap
			(0,t), \\ 
			&  \\ 
			\begin{array}{c}
				-N_{k}(\tau ,2n+2-t+x+\tau ;0)\times \\ 
				p_{0}(2(k-n)+t-x-2\tau ),\text{ \ \ \ \ \ }%
			\end{array}
			& \mathrm{if}\text{ }\tau \in ~(\frac{2(k-n)-t-x-1}{2},\frac{2(k-n)-t-x}{2}%
			)\cap (0,t),\text{ \ }%
		\end{array}%
		\right.  \notag
	\end{eqnarray}%
	Therefore, we obtain for any $k,n\geq 0$%
	\begin{eqnarray*}
		\Psi _{2}(t,\cdot ,Z_{0}) &=&-\sum_{k,n\geq 0}\int_{\frac{2(k-n)+t-x}{2}}^{%
			\frac{2(k-n)+t-x+1}{2}}\ \mathbbm{1} (\tau )_{(0,t)}P_{k,n}^{1}(t,x;\tau
		)q_{0}(2\tau -t+x+2(n-k))d\tau \\
		&&+\sum_{k,n\geq 0}\int_{\frac{2(k-n)+t-x+1}{2}}^{\frac{2(k-n)+t-x+2}{2}}%
		\mathbbm{1} (\tau )_{(0,t)}P_{k,n}^{1}(t,x;\tau )p_{0}(2(k-n+1)-2\tau
		+t-x)d\tau \\
		&&+\sum_{k,n\geq 0}\int_{\frac{2(k-n-1)-t-x}{2}}^{\frac{2(k-n)-t-x-1}{2}}%
		\mathbbm{1} (\tau )_{(0,t)}P_{k,n}^{2}(t,x;\tau )q_{0}(2(n-k+1)-t+x+2\tau
		)d\tau \\
		&&-\sum_{k,n\geq 0}\int_{\frac{2(k-n-1)-t-x}{2}}^{\frac{2(k-n)-t-x-1}{2}}%
		\mathbbm{1} (\tau )_{(0,t)}P_{k,n}^{2}(t,x;\tau )p_{0}(2(k-n)+t-x-2\tau
		)d\tau ,
	\end{eqnarray*}%
	where%
	\begin{eqnarray*}
		P_{k,n}^{1}(t,x;\tau ) &=&M_{n}^{1}(t,x;\tau )R_{k}(s,t-x-s-2n;0),\text{ } \\
		P_{k,n}^{1}(t,x;\tau ) &=&M_{n}^{2}(t,x;\tau )N_{k}(s,2n+2-t+x+\tau ;0).
	\end{eqnarray*}%
	Consequently%
	\begin{eqnarray*}
		&&\mathcal{C}\Psi _{2}(t,\cdot ,Z_{0}) \\
		&=&-B^{\ast }\sum_{k,n\geq 0}\int_{\frac{2(k-n)+t}{2}}^{\frac{2(k-n)+t+1}{2}%
		}\ \mathbbm{1} (\tau )_{(0,t)}P_{k,n}^{1}(t,0;\tau )q_{0}(2\tau
		-t+2(n-k))d\tau \\
		&&+B^{\ast }\sum_{k,n\geq 0}\int_{\frac{2(k-n)+t+1}{2}}^{\frac{2(k-n)+t+2}{2}%
		}\mathbbm{1} (\tau )_{(0,t)}P_{k,n}^{1}(t,0;\tau )p_{0}(2(k-n+1)-2\tau
		+t)d\tau \\
		&&+B^{\ast }\sum_{k,n\geq 0}\int_{\frac{2(k-n-1)-t}{2}}^{\frac{2(k-n)-t-1}{2}%
		}\mathbbm{1} (\tau )_{(0,t)}P_{k,n}^{2}(t,0;\tau )q_{0}(2(n-k+1)-t+2\tau
		)d\tau \\
		&&-B^{\ast }\sum_{k,n\geq 0}\int_{\frac{2(k-n-1)-t}{2}}^{\frac{2(k-n)-t-1}{2}%
		}\mathbbm{1} (\tau )_{(0,t)}P_{k,n}^{2}(t,0;\tau )p_{0}(2(k-n)+t-2\tau
		)d\tau .
	\end{eqnarray*}%
	The proof follows by \cite[Lemma 4]{Neves} which allows to conclude that $%
	\mathcal{C}\Psi _{2}(t,\cdot ,Z_{0})$ is a compact operator from $H$ to $%
	L^{2}(0,T)$ since it is a finite sum of such operators.
\end{proof}

\section{Unique continuation \label{Section unique continuation}}

In this section, we deal with the \emph{unique continuation} property for
System (\ref{Wave 1}). We give a necessary and sufficient condition for the
constant case, and a semi-explicit condition in the autonomous case. When
the coefficients depend on time, we give also a necessary and sufficient
condition for the cascade case with providing some non-trivial examples at
the end.

\subsection{The constant case}

Here, we assume that $a,b\in 
\mathbb{R}
$. Recall that the adjoint system of System (\ref{Wave 1}) is given by%
\begin{equation}
	\left\{ 
	\begin{array}{lll}
		\varphi _{tt}=\varphi _{xx}-M^{\ast }(a\varphi _{t}+b\varphi _{x}), & 
		\mathrm{in} & (0,T)\times (0,1), \\ 
		\varphi _{\mid x=0,1}=0,\text{ \ \ \ \ \ \ \ \ \ } & \mathrm{in} & (0,T),%
		\text{ \ \ \ \ \ \ \ \ \ \ \ \ \ \ \ \ \ \ \ \ } \\ 
		\left( \varphi ,\varphi _{t}\right) _{\mid t=T}=\left( \varphi^T
		_{0},\varphi^T _{1}\right) ,\text{ \ \ } & \mathrm{in} & (0,1).\text{ \ \ \
			\ \ \ \ \ \ \ \ \ \ \ \ \ \ \ \ \ \ }%
	\end{array}%
	\right.  \label{adjoint2}
\end{equation}%
The correspond \emph{unique continuation property }reads:%
\begin{equation}
	B^{\ast }\varphi _{x}(t,0)=0,\text{ }\forall t\in (0,T)\Rightarrow \left(
	\varphi _{0},\varphi _{1}\right) =\left( 0,0\right) ,\text{ }\forall \left(
	\varphi^T _{0},\varphi^T _{1}\right) \in H_{0}^{1}(0,1)^{2}\times
	L^{2}(0,1)^{2}.  \label{uniqqq}
\end{equation}

For the sake of simplicity, we set $e^{\frac{bx}{2}M^{\ast }}\widetilde{%
	\varphi }(x)=\varphi (x).$ Then $\widetilde{\varphi }$ is the solution of
the following system 
\begin{equation*}
	\left\{ 
	\begin{array}{lll}
		\widetilde{\varphi }_{tt}=\widetilde{\varphi }_{xx}-\frac{1}{4}b^{2}\left(
		M^{\ast }\right) ^{2}\widetilde{\varphi }-aM^{\ast }\widetilde{\varphi }_{t},
		& \mathrm{in} & (0,T)\times (0,1), \\ 
		\widetilde{\varphi }_{\mid x=0,1}=0,\text{ \ \ \ \ \ \ \ \ \ } & \mathrm{in}
		& (0,T),\text{ \ \ \ \ \ \ \ \ \ \ \ \ \ \ \ \ \ \ \ \ } \\ 
		\left( \widetilde{\varphi },\widetilde{\varphi }_{t}\right) _{\mid
			t=T}=\left( \widetilde{\varphi }^T_{0},\widetilde{\varphi}^T_{1}\right) ,%
		\text{ \ \ } & \mathrm{in} & (0,1).\text{ \ \ \ \ \ \ \ \ \ \ \ \ \ \ \ \ \
			\ \ \ \ }%
	\end{array}%
	\right.
\end{equation*}

By using a standard spectral decomposition of the solution to System (\ref%
{adjoint2}) (See for instance \cite{Avdonin 1,Bennour} in the hyperbolic
context or \cite{Duprez 2,Ammar Khodja 1} in the parabolic one), we can see
that proving that (\ref{uniqqq}) holds in a time $T\geq 4$ amounts to
proving that all the eigenvalues of the corresponding spectral problem%
\begin{equation}
	\left\{ 
	\begin{array}{l}
		\lambda ^{2}\widetilde{\varphi }(x)=\widetilde{\varphi }^{\prime \prime
		}(x)-\left( \frac{1}{4}b^{2}\left( M^{\ast }\right) ^{2}+aM^{\ast }\lambda
		\right) \widetilde{\varphi }(x),\text{ }x\in (0,1), \\ 
		\widetilde{\varphi }(0)=\widetilde{\varphi }(1)=0,\text{ }\widetilde{\varphi 
		}=\left( \widetilde{\varphi }_{1},\widetilde{\varphi }_{2}\right) ,%
	\end{array}%
	\right.  \label{Sturm}
\end{equation}%
are simple. First, let us assume first that $M^{\ast }$ is diagonalizable.

\begin{proposition}
	Assume that $M^{\ast }$ has $2$ distinct eigenvalues $\mu _{1},\mu _{2}.$
	Then, all the eigenvalues of the Sturm-Liouville problem (\ref{Sturm}) are
	simple if, and only if 
	\begin{equation}
		a^{2}\left( \mu _{1}-\mu _{2}\right) \left( \mu _{2}\xi _{n_{1}}^{2}-\mu
		_{1}\xi _{n_{2}}^{2}\right) \neq \left( \xi _{n_{1}}^{2}-\xi
		_{n_{2}}^{2}\right) ^{2},\text{ }\forall n_{1},n_{2}\in 
		\mathbb{Z}
		,  \label{distinct}
	\end{equation}%
	where $\xi _{n_{i}}=\frac{1}{4}b^{2}\mu _{i}^{2}+\left( n_{i}\pi \right)
	^{2},$ $i=1,2.$
\end{proposition}

\begin{proof}
	Since $M^{\ast }$ is diagonalizable, there exists a diagonal matrix $%
	D=diag(\mu _{1},\mu _{2})$ and $2\times 2$ invertible matrix $P$ such that $%
	M^{\ast }=PDP^{-1}.$ Letting $z=P^{-1}\widetilde{\varphi }$ with $z=\left(
	z_{1},z_{2}\right) $ in (\ref{Sturm}) yields the following Sturm-Liouville
	problem%
	\begin{equation*}
		\left\{ 
		\begin{array}{l}
			z_{i}^{\prime \prime }(x)=\left( \lambda ^{2}+a\lambda \mu _{i}+\frac{1}{4}%
			b^{2}\mu _{i}^{2}\right) z_{i}(x),\text{ }x\in (0,1),\text{ } \\ 
			z_{i}(0)=z_{i}(1)=0,\text{ \ }i=1,2.%
		\end{array}%
		\right.
	\end{equation*}
	In order to prove that the eigenvalues of the above problem are simple we
	have to check that the following polynomial equations%
	\begin{equation*}
		\lambda ^{2}+\lambda a\mu _{i}+\frac{1}{4}b^{2}\mu _{i}^{2}+\left( n_{i}\pi
		\right) ^{2}=0,\text{ }n_{i}\in 
		\mathbb{Z}
		,\text{ }i=1,2,
	\end{equation*}%
	don't have a common roots which is equivalent to check that the following
	Sylvester matrix is invertible 
	\begin{equation*}
		S_{k,n}=\left( 
		\begin{array}{cccc}
			1 & a\mu _{1} & \frac{1}{4}b^{2}\mu _{1}^{2}+\left( n_{1}\pi \right) ^{2} & 0
			\\ 
			0 & 1 & a\mu _{1} & \frac{1}{4}b^{2}\mu _{1}^{2}+\left( n_{1}\pi \right) ^{2}
			\\ 
			1 & a\mu _{2} & \frac{1}{4}b^{2}\mu _{2}^{2}+\left( n_{2}\pi \right) ^{2} & 0
			\\ 
			0 & 1 & a\mu _{2} & \frac{1}{4}b^{2}\mu _{2}^{2}+\left( n_{2}\pi \right) ^{2}%
		\end{array}%
		\right) ,\text{ }n_{1},n_{2}\in 
		\mathbb{Z}
		,
	\end{equation*}%
	which is the case if, and only if (\ref{distinct}) is satisfied.
\end{proof}

\begin{remark}
	In particular, if $\mu _{1}=-\mu _{2}=\mu ,$ assumption (\ref{distinct})
	becomes%
	\begin{equation*}
		-2a^{2}\mu ^{2}\left( \xi _{n_{1}}^{2}+\xi _{n_{2}}^{2}\right) \neq \left(
		\xi _{n_{1}}^{2}-\xi _{n_{2}}^{2}\right) ^{2},\text{ }\forall n_{1},n_{2}\in 
		\mathbb{Z}
		,
	\end{equation*}%
	which might occur only if, and only if $\mu \in i%
	\mathbb{R}
	.$
\end{remark}

Now, we consider the case where $M^{\ast }$ is not diagonalizable$.$

\begin{proposition}
	Assume that $M^{\ast }$ is not diagonalizable and let $\mu $ be its
	eigenvalue. Then, all the eigenvalues of the Sturm-Liouville problem (\ref%
	{Sturm}) are simple if, and only if 
	\begin{equation}
		\frac{1}{2}b^{2}\mu +a\neq 0.  \label{jordan 11}
	\end{equation}
\end{proposition}

\begin{proof}
	In this case, we write $M^{\ast }$ in the Jordan form: there exists a matrix 
	$J$ and a $2\times 2$ invertible matrix $P$ such that $M^{\ast }=PJP^{-1},$
	where%
	\begin{equation*}
		J=\left( 
		\begin{array}{ll}
			\mu & \text{ \ }1 \\ 
			0 & \text{ \ }\mu%
		\end{array}%
		\right) .
	\end{equation*}%
	Letting $z=P^{-1}\widetilde{\varphi }$ in (\ref{Sturm}) yields the following
	coupled Sturm-Liouville problem%
	\begin{equation*}
		\left\{ 
		\begin{array}{l}
			z_{1}^{\prime \prime }(x)=\left( \lambda ^{2}+a\mu \lambda +\frac{1}{4}%
			b^{2}\mu ^{2}\right) z_{1}(x)+\left( \frac{1}{2}b^{2}\mu +a\right) z_{2}, \\ 
			z_{2}^{\prime \prime }(x)=\left( \lambda ^{2}+a\mu \lambda +\frac{1}{4}%
			b^{2}\mu ^{2}\right) z_{2}(x), \\ 
			y(0)=y(1)=0, \\ 
			z(0)=z(1)=0.%
		\end{array}%
		\right.
	\end{equation*}%
	By the same reasoning as in \cite[Proposition 2.1]{Ammar Khodja 1,Duprez 2},
	it can be seen that the above system has non-trivial solution if, and only
	if $\frac{1}{2}b^{2}\mu +a\neq 0$ with $\lambda $ fulfills the following
	second order polynomial equation for some $n\in 
	\mathbb{Z}
	:$ 
	\begin{equation*}
		\lambda ^{2}+a\mu \lambda +\frac{1}{4}b^{2}\mu ^{2}+\left( n\pi \right)
		^{2}=0.
	\end{equation*}%
	It is clear that the above equation has simple roots for any $n\in 
	\mathbb{Z}
	.$ This finishes the proof.
\end{proof}

\begin{remark}
	As it is shown, if we let $a=0$, the coupling parameter $b$ affects
	controllability in low frequency. Actually, conditions (\ref{distinct}) and (%
	\ref{jordan 11}) become respectively 
	\begin{equation*}
		b^{2}\neq 2\pi ^{2}\frac{n_{2}^{2}-n_{1}^{2}}{\mu _{1}^{2}-\mu _{2}^{2}},%
		\text{ }\forall n_{1},n_{2}\in 
		\mathbb{Z}
		\text{ \ \textrm{and} }b\neq 0.
	\end{equation*}%
	However, $b$ doesn't affect controllability in high frequency unless it is
	time dependent. (See Remark \ref{Remark b independent of time}).
\end{remark}

\subsection{The authonomous case}

When the coefficients don't depend on time, we can give a characterization
to cover the invisible target states by applying the Fattorini criterion on
system (\ref{Z system}) which is possible since the difference between the
control maps is compact by Theroem \ref{theorem compactness} (See \cite[%
Remarks 2.4 and 1.5]{Duprez}). Note that the Fattorini criterion with the
fact that the two control maps is compact implies exact controllability of
the complete system (\ref{Z system}) in the same minimal time $T=4$. So, in
order to cover the invisible target states we have to ensure that 
\begin{equation}
	\ker (sI-\mathcal{A}^{\ast })\cap \ker (C^{\ast })=\{0\},\text{ }\forall
	s\in 
	\mathbb{C}
	,  \label{FH}
\end{equation}%
where $C^{\ast }:\left( p,q\right) \mapsto (B^{\ast }p_{|x=0},0)$ and $%
\mathcal{A}^{\ast }$ is the operator 
\begin{equation*}
	\mathcal{A}^{\ast }\left( 
	\begin{array}{c}
		p \\ 
		q%
	\end{array}%
	\right) =\left( 
	\begin{array}{c}
		-p_{x}-M^{\ast }\left( \eta _{1}p+\eta _{2}q\right) \\ 
		q_{x}-M^{\ast }\left( \eta _{1}p+\eta _{2}q\right)%
	\end{array}%
	\right) ,
\end{equation*}%
with domain%
\begin{equation}
	D(\mathcal{A}^{\ast })=\{\left( p,q\right) \in H^{1}(0,1)^{2}\times
	H^{1}(0,1)^{2},\text{ }\left( p+q\right) _{|x=0,1}=0,\text{ }%
	\int_{0}^{1}\left( p-q\right) =0\}.  \label{D(A)}
\end{equation}%
Introduce the matrices $Q_{i},i=1,2$ defined by%
\begin{equation*}
	Q_{0}=\left( 
	\begin{array}{cc}
		I_{2\times 2} & I_{2\times 2} \\ 
		0 & 0%
	\end{array}%
	\right) ,\text{ \ }Q_{1}=\left( 
	\begin{array}{cc}
		0 & 0 \\ 
		I_{2\times 2} & I_{2\times 2}%
	\end{array}%
	\right) ,
\end{equation*}%
and let $R_{s}(\cdot ,\cdot )$ the fundamental matrix of the finite
dimensional system 
\begin{equation}
	\left( 
	\begin{array}{c}
		p \\ 
		q%
	\end{array}%
	\right) _{x}=\left( 
	\begin{array}{cc}
		-sI_{2\times 2}-\eta _{1}M^{\ast } & -\eta _{2}M^{\ast } \\ 
		\eta _{1}M^{\ast } & \eta _{2}M^{\ast }+sI_{2\times 2}%
	\end{array}%
	\right) \left( 
	\begin{array}{c}
		p \\ 
		q%
	\end{array}%
	\right) .  \label{System diff}
\end{equation}%
Then we have:

\begin{proposition}
	The Fattorini criterion is satsified if, and only if 
	\begin{equation}
		\mathrm{rank}\text{ }\left( 
		\begin{array}{c}
			Q_{0}+Q_{1}\mathcal{R}_{s}(1,0) \\ 
			\left( B^{\ast },0,0\right)%
		\end{array}%
		\right) =4,\text{ for all }s\in 
		\mathbb{C}
		.\text{ }  \label{criterion}
	\end{equation}
\end{proposition}

\begin{proof}
	Let $(p,q)\in D(\mathcal{A}^{\ast }).$ So $(p,q)\in \ker (sI-\mathcal{A}%
	^{\ast })$ if and only if%
	\begin{equation}
		\left( 
		\begin{array}{c}
			p \\ 
			q%
		\end{array}%
		\right) _{x}=\left( 
		\begin{array}{cc}
			-sI_{2\times 2}-\eta _{1}M^{\ast } & -\eta _{2}M^{\ast } \\ 
			\eta _{1}M^{\ast } & \eta _{2}M^{\ast }+sI_{2\times 2}%
		\end{array}%
		\right) \left( 
		\begin{array}{c}
			p \\ 
			q%
		\end{array}%
		\right) ,  \label{SS1}
	\end{equation}%
	whith boundary conditions $\left( p+q\right) _{|x=0,1}=0$ which can be
	rewritten as 
	\begin{equation}
		\left( 
		\begin{array}{c}
			0_{2} \\ 
			0_{2}%
		\end{array}%
		\right) =Q_{0}\left( 
		\begin{array}{c}
			p(0) \\ 
			q(0)%
		\end{array}%
		\right) +Q_{1}\left( 
		\begin{array}{c}
			p(1) \\ 
			q(1)%
		\end{array}%
		\right) .  \label{SS2}
	\end{equation}%
	If $\mathcal{R}_{s}(\cdot ,\cdot )$ denotes the fundamental matrix solution
	to System (\ref{SS1}), then we have%
	\begin{equation*}
		\left( 
		\begin{array}{c}
			p(x) \\ 
			q(x)%
		\end{array}%
		\right) =\mathcal{R}_{s}(x,0)\left( 
		\begin{array}{c}
			p(0) \\ 
			q(0)%
		\end{array}%
		\right) .
	\end{equation*}%
	By using the boundary conditions (\ref{SS2}) we arrive at%
	\begin{equation*}
		0_{4}=\left( Q_{0}+Q_{1}\mathcal{R}_{s}(1,0)\right) \left( 
		\begin{array}{c}
			p(0) \\ 
			q(0)%
		\end{array}%
		\right) .
	\end{equation*}%
	On the other hand, $(p,q)\in \ker C^{\ast }$ if and only if $B^{\ast
	}p(0)=0_{2}.$ So the property (\ref{FH}) is satsified if and only if (\ref%
	{criterion}) holds which implies that $p(0)=q(0)=0$ which gives $p=q=0.$
\end{proof}

\begin{remark}
	Generally, knowing $\mathcal{R}_{s}(\cdot ,\cdot )$ is not possible.
	However, the above characterization could work for some particular classes
	of systems. For instance, for $M$ of the form%
	\begin{equation*}
		M=\left( 
		\begin{array}{cc}
			0\text{ } & \text{ }1 \\ 
			0\text{ } & \text{ }0%
		\end{array}%
		\right) .
	\end{equation*}
\end{remark}

\subsection{Cascade coupling}

In this subsection we prove the \emph{unique continuation} property for a
particular class of System (\ref{Z diag system_bis}). We assume in the
sequel that the matrix $M$ and the vector $B$ have the following form:\qquad 
\begin{equation*}
	M=\left( 
	\begin{array}{cc}
		0\text{ } & \text{ }1 \\ 
		0\text{ } & \text{ }0%
	\end{array}%
	\right) ,\text{ }B=\left( 
	\begin{array}{c}
		0 \\ 
		1%
	\end{array}%
	\right) .
\end{equation*}%
With this in mind, and by decomposing the system by writing $%
p=(p^{-},p^{+}), $ $q=(q^{-},q^{+}),$ the unique continuation problem of
System (\ref{Z system}) reads%
\begin{equation}
	\left\{ 
	\begin{array}{lll}
		p_{t}^{-}+p_{x}^{-}=0, & \mathrm{in} & Q_{T}, \\ 
		q_{t}^{-}-q_{x}^{-}=0, & \mathrm{in} & Q_{T}, \\ 
		p_{t}^{+}+p_{x}^{+}-\eta _{1}p^{-}-\eta _{2}q^{-}=0, & \mathrm{in} & Q_{T},
		\\ 
		q_{t}^{+}-q_{x}^{+}-\eta _{1}p^{-}-\eta _{2}q^{-}=0, & \mathrm{in} & Q_{T},
		\\ 
		\left( p^{-}+q^{-}\right) _{|x=0,1}=\left( p^{+}+q^{+}\right) _{|x=0,1}=0, & 
		\mathrm{in} & (0,T), \\ 
		p_{x=0}^{+}=0, & \mathrm{in} & (0,T), \\ 
		\left( p^{-},p^{+},q^{-},q^{+}\right) _{|t=0}=\left(
		p_{0}^{-},p_{0}^{+},q_{0}^{-},q_{0}^{+}\right) , & \mathrm{in} & (0,1),%
	\end{array}%
	\right. \Longrightarrow \left(
	p_{0}^{-},p_{0}^{+},q_{0}^{-},q_{0}^{+}\right) =0_{H},  \label{system unique}
\end{equation}%
for any $\left( p_{0}^{-},p_{0}^{+},q_{0}^{-},q_{0}^{+}\right) \in $ $D(%
\mathcal{A}^{\ast })$ where $D(\mathcal{A}^{\ast })$ is defined in (\ref%
{D(A)}).

Observe that the first and the third equations of the above system are free.
The idea is to solve these equations explicitly and then considering their
solution as a second member for the second and the fourth equations. Let us
start by solving the homogeneous part of System (\ref{system unique}). i.e.%
\begin{equation}
	\left\{ 
	\begin{array}{lll}
		p_{t}^{-}+p_{x}^{-}=0, & \mathrm{in} & Q_{T}, \\ 
		q_{t}^{-}-q_{x}^{-}=0, & \mathrm{in} & Q_{T}, \\ 
		\left( p^{+}+q^{+}\right) _{|x=0,1}=0, & \mathrm{in} & (0,T), \\ 
		\left( p^{-},q^{-}\right) _{|t=0}=\left( p_{0}^{-},q_{0}^{-}\right) , & 
		\mathrm{in} & (0,1).%
	\end{array}%
	\right.  \label{Homo}
\end{equation}

\begin{lemma}
	The solution $(p^{-},q^{-})$ to System (\ref{Homo}) are given by%
	\begin{equation}
		p^{-}(t,x)=\left\{ 
		\begin{array}{ccc}
			p_{0}^{-}(x-t+2n), & \mathrm{if} & 2n-1\leq t-x\leq 2n, \\ 
			-q_{0}^{-}(t-x-2n),\text{ \ } & \mathrm{if} & 2n\leq t-x\leq 2n+1,%
		\end{array}%
		\right. n\geq 0,  \label{p^-}
	\end{equation}%
	\begin{equation}
		q^{-}(t,x)=\left\{ 
		\begin{array}{ccc}
			q_{0}^{-}(x+t-2n),\text{ \ \ } & \mathrm{if} & 2n\leq x+t\leq 2n+1,\text{ \
				\ \ \ } \\ 
			-p_{0}^{-}(2n+2-x-t), & \mathrm{if} & 2n+1\leq x+t\leq 2n+2,%
		\end{array}%
		\right. n\geq 0.  \label{q^-}
	\end{equation}
\end{lemma}

\begin{proof}
	The proof follows immediately by using the characteristics method.
\end{proof}

Now, we focus on the nonhomogeneous part of System (\ref{system unique})
where boundary condition $q_{|x=0}^{+}=0$ has been removed, i.e.%
\begin{equation}
	\left\{ 
	\begin{array}{lll}
		p_{t}^{+}+p_{x}^{+}=f,\text{ } & \mathrm{in} & Q_{T}, \\ 
		q_{t}^{+}-q_{x}^{+}=f, & \mathrm{in} & Q_{T}, \\ 
		\left( p^{+}+q^{+}\right) _{|x=1}=p_{|x=0}^{+}=0, & \mathrm{in} & (0,T), \\ 
		\left( p^{+},q^{+}\right) _{|t=0}=\left( p_{0}^{+},q_{0}^{+}\right) . & 
		\mathrm{in} & (0,1).%
	\end{array}%
	\right.  \label{nonhom}
\end{equation}

Observe that the function $f=\eta _{1}p^{-}+\eta _{2}q^{-}$ plays the role
of a second member. The explicit solution to (\ref{nonhom}) is given in the
following lemma:

\begin{lemma}
	The solution $(p^{+},q^{+})$ to System (\ref{nonhom}) are given by%
	\begin{equation*}
		p^{+}\left( t,x\right) =\left\{ 
		\begin{array}{lll}
			p_{0}^{+}\left( x-t\right) +\int_{0}^{t}f\left( \tau ,\tau +x-t\right) d\tau
			, & \mathrm{if} & 0\leq x-t\leq 1 \\ 
			&  &  \\ 
			\int_{t-x}^{t}f(\tau ,\tau -t+x)d\tau , & \mathrm{if} & t-x\geq 1,\text{ }%
		\end{array}%
		\right.
	\end{equation*}%
	and%
	\begin{equation}
		q^{+}\left( t,x\right) =\left\{ 
		\begin{array}{lll}
			q_{0}^{+}\left( x+t\right) +\int_{0}^{t}f\left( \tau ,x+t-\tau \right) d\tau
			, & \mathrm{if} & 0\leq x+t\leq 1, \\ 
			&  &  \\ 
			\begin{array}{l}
				\int_{x+t-1}^{t}f(\tau ,x+t-\tau )d\tau -p_{0}^{+}\left( 2-x-t\right) \\ 
				-\int_{0}^{x+t-1}f\left( \tau ,\tau +2-x-t\right) d\tau%
			\end{array}
			& \mathrm{if} & 1\leq x+t\leq 2, \\ 
			&  &  \\ 
			\begin{array}{l}
				\int_{x+t-1}^{t}f(\tau ,-\tau +x+t)d\tau \\ 
				-\int_{x+t-2}^{x+t-1}f(\tau ,\tau +2-t-x)d\tau ,%
			\end{array}
			& \mathrm{if} & x+t\geq 2,%
		\end{array}%
		\right.  \label{q++}
	\end{equation}
\end{lemma}

\begin{proof}
	By using the characteristics method, it follows 
	\begin{equation*}
		p^{+}\left( t,x\right) =p_{0}^{+}\left( x-t\right) +\int_{0}^{t}f\left( \tau
		,\tau +x-t\right) d\tau ,\text{ }\mathrm{if}\text{ }0\leq x-t\leq 1,
	\end{equation*}%
	and%
	\begin{equation*}
		q^{+}\left( t,x\right) =q_{0}^{+}\left( x+t\right) +\int_{0}^{t}f\left( \tau
		,x+t-\tau \right) d\tau ,\text{ }\mathrm{if}\text{ }0\leq x+t\leq 1.
	\end{equation*}%
	Now, at $x=1$ we obtain%
	\begin{equation*}
		p^{+}\left( t,1\right) =p_{0}^{+}\left( 1-t\right) +\int_{0}^{t}f\left( \tau
		,\tau +1-t\right) d\tau ,\text{ }\mathrm{if}\text{ }0\leq 1-t\leq 1.
	\end{equation*}%
	Using the boundary condition, $q^{+}\left( s,1\right) =-p^{+}\left(
	s,1\right) ,$ $s\geq 0,$ entails%
	\begin{equation*}
		q^{+}\left( s,1\right) =-p_{0}^{+}\left( 1-s\right) -\int_{0}^{s}f\left(
		\tau ,\tau +1-s\right) d\tau ,\text{ }\mathrm{if}\text{ }0\leq 1-s\leq 1.
	\end{equation*}%
	Solving the second equation of System (\ref{nonhom}) along the
	characteristic $x(t)=-t+s+1$ we get%
	\begin{eqnarray*}
		q^{+}(t,-t+s+1) &=&q^{+}(s,1)+\int_{s}^{t}f(\tau ,s+1-\tau )d\tau \\
		&=&-p_{0}^{+}\left( 1-s\right) -\int_{0}^{s}f\left( \tau ,\tau +1-s\right)
		d\tau +\int_{s}^{t}f(\tau ,s+1-\tau )d\tau .
	\end{eqnarray*}%
	Letting $x=-t+s+1$ yields for any $0\leq 2-x-t\leq 1:$%
	\begin{equation*}
		q^{+}(t,x)=-p_{0}^{+}\left( 2-x-t\right) -\int_{0}^{x+t-1}f\left( \tau ,\tau
		+2-x-t\right) d\tau +\int_{x+t-1}^{t}f(\tau ,x+t-\tau )d\tau .
	\end{equation*}%
	Now, since $p(s,0)=0,$ $s\geq 0,$ we have for any $t-x\geq 1$%
	\begin{equation}
		p^{+}(t,x)=\int_{t-x}^{t}f(\tau ,\tau -t+x)d\tau .  \label{122}
	\end{equation}%
	By taking $x=1$ in (\ref{122}), the using the boundary conditions $%
	q^{+}\left( s,1\right) =-p^{+}\left( s,1\right) ,$ $s\geq 0,$ we get%
	\begin{equation*}
		q^{+}(t,x)=\int_{x+t-1}^{t}f(\tau ,x+t-\tau )d\tau .
	\end{equation*}%
	Similarly, we obtain for any $x+t\geq 2$%
	\begin{equation*}
		q^{+}(t,x)=-\int_{x+t-2}^{x+t-1}f(\tau ,\tau +2-t-x)d\tau
		+\int_{x+t-1}^{t}f(\tau ,-\tau +x+t)d\tau ,
	\end{equation*}%
	which ends the proof.
\end{proof}

To satisfy the remained boundary condition $q_{|x=0}^{+}=0$, it suffices to
replace $x$ by $\emph{zero}$ in (\ref{q++}). This gives the following system
of equations%
\begin{equation}
	0=q_{0}^{+}\left( t\right) +\int_{0}^{t}f\left( \tau ,t-\tau \right) d\tau ,%
	\text{ }\mathrm{if}\text{ }0\leq t\leq 1,  \label{system 1}
\end{equation}%
\begin{equation}
	0=-p_{0}^{+}\left( 2-t\right) -\int_{0}^{t-1}f\left( \tau ,\tau +2-t\right)
	d\tau +\int_{t-1}^{t}f(\tau ,t-\tau )d\tau =0,\text{ }\mathrm{if}\text{ }%
	1\leq t\leq 2,\text{ \ }  \label{system 2}
\end{equation}%
\begin{equation}
	0=-\int_{t-2}^{t-1}f(\tau ,\tau +2-t)d\tau +\int_{t-1}^{t}f(\tau ,-\tau
	+t)d\tau =0,\text{ }\mathrm{if}\text{ }t\geq 2.\text{ \ }  \label{system 3}
\end{equation}

As a consequence, we have:

\begin{proposition}
	The unique continuation property (\ref{system unique}) holds true if, and
	only if the solution $\left( p_{0},q_{0}\right) $ to System (\ref{system 1}%
	)-(\ref{system 3}) is the null one.
\end{proposition}

The strategy is the following: We solve Equation (\ref{system 3}) which
depends only on the initial states $p_{0}^{-},q_{0}^{-}$ since $f$ does.
Then, we prove that $p_{0}^{-}\equiv q_{0}^{-}\equiv 0$ which entails that $%
f $ $\equiv 0$ since it depends linearly on $p_{0}^{-}$ and $q_{0}^{-}$.
This leads to $p_{0}^{+}\equiv q_{0}^{+}\equiv 0$ by (\ref{system 1}) and (%
\ref{system 2}).

Let $t\geq 2.$ Recall that $f=\eta _{1}p^{-}+\eta _{2}q^{-}.$ Equation (\ref%
{system 3}) becomes%
\begin{eqnarray*}
	0 &=&-\int_{t-2}^{t-1}(\eta _{1}p^{-})(\tau ,\tau +2-t)d\tau
	+\int_{t-1}^{t}(\eta _{2}q^{-})(\tau ,-\tau +t)d\tau \\
	&&-\int_{t-2}^{t-1}(\eta _{2}q^{-})(\tau ,\tau +2-t)d\tau
	+\int_{t-1}^{t}(\eta _{1}p^{-})(\tau ,-\tau +t)d\tau \\
	&:&=I_{1}(t)+I_{2}(t)+I_{3}(t)+I_{4}(t).
\end{eqnarray*}%
Since $p^{-}$ (resp. $q^{-}$) is defined on the characteristics of slope $1$
(resp. $-1$), $I_{1}(\cdot )$ (resp. $I_{2}(\cdot )$) can be easily
computed. Indeed, by using the expressions of $p^{-}$ and $q^{-}$ given in (%
\ref{p^-}) and (\ref{q^-}) respectively, we obtain%
\begin{eqnarray}
	I_{1}(t) &=&-\int_{t-2}^{t-1}(\eta _{1}p^{-})(\tau ,\tau +2-t)d\tau
	\label{I_1} \\
	&=&\left\{ 
	\begin{array}{ccc}
		-\left( \int_{t-2}^{t-1}\eta _{1}(\tau ,\tau +2-t)d\tau \right)
		p_{0}^{-}(4+2n-t), & \mathrm{if} & 2n+3\leq t\leq 2n+4, \\ 
		\text{ \ }\left( \int_{t-2}^{t-1}\eta _{1}(\tau ,\tau +2-t)d\tau \right)
		q_{0}^{-}(t-2-2n),\text{\ \ } & \mathrm{if} & 2n+2\leq t\leq 2n+3,%
	\end{array}%
	\right. n\geq 0,  \notag
\end{eqnarray}%
and 
\begin{eqnarray}
	I_{2}(t) &=&\int_{t-1}^{t}(\eta _{1}q^{-})(\tau ,-\tau +t)d\tau  \label{I_2}
	\\
	&=&\left\{ 
	\begin{array}{ccc}
		\left( \int_{t-1}^{t}\eta _{1}(\tau ,-\tau +t)d\tau \right) q_{0}^{-}(t-2n),%
		\text{ \ \ } & \mathrm{if} & 2n\leq t\leq 2n+1,\text{ \ \ \ \ } \\ 
		-\left( \int_{t-1}^{t}\eta _{1}(\tau ,-\tau +t)d\tau \right)
		p_{0}^{-}(2n+2-t), & \mathrm{if} & 2n+1\leq t\leq 2n+2,%
	\end{array}%
	\right. n\geq 1.  \notag
\end{eqnarray}%
It can be seen that $I_{1}(\cdot )$ and $I_{2}(\cdot )$ are the high
frequency part of the solution. Now, we deal with the compact terms $%
I_{3}(\cdot )$ and $I_{4}(\cdot ).$

By using the expression of $q^{-}$ given in (\ref{q^-}), we obtain for any $%
n\geq 0$%
\begin{eqnarray}
	I_{3}(t) &=&-\int_{t-2}^{t-1}(\eta _{2}q^{-})(\tau ,\tau +2-t)d\tau
	\label{I-3} \\
	&&\left\{ 
	\begin{array}{ccc}
		-\int_{t-2}^{t-1}\eta _{2}(\tau ,\tau +2-t)q_{0}^{-}(2\tau +2-t-2n)d\tau , & 
		\mathrm{if} & \tau \in \left( \frac{2n-2+t}{2},\frac{2n-1+t}{2}\right) \cap
		(0,t), \\ 
		\begin{array}{c}
			\\ 
			\int_{t-2}^{t-1}\eta _{2}(\tau ,\tau +2-t)p_{0}^{-}(2n-2\tau +t)d\tau ,\text{
				\ }%
		\end{array}
		& 
		\begin{array}{c}
			\\ 
			\mathrm{if}%
		\end{array}
		& 
		\begin{array}{c}
			\\ 
			\tau \in \left( \frac{2n-1+t}{2},\frac{2n+t}{2}\right) \cap (0,t).\text{ \ }%
		\end{array}%
	\end{array}%
	\right.  \notag
\end{eqnarray}

\begin{itemize}
	\item If $t\in \lbrack 2n-1,2n),$ $n\geq 2:$ In this case, the interval $%
	(t-2,t-1)$ can be decomposed as 
	\begin{equation*}
		(t-2,t-1)=[t-2,\frac{2n-4+t}{2})\cup (\frac{2n-4+t}{2},\frac{2n-3+t}{2})\cup
		(\frac{2n-3+t}{2},t-1).
	\end{equation*}%
	Therefore, by using (\ref{I-3}) we get 
	\begin{eqnarray*}
		I_{3}(t) &=&\int_{t-2}^{\frac{2n-4+t}{2}}\eta _{2}(\tau ,\tau
		+2-t)p_{0}^{-}(2n-4+t-2\tau )d\tau \\
		&&-\int_{\frac{2n-4+t}{2}}^{\frac{2n-3+t}{2}}\eta _{2}(\tau ,\tau
		+2-t)q_{0}^{-}(2\tau +4-t-2n)d\tau \\
		&&+\int_{\frac{2n-3+t}{2}}^{t-1}\eta _{2}(\tau ,\tau
		+2-t)p_{0}^{-}(2n-2+t-2\tau )d\tau .
	\end{eqnarray*}%
	And after a change of variables, we obtain%
	\begin{eqnarray}
		I_{3}(t) &=&\frac{1}{2}\int_{0}^{2n-t}\eta _{2}(\frac{2n-4+t-s}{2},\frac{%
			2n-t-s}{2})p_{0}^{-}(s)ds  \label{I_3 (2n-1,2n)} \\
		&&-\frac{1}{2}\int_{0}^{1}\eta _{2}(\frac{s+t-4+2n}{2},\frac{s-t+2n}{2}%
		)q_{0}^{-}(s)ds  \notag \\
		&&+\frac{1}{2}\int_{2n-t}^{1}\eta _{2}(\frac{2n-2+t-s}{2},\frac{2n+2-t-s}{2}%
		)p_{0}^{-}(s)ds.  \notag
	\end{eqnarray}
	
	\item If $t\in \lbrack 2n,2n+1),$ $n\geq 1:$
	
	The interval $[t-2,t-1)$ can be written as%
	\begin{equation*}
		(t-2,t-1)=(t-2,\frac{2n-3+t}{2})\cup (\frac{2n-3+t}{2},\frac{2n-2+t}{2})\cup
		(\frac{2n-2+t}{2},t-1).
	\end{equation*}%
	Similarly, we obtain%
	\begin{eqnarray*}
		I_{3}(t) &=&-\int_{t-2}^{\frac{2n-3+t}{2}}\eta _{2}(\tau ,\tau
		+2-t)q_{0}^{-}(2\tau +4-t-2n)d\tau \\
		&&+\int_{\frac{2n-3+t}{2}}^{\frac{2n-2+t}{2}}\eta _{2}(\tau ,\tau
		+2-t)p_{0}^{-}(2n-2+t-2\tau )d\tau \\
		&&-\int_{\frac{2n-2+t}{2}}^{t-1}\eta _{2}(\tau ,\tau +2-t)q_{0}^{-}(2\tau
		+2-t-2n)d\tau .
	\end{eqnarray*}%
	And after a change of variables we get%
	\begin{eqnarray}
		I_{3}(t) &=&-\frac{1}{2}\int_{t-2n}^{1}\eta _{2}(\frac{s+2n+t-4}{2},\frac{%
			s+2n-t}{2})q_{0}^{-}(s)ds  \label{I_3 (2n,2n+1)} \\
		&&+\frac{1}{2}\int_{0}^{1}\eta _{2}(\frac{2n+t-2-s}{2},\frac{2n-t+2-s}{2}%
		)p_{0}^{-}(s)ds  \notag \\
		&&-\frac{1}{2}\int_{0}^{t-2n}\eta _{2}(\frac{2n+t-2+s}{2},\frac{2n-t+2+s}{2}%
		)q_{0}^{-}(s)ds.  \notag
	\end{eqnarray}
\end{itemize}

Now, we deal with $I_{4}(\cdot ).$ In the same way, we have for any $n\geq
0: $%
\begin{eqnarray}
	I_{4}(t) &=&\int_{t-1}^{t}(\eta _{1}p^{-})(\tau ,-\tau +t)d\tau  \label{I-4}
	\\
	&&\left\{ 
	\begin{array}{ccc}
		\int_{t-1}^{t}\eta _{1}(\tau ,-\tau +t)p_{0}^{-}(-2\tau +t+2n)d\tau , & 
		\mathrm{if} & \tau \in \left( \frac{2n-1+t}{2},\frac{2n+t}{2}\right) \cap
		(0,t), \\ 
		\begin{array}{c}
			\\ 
			-\int_{t-1}^{t}\eta _{1}(\tau ,-\tau +t)q_{0}^{-}(2\tau -t-2n)d\tau ,\text{
				\ }%
		\end{array}%
		\text{ \ } & 
		\begin{array}{c}
			\\ 
			\mathrm{if}%
		\end{array}
		& 
		\begin{array}{c}
			\\ 
			\tau \in \left( \frac{2n+t}{2},\frac{2n+t+1}{2}\right) \cap (0,t),%
		\end{array}%
	\end{array}%
	\right.  \notag
\end{eqnarray}

Consider the first case:

\begin{itemize}
	\item If $t\in \lbrack 2n,2n+1),$ $n\geq 1:$ In this case, we write%
	\begin{equation*}
		(t-1,t)=(t-1,\frac{t+2n-1}{2})\cup (\frac{t+2n-1}{2},\frac{t+2n}{2})\cup (%
		\frac{t+2n}{2},t),
	\end{equation*}%
	\begin{eqnarray*}
		I_{4}(t) &=&-\int_{t-1}^{\frac{t+2n-1}{2}}\eta _{1}(\tau ,-\tau
		+t)q_{0}^{-}(2\tau -t-2n+2)d\tau \\
		&&+\int_{\frac{t+2n-1}{2}}^{\frac{t+2n}{2}}\eta _{1}(\tau ,-\tau
		+t)p_{0}^{-}(-2\tau +t+2n)d\tau \\
		&&-\int_{\frac{t+2n}{2}}^{t}\eta _{1}(\tau ,-\tau +t)q_{0}^{-}(2\tau
		-t-2n)d\tau ,
	\end{eqnarray*}%
	which yields after a change of variables 
	\begin{eqnarray}
		I_{4}(t) &=&-\frac{1}{2}\int_{t-2n}^{1}\eta _{1}(\frac{s+t+2n-2}{2},\frac{%
			t-2n-s+2}{2})q_{0}^{-}(s)ds  \label{I_4 (2n,2n+1)} \\
		&&+\frac{1}{2}\int_{0}^{1}\eta _{1}(\frac{2n+t-s}{2},\frac{t+s-2n}{2}%
		)p_{0}^{-}(s)ds  \notag \\
		&&-\frac{1}{2}\int_{0}^{t-2n}\eta _{1}(\frac{2n+t+s}{2},\frac{t-2n-s}{2}%
		)q_{0}^{-}(s)ds.  \notag
	\end{eqnarray}
	
	\item If $t\in \lbrack 2n+1,2n+2),$ $n\geq 1:$ In the same way, 
	\begin{equation*}
		(t-1,t)=(t-1,\frac{t+2n}{2})\cup (\frac{t+2n}{2},\frac{t+2n+1}{2})\cup (%
		\frac{t+2n+1}{2},t),
	\end{equation*}%
	so, we obtain%
	\begin{eqnarray*}
		I_{4}(t) &=&\int_{t-1}^{\frac{t+2n}{2}}\eta _{1}(\tau ,-\tau
		+t)p_{0}^{-}(-2\tau +t+2n)d\tau \\
		&&-\int_{\frac{t+2n}{2}}^{\frac{t+2n+1}{2}}\eta _{1}(\tau ,-\tau
		+t)q_{0}^{-}(2\tau -t-2n)d\tau \\
		&&+\int_{\frac{t+2n+1}{2}}^{t}\eta _{1}(\tau ,-\tau +t)p_{0}^{-}(-2\tau
		+t+2n+2)d\tau ,
	\end{eqnarray*}%
	and after a change of variables we obtain%
	\begin{eqnarray}
		I_{4}(t) &=&\frac{1}{2}\int_{0}^{2n+2-t}\eta _{1}(\frac{2n+t-s}{2},\frac{%
			-2n+t+s}{2})p_{0}^{-}(s)ds  \label{I_4 (2n+1,2n+2)} \\
		&&-\frac{1}{2}\int_{0}^{1}\eta _{1}(\frac{2n+t+s}{2},\frac{t-2n-s}{2}%
		)q_{0}^{-}(s)ds  \notag \\
		&&+\frac{1}{2}\int_{2n+2-t}^{1}\eta _{1}(\frac{2n+2+t-s}{2},\frac{t+s-2n-2}{2%
		})p_{0}^{-}(s)ds.  \notag
	\end{eqnarray}
\end{itemize}

To summarize, the initial states $p_{0}^{-},q_{0}^{-}$ are solutions of the
following two equations:

\begin{itemize}
	\item If $t\in \lbrack 2n,2n+1),$ $n\geq 1:$
	
	By gathering (\ref{I_1}), (\ref{I_2}), (\ref{I_3 (2n,2n+1)}), and (\ref{I_4
		(2n,2n+1)}) we get%
	\begin{eqnarray*}
		0 &=&\left( \int_{t-2}^{t-1}\eta _{1}(\tau ,\tau +2-t)d\tau
		+\int_{t-1}^{t}\eta _{1}(\tau ,-\tau +t)d\tau \right) q_{0}^{-}(t-2n) \\
		&&-\text{\ }\frac{1}{2}\int_{t-2n}^{1}\left[ \eta _{1}(\frac{s+t+2n-2}{2},%
		\frac{t-2n-s+2}{2})+\eta _{2}(\frac{s+2n+t-4}{2},\frac{s+2n-t}{2})\right]
		q_{0}^{-}(s)ds \\
		&&+\frac{1}{2}\int_{0}^{1}\left[ \eta _{1}(\frac{2n+t-s}{2},\frac{t+s-2n}{2}%
		)+\eta _{2}(\frac{2n+t-2-s}{2},\frac{2n-t+2-s}{2})\right] p_{0}^{-}(s)ds \\
		&&-\frac{1}{2}\int_{0}^{t-2n}\left[ \eta _{1}(\frac{2n+t+s}{2},\frac{t-2n-s}{%
			2})+\eta _{2}(\frac{2n+t-2+s}{2},\frac{2n-t+2+s}{2})\right] q_{0}^{-}(s)ds.
	\end{eqnarray*}%
	Letting $x=t-2n$ yields%
	\begin{equation}
		\phi (x+2n)q_{0}^{-}(x)-\int_{0}^{1}\left( K_{n}^{21}(s,x)p_{0}^{-}\left(
		s\right) +K_{n}^{22}(s,x)q_{0}^{-}\left( s\right) \right) ds=0,\text{ }x\in
		(0,1),  \label{integral 1}
	\end{equation}%
	where $\phi $ is defined in (\ref{phi}) and the kernels $K_{n}^{21}\left(
	\cdot ,\cdot \right) ,$ $K_{n}^{22}\left( \cdot ,\cdot \right) ,$ are given
	by%
	\begin{equation}
		K_{n}^{21}(s,x)=-\frac{1}{2}\eta _{1}(\frac{4n+x-s}{2},\frac{x+s}{2})-\frac{1%
		}{2}\eta _{2}(\frac{4n+x-2-s}{2},\frac{2-s-x}{2}),\text{ }\left( s,x\right)
		\in (0,1)^{2},  \label{K21}
	\end{equation}%
	\begin{equation}
		K_{n}^{22}(s,x)=\frac{1}{2}\left\{ 
		\begin{array}{lll}
			\eta _{1}(\frac{4n+x+s}{2},\frac{x-s}{2})+\eta _{2}(\frac{4n-2+x+s}{2},\frac{%
				2+s-x}{2}), & \mathrm{if} & 0\leq s\leq x, \\ 
			&  &  \\ 
			\eta _{1}(\frac{4n-2+s+x}{2},\frac{x-s+2}{2})+\eta _{2}(\frac{4n-4+s+x}{2},%
			\frac{s-x}{2}), & \mathrm{if} & x\leq s\leq 1.%
		\end{array}%
		\right.  \label{K22}
	\end{equation}
	
	Observe that when $t$ varies in $[2n,2n+1)$ the $x$ varies in $[0,T-2n).$
	
	\item If $t\in \lbrack 2n+1,2n+2),$ $n\geq 1:$
	
	Gathering (\ref{I_1}), (\ref{I_2}), \ref{I_3 (2n-1,2n)}, and (\ref{I_4
		(2n+1,2n+2)}) yields 
	\begin{eqnarray*}
		0 &=&-\left( \int_{t-2}^{t-1}\eta _{1}(\tau ,\tau +2-t)d\tau
		+\int_{t-1}^{t}\eta _{1}(\tau ,-\tau +t)d\tau \right) p_{0}^{-}(2n+2-t) \\
		&&+\frac{1}{2}\int_{0}^{2n+2-t}\left[ \eta _{1}(\frac{2n+t-s}{2},\frac{%
			-2n+t+s}{2})+\eta _{2}(\frac{2n-2+t-s}{2},\frac{2n+2-t-s}{2})\right]
		p_{0}^{-}(s)ds \\
		&&+\frac{1}{2}\int_{2n+2-t}^{1}\left[ \eta _{1}(\frac{2n+2+t-s}{2},\frac{%
			t+s-2n-2}{2})+\eta _{2}(\frac{2n+t-s}{2},\frac{2n+4-t-s}{2})\right]
		p_{0}^{-}(s)ds \\
		&&-\frac{1}{2}\int_{0}^{1}\left[ \eta _{1}(\frac{2n+t+s}{2},\frac{t-2n-s}{2}%
		)+\eta _{2}(\frac{s+t-2+2n}{2},\frac{s-t+2n+2}{2})\right] q_{0}^{-}(s)ds.
	\end{eqnarray*}%
	Letting $2n+2-t=x$ yields%
	\begin{equation}
		\phi (2n+2-x)p_{0}^{-}(x)-\int_{0}^{1}\left[
		K_{n}^{11}(s,x)p_{0}^{-}(s)+K_{n}^{12}(s,x)q_{0}^{-}(s)\right] ds=0,\text{ }%
		x\in (0,1),\text{ }n\geq 1,  \label{integral 2}
	\end{equation}%
	where $\phi $ is defined in (\ref{phi}) and the kernels $K_{n}^{11}\left(
	\cdot ,\cdot \right) ,$ $K_{n}^{12}\left( \cdot ,\cdot \right) ,$ are given
	by 
	\begin{equation}
		K_{n}^{11}(s,x)=\frac{1}{2}\left\{ 
		\begin{array}{lll}
			\eta _{1}(\frac{4n+2-x-s}{2},\frac{2-x+s}{2})+\eta _{2}(\frac{4n-x-s}{2},%
			\frac{x-s}{2}), & \mathrm{if} & 0\leq s\leq x, \\ 
			&  &  \\ 
			\eta _{1}(\frac{4n+4-x-s}{2},\frac{s-x}{2})+\eta _{2}(\frac{4n+2-x-s}{2},%
			\frac{2+x-s}{2}) & \mathrm{if} & x\leq s\leq 1,%
		\end{array}%
		\right.  \label{K11}
	\end{equation}%
	\begin{equation}
		K_{n}^{12}(s,x)=-\frac{1}{2}\eta _{1}(\frac{4n+2-x+s}{2},\frac{2-x-s}{2})-%
		\frac{1}{2}\eta _{2}(\frac{s+4n-x}{2},\frac{s+x}{2}),\text{ }\left(
		s,x\right) \in (0,1)^{2}.  \label{K12}
	\end{equation}%
	Similarly, when $t$ varies $[2n+1,2n+2)$ then $x$ varies $(2n+2-T,1].$
\end{itemize}

Observe that equations (\ref{integral 1}) and (\ref{integral 2}) form a
system of Fredholm Integral equations of third kind. Indeed, if $t\in
\lbrack 2n,2n+1),$ the interval $[2,t)$ can be written as%
\begin{equation*}
	\lbrack 2,t)=\left( \cup _{l=1}^{n-1}[2l,2l+1)\right) \cup \left( \cup
	_{k=1}^{n-1}[2k+1,2k+2])\right) \cup \lbrack 2n,t),
\end{equation*}

In this case, we have to solve Equation (\ref{integral 1}) (resp. Equation (%
\ref{integral 2})) in intervals of the form $[2l,2l+1)$ (resp. $[2k+1,2l+2)$%
) for some $k,l\geq 1$. In the same way, for $t\in \lbrack 2n+1,2n+2),$ the
interval $[2,t)$ can be written as%
\begin{equation*}
	\lbrack 2,t)=\left( \cup _{l=1}^{n}[2l,2l+1)\right) \cup \left( \cup
	_{k=1}^{n-1}[2k+1,2k+2])\right) \cup \lbrack 2n+1,T).
\end{equation*}%
Similarly, we have to solve Equation (\ref{integral 1}) (resp. Equation (\ref%
{integral 2})) in the intervals of the form $[2l,2l+1)$ (resp. $[2k+1,2l+2)$%
) for some $k,l\geq 1$. More precisely, introduce the kernel $\mathbb{K}%
_{n,k,l}(\cdot ,\cdot ),$ $1\leq i\leq 2,$ $n\geq 2,$ $k,l\geq 1,$ by

\begin{itemize}
	\item If $T\in \lbrack 2n,2n+1)$%
	\begin{equation*}
		\mathbb{K}_{n,k,l}(s,x)=\left\{ 
		\begin{array}{llll}
			\mathbf{K}_{k,l}(s,x), & 1\leq k\leq n-1,\text{ }1\leq l\leq n, & \mathrm{if}
			& x\in \lbrack 0,T-2n), \\ 
			&  &  &  \\ 
			\mathbf{K}_{k,l}(s,x), & 1\leq k\leq n-1,\text{ }1\leq l\leq n-1, & \mathrm{%
				if} & x\in \lbrack T-2n,1],%
		\end{array}%
		\right.
	\end{equation*}
	
	\item If $T\in \lbrack 2n+1,2n+2)$%
	\begin{equation*}
		\mathbb{K}_{n,k,l}(s,x)=\left\{ 
		\begin{array}{llll}
			\mathbf{K}_{k,l}(s,x), & 1\leq k\leq n,\text{ }1\leq l\leq n, & \mathrm{if}
			& x\in \lbrack 0,2n+2-T), \\ 
			&  &  &  \\ 
			\mathbf{K}_{k,l}(s,x), & 1\leq k\leq n+1,\text{ }1\leq l\leq n, & \mathrm{if}
			& x\in \lbrack 2n+2-T,1],%
		\end{array}%
		\right.
	\end{equation*}
	where
\end{itemize}

\begin{eqnarray*}
	A_{k,l}(x) &=&\left( 
	\begin{array}{cc}
		\phi (2k+2-x) & 0 \\ 
		0 & \phi (2l+x)%
	\end{array}%
	\right) ,\text{ }x\in \left[ 0,1\right] ,\text{ }k,l\geq 1, \\
	\text{ }\mathbf{K}_{k,l}(s,x) &=&\left( 
	\begin{array}{cc}
		K_{k}^{11}(s,x) & K_{k}^{12}(s,x) \\ 
		K_{l}^{21}(s,x) & K_{l}^{22}(s,x)%
	\end{array}%
	\right) ,\text{ }\left( s,x\right) \in \left[ 0,1\right] ^{2},\text{ }%
	k,l\geq 1,
\end{eqnarray*}%
associated with the third kind Fredholm integral equations%
\begin{equation}
	A_{k,l}(x)\left( 
	\begin{array}{c}
		p_{0}^{-}(x) \\ 
		q_{0}^{-}(x)%
	\end{array}%
	\right) =\int_{0}^{1}\mathbb{K}_{n,k,l}(s,x)\left( 
	\begin{array}{c}
		p_{0}^{-}(s) \\ 
		q_{0}^{-}(s)%
	\end{array}%
	\right) ds,  \label{Fredholm}
\end{equation}

Now, we come to the main theorem of this section:

\begin{theorem}
	\label{Theroem main approx}Let $n\geq 2$ be an integer. The\emph{\ unique
		continuation property }for (\ref{system unique}) holds true at time $T$ if,
	and only if, there exist $k,l\geq 1$ such that the unique solution $\left(
	p_{0}^{-},q_{0}^{-}\right) $ to Equation (\ref{Fredholm}) is the null one.
\end{theorem}

Now, let us assume that the weak observability holds, i.e

\begin{itemize}
	\item If $2n\leq T<2n+1:$
	
	\begin{enumerate}
		\item There exist $1\leq k\leq n-1,$ $1\leq l\leq n,$ such that 
		\begin{equation}
			\phi (2k+2-x)\neq 0,~\phi (2l+x)\neq 0,\text{ }\forall x\in \lbrack 0,T-2n).
			\label{AAA1}
		\end{equation}
		
		\item There exist $1\leq k\leq n-1,$ $1\leq l\leq n-1,$ such that 
		\begin{equation}
			\phi (2k+2-x)\neq 0,~\phi (2l+x)\neq 0,\text{ }\forall x\in \lbrack T-2n,1].
			\label{AA2}
		\end{equation}
	\end{enumerate}
	
	\item If $2n+1\leq T<2n+2:$
	
	\begin{enumerate}
		\item There exist $1\leq k\leq n-1,$ $1\leq l\leq n,$ such that 
		\begin{equation}
			\phi (2k+2-x)\neq 0,~\phi (2l+x)\neq 0,\text{ }\forall x\in \lbrack
			2n+2-T,0).  \label{AA3}
		\end{equation}
		
		\item There exist $1\leq k\leq n,$ $1\leq l\leq n,$ such that 
		\begin{equation}
			\phi (2k+2-x)\neq 0,~\phi (2l+x)\neq 0,\text{ }\forall x\in \lbrack
			2n+2-T,1].  \label{AA4}
		\end{equation}
	\end{enumerate}
\end{itemize}

Under assumptions (\ref{AAA1}) and (\ref{AA2}) (resp. (\ref{AA3}) and (\ref%
{AA4})), Equations (\ref{Fredholm}) writes 
\begin{equation}
	\left( 
	\begin{array}{c}
		p_{0}^{-}(x) \\ 
		q_{0}^{-}(x)%
	\end{array}%
	\right) =\int_{0}^{1}A_{k,l}^{-1}(x)\mathbb{K}_{n,k,l}(s,x)\left( 
	\begin{array}{c}
		p_{0}^{-}(s) \\ 
		q_{0}^{-}(s)%
	\end{array}%
	\right) ds:=\mathcal{K}_{k,l}\left( 
	\begin{array}{c}
		p_{0}^{-} \\ 
		q_{0}^{-}%
	\end{array}%
	\right) (x),\text{ }i=1,2,  \label{Fredholm 11}
\end{equation}%
which are a second kind Fredholm integral equations. The following corollary
is a straightforward consequence of the above theorem:

\begin{corollary}
	\label{Corollary approximat_Control}Let $n\geq 2$ be an integer. Assume that
	(\ref{AAA1}) and (\ref{AA4}) hold for some $k,l\geq 1$. The \emph{unique
		continuation} property for (\ref{system unique}) holds at time $T$ if, and
	only if $1\notin \sigma \left( \mathcal{K}_{k,l}\right) .$
\end{corollary}

\begin{proof}
	By assumptions (\ref{AAA1})-(\ref{AA4}), the Equation (\ref{Fredholm}) is
	Fredholm integral equation of second kind. By the Fredholm alternative, it
	possesses a unique solution if, and only if $1\notin \sigma \left( \mathcal{K%
	}_{k,l}\right) ,$ for some $l,k\geq 1.$
\end{proof}

\begin{remark}
	Since the component of the kernel $A_{k,l}^{-1}(\cdot )\mathbb{K}%
	_{n,k,l}(\cdot ,\cdot ),$ $n\geq k,l\geq 1,$ are completely known, we can
	always prove that $1\notin \sigma \left( \mathcal{K}_{k,l}\right) ,$ by
	assuming that $\left\Vert \mathcal{K}_{k,l}\right\Vert _{\mathcal{L}\left(
		L^{2}(0,1;M_{2\times 2}(%
		\mathbb{R}
		))\right) }$ $<1.$ Notice that this is not necessarily a smallness
	assumption on the coupling coefficients since the kernel involves the matrix 
	$A_{k,l}^{-1}(\cdot )$.
\end{remark}

\begin{remark}
	It is not difficult to see that if $\eta _{1}$ and $\eta _{2}$ are time
	independent, the compact operator $\mathcal{K}_{k,l}=\mathcal{K},$ is
	symmetric. Therefore, its spectrum consists of real eigenvalues. However,
	giving a characterization of the spectrum needs more care.
\end{remark}

\begin{remark}
	In \cite{Bennour}, it has been shown that the \emph{unique continuation}
	property for $a(t,x)=a(x)$ and $b=0$ holds at time $T>0$ for System (\ref%
	{system unique}) if, and only if $T\geq 4,$ and 
	\begin{equation}
		\int_{0}^{1}a(s)\sin ^{2}(\pi ns)ds\neq 0,\text{ }\forall n\geq 1.
		\label{Moment}
	\end{equation}%
	It is natural to expect that condition (\ref{Moment}) is equivalent to the
	uniqueness of the solution to Equation (\ref{Fredholm}) for $\mathbb{K}%
	_{n,k,l}=\mathbb{K}$.
\end{remark}

\begin{remark}
	With a simple change of variable (see \cite[Remark 15]{Alabau 4}), we can
	deduce that all the results proved for cascade system with velocity coupling
	($a\neq 0$) hold true for zero order coupling (replace $a\varphi _{t}$ by $%
	a\varphi $ in System (\ref{Wave_Adjoint})) in the space 
	\begin{equation*}
		H_{0}^{1}(0,1)\times L^{2}(0,1)\times L^{2}(0,1)\times H^{-1}(0,1).
	\end{equation*}
\end{remark}

\subsection{Examples}

\label{Examples unique}

Here, we will provide some illustrations of Theorem \ref{Theroem main approx}%
. For the sake of simplicity, we will use both coupling functions $a$ and $%
b, $ also, we will choose the time of control to be $T=2n$ for some $n\geq
2. $

\subsubsection{The case $\protect\eta _{1}(t,x)=\protect\alpha (t-x),$ $%
	\protect\eta _{2}(t,x)=\protect\beta (t+x)$}

Set $\eta _{1}(t,x)=\alpha (t-x)$ and $\eta _{1}(t,x)=\beta (t+x)$ for some
smooth functions $\alpha $ and $\beta $ in $Q_{T}.$ After performing a
simple computations, the components $K_{n}^{i,j}\left( \cdot ,\cdot \right)
, $ $1\leq i,j\leq 2,$ $n\geq 1,$ given in (\ref{K21}), (\ref{K22}), (\ref%
{K11}) and (\ref{K12}) and the function $\phi $ defined in (\ref{phi}) take
the form%
\begin{equation*}
	2K_{k}^{12}(s,x)=-\alpha (2k+s)-\beta (2k+s),\text{ }s\in (0,1).
\end{equation*}%
\begin{equation*}
	2K_{l}^{21}(s,x)=-\alpha (2l-s)-\beta (2l-s),\text{\ }s\in (0,1).
\end{equation*}%
\begin{equation*}
	2K_{k}^{11}(s,x)=\left\{ 
	\begin{array}{lll}
		\alpha (2k-s)+\beta (2k-s), & \mathrm{if} & 0\leq s\leq x, \\ 
		\alpha (2k+2-s)+\beta (2k+2-s), & \mathrm{if} & x\leq s\leq 1,%
	\end{array}%
	\right.
\end{equation*}%
\begin{equation*}
	2K_{l}^{22}(s,x)=\left\{ 
	\begin{array}{lll}
		\alpha (2l+s)+\beta (2l+s), & \mathrm{if} & 0\leq s\leq x, \\ 
		\alpha (2l-2+s)+\beta (2l-2+s), & \mathrm{if} & x\leq s\leq 1.%
	\end{array}%
	\right.
\end{equation*}%
\begin{eqnarray*}
	\phi \left( t\right) &=&\int_{t-2}^{t-1}\alpha (t-2)d\tau
	+\int_{t-1}^{t}\beta (t)d\tau \\
	&=&\alpha (t-2)+\beta (t),\text{ }t\geq 2.
\end{eqnarray*}%
Next, we define the functions $S_{j}^{i},$ $i=1,2$ on $[0,1]$ for any $j\geq
1$ by 
\begin{equation*}
	S_{j}^{1}(s)=\left( 2\phi (2j+2-s)\right) ^{-1}\left( 
	\begin{array}{c}
		-2\phi ^{\prime }(2j+2-s)-\alpha (2j+2-s) \\ 
		+\beta (2j-s)-\beta (2j+2-s)+\alpha (2j-s)%
	\end{array}%
	\right) ,
\end{equation*}%
and%
\begin{equation*}
	S_{j}^{2}(s)=\left( 2\phi (2j+s)\right) ^{-1}\left( 
	\begin{array}{c}
		-\phi ^{\prime }(2j+s)+\alpha (2j+s)+\beta (2j+s) \\ 
		-\alpha (2j-2+s)-\beta (2j-2+s)%
	\end{array}%
	\right) .
\end{equation*}

We have the following unique continuation result:

\begin{proposition}
	Let $n\geq 2.$ Assume that $A_{k,l}^{-1}(\cdot )$ exists on $[0,1]$ for some 
	$1\leq k,l\leq n-1$. Then the \emph{unique continuation} property (\ref%
	{system unique}) holds true in time $T=2n$ if%
	\begin{equation*}
		\int_{0}^{1}S_{k}^{1}(s)ds\neq \int_{0}^{1}S_{l}^{2}(s)ds,
	\end{equation*}
\end{proposition}

\begin{proof}
	For $T=2n$, the system of integral equations (\ref{Fredholm}) writes:%
	\begin{eqnarray}
		&&2\phi (2k+2-x)p_{0}^{-}(x)=\int_{0}^{x}\left( \alpha (2k-s)+\beta
		(2k-s)\right) p_{0}^{-}(s)ds  \label{EQ11} \\
		&&+\int_{x}^{1}\left( \alpha (2k+2-s)+\beta (2k+2-s)\right) p_{0}^{-}(s)ds 
		\notag \\
		&&-\int_{0}^{1}\left( \alpha (2k+s)+\beta (2k+s)\right) q_{0}^{-}(s)ds, 
		\notag
	\end{eqnarray}%
	\begin{eqnarray}
		&&2\phi (2l+x)q_{0}^{-}(x)=\int_{0}^{x}\left( \alpha (2l+s)+\beta
		(2l+s)\right) q_{0}^{-}(s)ds  \label{EQ12} \\
		&&+\int_{x}^{1}\left( \alpha (2l-2+s)+\beta (2l-2+s)\right) q_{0}^{-}(s)ds 
		\notag \\
		&&-\int_{0}^{1}\left( \alpha (2l-s)+\beta (2l-s)\right) p_{0}^{-}(s)ds. 
		\notag
	\end{eqnarray}%
	Taking the derivative of (\ref{EQ11}) and (\ref{EQ12}) yields%
	\begin{equation}
		2\phi (2k+2-x)\left( p_{0}^{-}(x)\right) ^{\prime }=\left( 
		\begin{array}{c}
			-2\phi ^{\prime }(k+2-x)-\alpha (2k+2-x) \\ 
			+\beta (2k-x)-\beta (2k+2-x)+\alpha (2k-x)%
		\end{array}%
		\right) p_{0}^{-}(x),  \label{EQ3}
	\end{equation}%
	\begin{equation}
		2\phi (2l+x)\left( q_{0}^{-}(x)\right) ^{\prime }=\left( 
		\begin{array}{c}
			-\phi ^{\prime }(2l+x)+\alpha (2l+x)+\beta (2l+x) \\ 
			-\alpha (2l-2+x)-\beta (2l-2+x)%
		\end{array}%
		\right) q_{0}^{-}(x).  \label{EQ4}
	\end{equation}%
	Now, we devide by $2\phi (2k+2-x)$ and $2\phi (2l+x)$ the equations (\ref%
	{EQ3}) and (\ref{EQ4}) and integrating the later two systems to get%
	\begin{equation*}
		p_{0}^{-}(x)=\exp \left( \int_{0}^{x}S_{l}^{1}(s)ds\right) p_{0}^{-}(0),%
		\text{ \ }q_{0}^{-}(x)=\exp \left( \int_{0}^{x}S_{k}^{2}(s)ds\right)
		q_{0}^{-}(0),\ 
	\end{equation*}%
	where%
	\begin{equation*}
		S_{k}^{1}(s)=\left( 2\phi (2k+2-s)\right) ^{-1}\left( 
		\begin{array}{c}
			-2\phi ^{\prime }(k+2-s)-\alpha (2k+2-s) \\ 
			+\beta (2k-s)-\beta (2k+2-s)-\alpha (2k-s)%
		\end{array}%
		\right)
	\end{equation*}%
	and%
	\begin{equation*}
		S_{l}^{2}(s)=\left( 2\phi (2l+s)\right) ^{-1}\left( 
		\begin{array}{c}
			-2\phi ^{\prime }(2l+s)+\alpha (2l+s)+\beta (2l+s) \\ 
			-\alpha (2l-2+s)-\beta (2l-2+s)%
		\end{array}%
		\right) .
	\end{equation*}%
	Since $p_{0}^{-}$ and $q_{0}^{-}$ lie in $D(\mathcal{A}^{\ast }),$ they are
	linked by the boundary conditions $(p_{0}^{-}+q_{0}^{-})_{|x=0,1}$, we get
	the system%
	\begin{equation*}
		\left\{ 
		\begin{array}{l}
			p_{0}^{-}(0)+q_{0}^{-}(0)=0, \\ 
			p_{0}^{-}(x)=\exp \left( \int_{0}^{1}S_{k}^{1}(s)ds\right) p_{0}^{-}(0)+\exp
			\left( \int_{0}^{1}S_{l}^{2}(s)ds\right) q_{0}^{-}(0)=0,%
		\end{array}%
		\right.
	\end{equation*}%
	which has a unique solution if, and only if 
	\begin{equation*}
		\int_{0}^{1}S_{k}^{1}(s)ds\neq \int_{0}^{1}S_{l}^{2}(s)ds.
	\end{equation*}
\end{proof}

\subsubsection{The case $\protect\eta _{1}(t,x)=\protect\alpha (t+x),$ $%
	\protect\eta _{2}(t,x)=\protect\beta (t-x)$}

This time set $\eta _{1}(t,x)=\alpha (t+x)$ and $\eta _{2}(t,x)=\beta (t-x)$
for some smooth functions $\alpha $ and $\beta $ in $Q_{T}.$ The components $%
K_{n}^{i,j}\left( \cdot ,\cdot \right) ,$ $1\leq i,j\leq 2,$ $n\geq 1,$
given in (\ref{K21}), (\ref{K22}), (\ref{K11}) and (\ref{K12}) and the
function $\phi $ defined in (\ref{phi}) take the form%
\begin{eqnarray*}
	2K_{k}^{11}(x) &=&-2K_{k}^{12}(x)=\alpha (2k+2-x)+\beta (2k-x),\text{ }x\in
	\lbrack 0,1], \\
	2K_{l}^{22}(x) &=&-2K_{l}^{21}(x)=\alpha (2l+x)+\beta (2l-2+x),\text{ }x\in
	\lbrack 0,1],
\end{eqnarray*}%
\begin{eqnarray*}
	\phi \left( t\right) &=&\int_{t-2}^{t-1}a(\tau ,\tau -\left( t-2\right)
	)d\tau +\int_{t-1}^{t}\beta (\tau ,t-\tau )d\tau \\
	&=&\int_{t-2}^{t-1}\alpha (2\tau -\left( t-2\right) )d\tau
	+\int_{t-1}^{t}\beta (2\tau -t)d\tau \\
	&=&\frac{1}{2}\int_{t-2}^{t}\left( \alpha (s)+\beta (s)\right) ds,\text{ }%
	t\geq 2.
\end{eqnarray*}%
At time $T=2n,$ the integral equation (\ref{Fredholm 11}) turns to 
\begin{equation*}
	\left( \mathcal{K}_{k,l}\Phi \right) \left( x\right) \mathcal{=}%
	A_{k,l}^{-1}(x)\mathbb{K}_{n,k,l}(x)\int_{0}^{1}\Phi \left( s\right)
	ds,~\Phi \in L^{2}\left( 0,1\right) ^{2},\text{ }k,l\geq 1,\text{ }x\in %
	\left[ 0,1\right] ,
\end{equation*}%
where%
\begin{equation*}
	A_{k,l}^{-1}(x)\mathbb{K}_{n,k,l}(x)=\left( 
	\begin{array}{cc}
		\frac{K_{k}^{11}(x)}{\phi (2k+2-x)} & -\frac{K_{k}^{11}(x)}{\phi (2k+2-x)}
		\\ 
		\begin{array}{c}
			\\ 
			-\frac{K_{l}^{22}(x)}{\phi (2l+x)}%
		\end{array}
		& 
		\begin{array}{c}
			\\ 
			\frac{K_{l}^{22}(x)}{\phi (2l+x)}%
		\end{array}%
	\end{array}%
	\right) ,\text{ }1\leq k,l\leq n-1,\text{ }x\in \lbrack 0,1].
\end{equation*}

We have the following unique continuation result:

\begin{proposition}
	\label{Propositopn Example 2}Let $n\geq 2.$ Assume that $A_{k,l}^{-1}(\cdot
	) $ exists on $[0,1]$ for some $1\leq k,l\leq n-1$. Then the \emph{unique
		continuation} property (\ref{system unique}) holds true in time $T=2n$ if,
	and only if 
	\begin{equation*}
		1\notin \sigma \left( \int_{0}^{1}A_{k,l}^{-1}(x)\mathbb{K}%
		_{n,k,l}(x)dx\right) .
	\end{equation*}
\end{proposition}

The proof of the above proposition is an immediate consequence of combining
Corollary \ref{Corollary approximat_Control} and the following lemma:

\begin{lemma}
	\label{Lemma operator J}Let $J:L^{2}(0,1)^{n}\rightarrow L^{2}(0,1)^{n}$ be
	the compact operator%
	\begin{equation*}
		\left( Jf\right) (x)=M(x)\int_{0}^{1}f(s)ds,\text{ }x\in (0,1),
	\end{equation*}
	
	where $M\in C\left( [0,1],M_{n\times n}\left( 
	\mathbb{R}
	\right) \right) .$ Then $1\in \sigma (J)$ if, and only if $1\in \sigma
	\left( \int_{0}^{1}M\right) .$
\end{lemma}

\begin{proof}
	See Appendix \ref{Appendix B}.
\end{proof}

\begin{remark}
	In the above two examples, we have used both $a$ and $b$ to simplify the
	computations. Dealing with the integral equation (\ref{Fredholm}) with only
	one of them seems to be not accessible unless they are in a very particular
	class of simple functions (for instance $a(t,x)=\kappa _{2}t+\kappa
	_{1}x+\kappa _{0},$ $\kappa _{i}\in 
	\mathbb{R}
	,$ $i=1,2,3$).
\end{remark}

\section{Appendix\label{Appendix}}

\subsection{Proof of Proposition \protect\ref{MO_sxn} \label{Appendix A}}

We start by proving that $1\Rightarrow 2.$

Suppose there exists $x_{0}\in \left[ 0,1\right] $ such for any $s\times s$
matrix $M_{\mathrm{ext}}$, extracted from $M,$ we have $\det M_{\mathrm{ext}%
}\left( x_{0}\right) =0.$ Denote by $M_{1},M_{2},...,M_{n}$ the $s-$
dimensional row vector functions of the matrix $M$ so that:%
\begin{equation*}
	M=\left[ 
	\begin{array}{c}
		M_{1} \\ 
		M_{2} \\ 
		\vdots \\ 
		M_{n}%
	\end{array}%
	\right] .
\end{equation*}%
Then from our assumption, the vector space spanned by the $M_{i}\left(
x_{0}\right) $ is at most of dimension $s-1.$ Thus there exists $V\in 
\mathbb{R}
^{s}$ such that 
\begin{equation*}
	\left\vert V\right\vert _{\mathbb{R}^{s}}=1~\mathrm{and}~M_{i}\left(
	x_{0}\right) \cdot V=0,~1\leq i\leq n.
\end{equation*}%
Let $u\in C_{0}^{\infty }\left( \mathbb{R}\right) $ with $\int_{\mathbb{R}%
}u^{2}\left( x\right) dx=1$ and (for example) \textrm{supp}$\left( u\right) =%
\left[ -1,1\right] .$ We are going to prove that the sequence:%
\begin{equation*}
	h_{j}\left( x\right) =\sqrt{j}u\left( j\left( x-x_{0}\right) \right)
	V,~j\geq j_{0},~x\in \mathbb{R},
\end{equation*}%
is a singular sequence for the multiplication operator $\mathbb{M}_{\mathrm{%
		ext}}$ on $L^{2}\left( 0,1\right) ^{s}$ whose matrix is some $s\times s$
matrix $M_{\mathrm{ext}}$, extracted from $M$. If $x_{0}\in \left(
0,1\right) $ (we leave to the reader to check that the proof works with $%
x_{0}=0$ by choosing $h_{j}\left( x\right) =\sqrt{j}u\left( jx-1\right) V$
and with $x_{0}=1$ by choosing $h_{j}\left( x\right) =\sqrt{j}u\left(
j\left( x-1\right) +1\right) V$), we see from the definition of $u$ that for 
$j_{0}$ sufficiently large%
\begin{equation*}
	\text{\textrm{supp}}\left( h_{j}\right) =\left[ x_{0}-\frac{1}{j},x_{0}+%
	\frac{1}{j}\right] \subset \left[ 0,1\right] ,~~j\geq j_{0}.
\end{equation*}%
Moreover, for all $~j\geq j_{0}.$%
\begin{eqnarray*}
	\int_{0}^{1}h_{j}^{2}\left( x\right) dx &=&j\int_{0}^{1}u^{2}\left( j\left(
	x-x_{0}\right) \right) \left\vert V\right\vert _{\mathbb{R}^{s}}^{2}dx \\
	&=&j\int_{x_{0}-\frac{1}{j}}^{x_{0}+\frac{1}{j}}u^{2}\left( j\left(
	x-x_{0}\right) \right) dx \\
	&=&\int_{-1}^{1}u^{2}\left( \xi \right) d\xi =1.
\end{eqnarray*}%
Since \textrm{supp}$\left( h_{j}\right) \rightarrow \left\{ x_{0}\right\} $
as $j\rightarrow \infty ,$ $\left( h_{j}\right) $ has no convergent
subsequence in $L^{2}\left( 0,1\right) ^{n}.$ On the other hand:%
\begin{eqnarray*}
	\left\Vert \mathbb{M}_{\mathrm{ext}}h_{j}\right\Vert _{L^{2}\left(
		0,1\right) ^{s}}^{2} &=&\int_{0}^{1}\left\vert M_{\mathrm{ext}}\left(
	x\right) h_{j}\left( x\right) \right\vert ^{2}dx \\
	&=&j\int_{x_{0}-\frac{1}{j}}^{x_{0}+\frac{1}{j}}u^{2}\left( j\left(
	x-x_{0}\right) \right) \left\vert M_{\mathrm{ext}}\left( x\right)
	V\right\vert ^{2}dx \\
	&=&\int_{-1}^{1}u^{2}\left( \xi \right) \left\vert M_{\mathrm{ext}}\left(
	x_{0}+\frac{1}{j}\xi \right) V\right\vert ^{2}dx.
\end{eqnarray*}%
Thus, from Lebesgue's dominated convergence theorem, we get:%
\begin{equation*}
	\lim_{j\rightarrow \infty }\left\Vert \mathbb{M}_{\mathrm{ext}%
	}h_{j}\right\Vert _{L^{2}\left( 0,1\right) ^{s}}^{2}=0.
\end{equation*}%
Since, in particular, the choice of $V$ \ and the continuity of $M_{\mathrm{%
		ext}}$ give$:$ 
\begin{equation*}
	\lim_{j\rightarrow \infty }M_{\mathrm{ext}}\left( x_{0}+\frac{1}{j}\xi
	\right) V=M_{\mathrm{ext}}\left( x_{0}\right) V=0.
\end{equation*}%
If $n=ds+q$ with $d\geq 1$ and $0\leq q\leq s-1,$ we form the $s\times s$
extracted matrices:%
\begin{equation*}
	M_{\mathrm{ext}}^{1}=\left[ 
	\begin{array}{c}
		M_{1} \\ 
		\vdots \\ 
		M_{s}%
	\end{array}%
	\right] ,\text{ }...,\text{ }M_{\mathrm{ext}}^{d}\left[ 
	\begin{array}{c}
		M_{\left( d-1\right) s+1} \\ 
		\vdots \\ 
		M_{ds}%
	\end{array}%
	\right] ,~M_{\mathrm{ext}}^{d+1}=\left[ 
	\begin{array}{c}
		M_{ds+1} \\ 
		\vdots \\ 
		M_{n} \\ 
		M_{1} \\ 
		\vdots \\ 
		M_{s-q}%
	\end{array}%
	\right] .
\end{equation*}%
We then have:%
\begin{eqnarray*}
	\left\Vert \mathbb{M}h_{j}\right\Vert _{L^{2}\left( 0,1\right) ^{n}}^{2}
	&=&\sum_{k=1}^{n}\left\Vert M_{k}\cdot h_{j}\right\Vert _{L^{2}\left(
		0,1\right) }^{2} \\
	&\leq &\sum_{k=1}^{d+1}\left\Vert \mathbb{M}_{\mathrm{ext}%
	}^{k}h_{j}\right\Vert _{L^{2}\left( 0,1\right) ^{s}}^{2}.
\end{eqnarray*}%
It readily follows that: 
\begin{equation*}
	\lim_{j\rightarrow \infty }\left\Vert \mathbb{M}h_{j}\right\Vert
	_{L^{2}\left( 0,1\right) ^{n}}^{2}=0.
\end{equation*}%
This proves that $1\Rightarrow 2.$

To prove that $2\Rightarrow 1,$ we assume that for all $x\in \left[ 0,1%
\right] ,$ there exists a $s\times s$ matrix $M_{\mathrm{ext}}$, extracted
from $M,~$ such that 
\begin{equation*}
	\det M_{\mathrm{ext}}\left( x\right) \neq 0.
\end{equation*}%
Each one of the functions $\left\vert \det M_{\mathrm{ext}}\left( x\right)
\right\vert $ is uniformly continuous on $\left[ 0,1\right] :$%
\begin{equation*}
	\forall \varepsilon >0,\exists \eta _{M_{\mathrm{ext}}}>0,~\left\vert
	x-y\right\vert <\eta _{M_{\mathrm{ext}}}\Rightarrow \left\vert \left\vert
	\det M_{\mathrm{ext}}\left( x\right) \right\vert -\left\vert \det M_{\mathrm{%
			ext}}\left( y\right) \right\vert \right\vert <\varepsilon ,~\forall \left(
	x,y\right) \in \left[ 0,1\right] ^{2}.
\end{equation*}%
In the sequel, we set $\eta =\min \left\{ \eta _{M_{\mathrm{ext}}},~M_{%
	\mathrm{ext}}~\mathrm{extracted~}s\times s\mathrm{~matrix}\right\} .$ Let $%
0\leq j\leq m-1$ and $\xi _{j}\in \left[ \frac{j}{m},\frac{j+1}{m}\right] .$
There exists a $s\times s$ matrix $M_{\mathrm{ext}}^{j}$%
\begin{equation*}
	\left\vert \det M_{\mathrm{ext}}^{j}\left( \xi _{j}\right) \right\vert
	:=\delta _{j}>0.
\end{equation*}%
Choosing $m$ such that $\frac{1}{m}<\eta ,$ we get with $\varepsilon <%
\underset{0\leq j\leq m-1}{\min }\delta _{j}=\delta $ 
\begin{equation*}
	\left\vert \det M_{\mathrm{ext}}^{j}\left( x\right) \right\vert >\delta
	-\varepsilon ,~\forall x\in \left[ \frac{j}{m},\frac{j+1}{m}\right] .
\end{equation*}%
From Proposition \ref{MO_sxs}, this ensures that for each $0\leq j\leq m-1,$ 
$\mathbb{M}_{\mathrm{ext}}^{j}$ is invertible on $L^{2}\left( \frac{j}{m},%
\frac{j+1}{m}\right) ^{s}.$ Now, we can write for any $h\in L^{2}\left(
0,1\right) ^{s}:$%
\begin{eqnarray*}
	\left\Vert \mathbb{M}h\right\Vert _{L^{2}\left( 0,1\right) ^{n}}^{2}
	&=&\sum_{k=1}^{n}\left\Vert M_{k}\cdot h\right\Vert _{L^{2}\left( 0,1\right)
	}^{2} \\
	&=&\sum_{j=0}^{m-1}\sum\limits_{k=1}^{n}\int_{\frac{j}{m}}^{\frac{j+1}{m}%
	}\left\vert M_{k}\left( x\right) \cdot h(x)\right\vert ^{2}dx \\
	&\geq &\sum_{j=0}^{m-1}\int_{\frac{j}{m}}^{\frac{j+1}{m}}\left\vert M_{%
		\mathrm{ext}}^{j}\left( x\right) h(x)\right\vert ^{2}dx \\
	&\geq &C\sum_{j=0}^{m-1}\int_{\frac{j}{m}}^{\frac{j+1}{m}}\left\vert
	h(x)\right\vert ^{2}dx \\
	&\geq &C\int_{0}^{1}\left\vert h(x)\right\vert ^{2}dx.
\end{eqnarray*}%
This ends the proof.

\subsection{Proof of Lemma \protect\ref{Lemma operator J} \label{Appendix B}}

Proof of $\Rightarrow $: If $1\in \sigma (J)$, there will exist a non-zero $%
f\in L^{2}(0,1)^{n}$ such that%
\begin{equation}
	f(x)=M(x)\int_{0}^{1}f(s)ds,\text{ }x\in (0,1).  \label{J1}
\end{equation}%
Integrating (\ref{J1}) over $(0,1)$ yields%
\begin{equation*}
	\int_{0}^{1}f(s)ds=\int_{0}^{1}M(s)ds\int_{0}^{1}f(s)ds.
\end{equation*}%
This shows that $\int_{0}^{1}f$ is an eigenvector of the matrix $%
\int_{0}^{1}M$ associated with the eigenvalue $1.$

Proof of $\Leftarrow $: If $1\in \sigma \left( \int_{0}^{1}M\right) ,$ then
there exists an eigenvector of $\int_{0}^{1}M$ denoted by $V\in 
\mathbb{R}
^{n}$ such that%
\begin{equation*}
	V=\left( \int_{0}^{1}M(x)dx\right) V.
\end{equation*}%
applying the matrix $M(\cdot )$ yields%
\begin{equation*}
	M(x)V=M(x)\left( \int_{0}^{1}M(x)Vdx\right) ,\text{ }x\in \lbrack 0,1].
\end{equation*}%
This shows that the vector $M(\cdot )V$ is an eigenvector of $J$ associated
with the eigenvalue $1$.

\bigskip


\begin{thebibliography}{99}
	\bibitem{Alabau 1} F. Alabau-Boussouira. A two-level energy method for
	indirect boundary observability and controllability of weakly coupled
	hyperbolic systems. \textit{SIAM J. Control Optim, 42 (2003), 871-906}.
	
	\bibitem{Alabau 2} F. Alabau-Boussouira and M. L\'{e}autaud. Indirect
	controllability of locally coupled systems under geometric conditions. 
	\textit{C. R. Acad. Sci. Paris, Ser. I , 349 (2011), 395-400}.
	
	\bibitem{Alabau 3} F. Alabau-Boussouira and M. L\'{e}autaud. Indirect
	controllability of locally coupled wave-type systems and applications. 
	\textit{J. Math. Pures Appl., 99 (2013), 544-576}.
	
	\bibitem{Alabau 4} F. Alabau-Boussouira. On the influence of the coupling on
	the dynamics of single-observed cascade systems of PDE'S. \textit{J. Math.
		Control and Related Fields 5 (2015) 1-30}.
	
	\bibitem{Alabau 5} F. Alabau-Boussouira, J.-M. Coron, and G. Olive. Internal
	controllability of first-order quasi-linear hyperbolic systems with a
	reduced number of controls. SIAM J. Control Optim., 55(1), (2017), 300-323.
	
	\bibitem{Ammar khodja-Bader} F. Ammar Khodja, A. Bader. Stabilizability of
	systems of one-dimensional wave equations by one internal or boundary
	control force. \textit{SIAM J. Control Optim. 39 (2001) 1833-1851}.
	
	\bibitem{Ammar Khodja 1} F. Ammar-Khodja, A. Benabdallah, M. Gonz\`{a}%
	lez-Burgos, L. de Teresa, New phenomena for the null controllability of
	parabolic systems: Minimal time and geometrical dependence, J. Math. Anal.
	Appl. 444 (2016), no. 2, 1071-1113.
	
	\bibitem{Avdonin 1} S. A. Avdonin, A. Choque Rivero, L. de Teresa, Exact
	boundary controllability of coupled hyperbolic equations. \textit{Int. J.
		Appl. Math. Comput. Sci., 23 (2013), 701-709}.
	
	\bibitem{Avdonin 2} S. A. Avdonin, J. Park 1, L. de Teresa, The Kalman
	condition for the boundary controllability of coupled 1-d wave equations,
	Evolution equations and control theory, 9(1), (2020), 255-273.
	
	\bibitem{Bardos-Lebeau-Rauch} C. Bardos, G. Lebeau and J. Rauch. Sharp
	sufficient conditions for the observation, control, and stabilization of
	waves from the boundary. \textit{SIAM J. Control Optim., 30 (1992), 1024-1065%
	}.
	
	\bibitem{Bennour} A. Bennour, F. Ammar Khodja, and D. Teniou. Exact and 	approximate controllability of coupled one-dimensional hyperbolic equations.  \textit{Ev. Eq. and Cont. Theo., 6 (2017), 487-516}.
	
	\bibitem{Brezis} H. Br\'{e}zis. \textit{Functional Analysis}. Springer
	(2011).
	
	\bibitem{Cui} Y. Cui, C. Laurent, Z. Wang, On the observability inequality
	of coupled wave equations: the case without boundary. ESAIM Control Optim.
	Calc. Var, 26 (14), 2020.
	
	\bibitem{Coron} Jean-Michel Coron and Hoai-Minh Nguyen, On the optimal
	controllability time for linear hyperbolic systems with time-dependent
	coefficients, arXiv:2103.02653.
	
	\bibitem{Dehman} B. Dehman, J. Le Rousseau and M. L eautaud. Controllability
	of two coupled wave equations on a compact manifold. \textit{Arch. Rat.
		Mech. Anal. 211 (2014) 113-187}.
	
	\bibitem{Duprez} M. Duprez and G. Olive, Compact perturbations of controlled
	systems, Math. Control Relat. Fields 8 (2018) 397-410.
	
	\bibitem{Duprez 2} M. Duprez, Controllability of a $2\times 2$ parabolic
	system by one force with space-dependent coupling term of order one, ESAIM
	Control Optim. Calc. Var., 23 (2017) 1473-1498.
	
	\bibitem{Furikov} A. V. Fursikov and O. Yu. Imanuvilov, \textit{%
		Controllability of Evolution Equations}, Lecture Notes Series 34, Research
	Institute of Mathematics, Seoul National University, Seoul, Korea, 1994.
	
	\bibitem{Hardt} V. Hardt and E. Wagenfuhrer. Spectral Properties of a
	Multiplication Operator. \textit{Math. Nachr. 178 (1996) 135-156.}
	
	\bibitem{Liard} T. Liard and P. Lissy, A Kalman rank condition for the
	indirect controllability of coupled systems of linear operator groups Math.
	Control Signals Syst. 29, 9, (2017).
	
	\bibitem{Neves} A. F. Neves, H. de Souza Ribeiro, and O. Lopes, On the
	spectrum of evolution operators generated by hyperbolic systems, J. Funct.
	Anal. 67, no. 3, (1986) 320-344.
	
	\bibitem{Olive} L. Hu, G. Olive, Minimal time for the exact controllability
	of one-dimensional first-order linear hyperbolic systems by one-sided
	boundary controls, Journal de Math\'{e}matiques Pures et Appliqu\'{e}es,
	148, (2021) 24-74.
	
	\bibitem{Pazy} A. Pazy. \textit{Semigroups of linear operators and
		applications to partial differential equations}, volume 44 of Applied
	Mathematical Sciences. Springer-Verlag, New York, 1983.
	
	\bibitem{Peetre} J. Peetre, Another approach to elliptic boundary problems,
	Comm. Pure Appl. Math. 14 (1961), 711-731.
	
	\bibitem{Russell} David L. Russell, Controllability and stabilizability
	theory for linear partial differential equations: recent progress and open
	questions, SIAM Rev. 20 (1978), no. 4, 639--739. MR 508380.
	
	\bibitem{Zabczyk} J. Zabczyk. \textit{Mathematical Control Theory}. An
	Introduction. Birkhauser (2008).
	
	\bibitem{Zhang} X. Zhang. Explicit observability inequalities for the wave
	equation with lower order terms by means of Carleman inequalities. \textit{%
		SIAM J. Control Optim. 39 (2000) 812-834}.
\end{thebibliography}
\end{document}